\documentclass{cmslatex}
\usepackage[paperwidth=7in, paperheight=10in, margin=.875in]{geometry}
\usepackage{graphicx}
\usepackage{amsfonts}
\usepackage{amsmath}
\usepackage{amssymb}
\usepackage{fancyhdr}
\usepackage{footmisc}
\usepackage{indentfirst}
\usepackage[normalem]{ulem}
\usepackage{placeins}
\usepackage{listings}
\usepackage{subfigure}
\usepackage{multirow}
\usepackage{comment}
\usepackage{todonotes}
\usepackage{xcolor}
\usepackage{soul}
\usepackage{url}
\usepackage{bm}
\usepackage{bbm}
\usepackage{algorithm}
\usepackage{algpseudocode}
\usepackage[version=3]{mhchem}

\DeclareMathOperator*{\argmin}{argmin}

\renewcommand{\figurename}{Figure}

\newtheorem{thm}{Theorem}[section]
\newtheorem{remark}[thm]{Remark}
\newtheorem{assumption}{Assumption}

\renewcommand{\theequation}{\arabic{section}.\arabic{equation}}

\allowdisplaybreaks
\definecolor{revisecolor}{rgb}{0.0, 0.0, 0.0}

\begin{document}
\title{Optimal control of Markov jump processes : asymptotic analysis,
algorithms and applications to the modelling of chemical reaction systems}
\author{
 Wei Zhang\thanks{Institute of Mathematics, Freie Universit\"{a}t Berlin,
 Arnimallee 6, 14195 Berlin, Germany,  wei.zhang@fu-berlin.de}
 \and Carsten Hartmann\thanks{Institute of Mathematics, Brandenburgische
     Technische Universit\"{a}t Cottbus-Senftenberg, Platz der Deutschen Einheit 1, 03046 Cottbus, Germany,
 hartmanc@b-tu.de}
\and Max von Kleist\thanks{Institute of Mathematics, Freie Universit\"{a}t Berlin,
 Arnimallee 6, 14195 Berlin, Germany,  vkleist@zedat.fu-berlin.de}}
\maketitle

\begin{abstract}
Markov jump processes are widely used to model natural and engineered processes. In the context of biological or
chemical applications one typically refers to the chemical master equation
(CME), which models the evolution of the probability mass of any copy-number
combination of the interacting particles. When many interacting particles
(``species'') are considered, the complexity of the CME quickly increases,
making direct numerical simulations impossible. This is even more problematic
when one aims at \textit{controlling} the Markov jump processes defined by the CME.

In this work, we study both \textit{open loop} and \textit{feedback} optimal control problems of the Markov jump
processes in the case that the controls can only be
switched at fixed control stages. Based on Kurtz's limit theorems, we prove the convergence of the respective control value
functions of the underlying Markov decision problem as the copy numbers of the species go to infinity. 
In the case of the optimal
control problem on a finite time-horizon, we propose a hybrid control policy algorithm to overcome the difficulties due to the
curse of dimensionality when the copy number of the involved species is large.
Two numerical examples demonstrate the suitability of both the analysis and the proposed algorithms.
\end{abstract}
\section{Introduction}
\label{sec-intro}
In the past decades, discrete-state Markov jump processes have been a major research
topic in probability theory receiving much attention in applications like economics, physics, biology and chemistry; see e.g.~ \cite{Wilkinson2006,Kampen2007,Gardiner2010,Gillespie1991,Allen2011,Weibull1997}.
For example, in the modelling of chemical reactions, a single state is defined
as one possible copy-number combination of the distinct interacting chemical
species. After a random waiting time, a reaction occurs and changes this
copy-number combination. Since the time and order in which chemical reactions
occur is random (referred as \textit{intrinsic} noise), the evolution of the
state of the system is random as well. The chemical master equation (CME) models the probability of all possible outcomes over time, giving rise to an extremely large state space (consisting of all copy-number combinations). Consequently, solving the chemical master equation or approximating its solution computationally is a non-trivial, yet unsolved task that has been the objective of intense research over the past decades (see e.g. \cite{Pahle2009} for a summary).

In many real world applications, one does not only aim at propagating or
simulating a process forward in time, but also aims at controlling and optimizing it. In this case, the model equations of a controlled system contain extra
terms or parameters that can be manipulated by the decision maker according to some control policy. The latter is
chosen so that a given cost functional reaches an optimal (e.g. minimum) value.
There are two general approaches to an optimal control task, depending on whether the admissible control policies are allowed to depend on the system states (\emph{feedback} or \emph{closed loop} control problem) or not (\emph{open loop} control problem).
In the case of an open loop control, the control follows a fixed, deterministic policy regardless of the fact that the underlying dynamics are stochastic.
On the other hand, feedback controls are random in the sense that each
realization of the process gives rise to a different control that is adapted
according to the random states of the system. 
\textcolor{revisecolor}{
In principle, one can also consider the case where the control policies depend not only on the
current states of the system but also on the past.
However, for Markov jump processes, it is known that under certain assumptions the optimal cost value can be
achieved by a feedback control policy, which only depends on system's current
states (see Section 4.4 of~\cite{Puterman1994} for a precise statement).}

For small or moderately sized systems, the underlying optimal control
problem can be solved numerically using the dynamic programming principle~\cite{Puterman1994,Bellman2003,Winkelmann2014,Duwal2015}. However, for large systems,
solving the optimal control policy by dynamic programming or related methods
becomes difficult without suitable approximations or remodelling
steps~\cite{adp_powell,Sutton1998,adp_bertsekas}. Within the area of systems biology
or chemical engineering, one such remodelling step that has been
extensively exploited by control engineers is to replace the stochastic
dynamics by a deterministic system of ordinary differential equations (ODE) that ignores the intrinsic noise (e.g., see
\cite{kirk2004optimal, Lenhart2007}). These continuous deterministic reaction
rate equations model the concentrations of the interacting chemical species by
one ODE per species. The approximation of the
stochastic system using the ODE system is mainly based on Kurtz's seminal work \cite{limit_thm_ode1971, ode_kurtz_limit,
Kurtz1978, kurtz_limit_thm_diffusion_approx,ball06,kang2013,kang2014} (also
see the recent work on multiscale analysis~\cite{cappelletti2016,pfaffelhuber2015}), 
which shows that the particle numbers per unit volume of the original Markov jump processes without control can be approximated by the classical reaction rate equations in the large copy-number regime (parameterized by either the total number of particles $N$ or the reaction volume $V$).

In this article, we investigate the relationship between the optimal control problem for the original Markov jump process 
and the limiting ODE system. Stochastic control problems for Markov jump
processes are also termed ``Markov decision processes'' (MDP) \cite{Bellman2003,Puterman1994,howardr60}.
We confine our analysis to the situation that the control can only be
changed at given discrete points in time (called \textit{control stages}).
The key contribution of this paper is twofold: Firstly, applying Kurtz's
limit theorem, we prove convergence of the cost value of the controlled Markov jump
process to the cost value of the controlled limiting ODE
system as $N\rightarrow \infty$, both in the open loop and the feedback case; the convergence results then imply that the optimal open loop control policy for the ODE
system can be applied to control the Markov jump process when $N$
is large where the optimal cost is achieved asymptotically. Secondly, based on
these theoretical results, we propose a hybrid control policy for the optimal control problem
of the Markov jump process on a finite time-horizon;
the hybrid control policy not only exploits the information of the optimal
control policy for the limiting ODE, but also takes into account the
stochasticity of the jump process and thus improves
the optimal control policy from the ODE approximation in the pre-limit regime when $N$ is moderately large; in terms of
computational complexity, the hybrid algorithm avoids the curse of
dimensionality by using an on-the-fly state space truncation.
Broadly speaking, the hybrid control policy is related to approximative dynamical programming (ADP)
and reinforcement learning that have been extensively studied in the last years~\cite{adp_powell,adp_bertsekas,Sutton1998,Shani_pomdp_survey}.

\subsection*{Related work}

Although this work is mainly motivated by epidemic, biological and chemical reaction
models, it is important to note that the asymptotic analysis of the related optimal control problems appears relevant in scheduling  and queueing theory~\cite{gross_queue,limits_whitt}. In the context of scheduling and queueing problems, the relevant asymptotic regime is the \emph{heavy traffic limit}, under which the stochastic model can be approximated by either a diffusion process or an ODE system; the limit models are named \emph{Brownian network} or \emph{fluid approximation}, depending on whether the limiting differential equation is stochastic or deterministic.
Readers interested in the Brownian network approach may
consult~\cite{harrison1997,martin_soner_1996,Mandelbaum1998,kushner_martins1996,kushner_ramachandran1989,limits_whitt}
and references therein. For the fluid approximation of stochastic queueing
networks, we refer to~\cite{dai_meyn1995,dai1995,Mandelbaum1998};
cf.~\cite{mandelbaum_state_dependent_1998} for a discussion of both the fluid
and the Brownian network approximation. Optimal control of
queueing networks and their fluid approximations has been studied in~\cite{baeuer2000,baeuerle2002,dupuis_game2003,Meyn_book,maglaras2000,meyn_parki,Piunovskiy2011,ramanan2011}; see also \cite{Kushner2001,PangDay2007} for an approach using weak convergence techniques.

Despite the vast literature on queueing systems, we emphasize that
the models and problems therein are quite different from the ones studied herein. For example, for queueing networks, one is often interested in minimizing the total queue length
(or its linear combination)
by controlling how each server should allocate the service time to each queue,  which explains that many of the rigorous results are confined to linear cost functions or birth-death-processes (e.g. \cite{baeuerle2002,Piunovskiy2011}).
In the current work, besides the differences of the models, the running cost
is allowed to be an arbitrary bounded and (local) Lipschitz function in the system states (see Assumption~\ref{assump-4} in Section~\ref{sec-setup}) and the jump rates
of the process may depend on the controls.
A limitation of our work is that the controls are switched only at discrete
time points (control stages). However, this
assumption allows us to obtain stronger convergence results (with explicit convergence order in some cases)
and covers applications in epidemic or chemical reaction
networks~\cite{Schaefer2004,Bennett2013,Hauskrecht_Fraser2000,Winkelmann2014,Duwal2015}.
Specifically, we will prove the asymptotic optimality of finite and infinite time-horizon open loop policies arising from the deterministic limit equations. 
Our work complements available results on the asymptotic optimality of the
associated closed loop policies or \emph{tracking policies}
(e.g.~\cite{baeuer2000,maglaras2000,meyn_parki}) and gives rise to numerical
algorithms that do not require to solve the dynamic programming equations on
the whole state space (see Section \ref{sec-hybrid}).

\subsection*{Outline}
The remainder of this paper is organized as follows: In Section~\ref{sec-setup}, we introduce
the mathematical problem along with the notations used throughout this paper
and two paradigmatic examples.
Section~\ref{sec-analysis-ocp} is devoted to the extension of Kurtz's limit
theorem for Markov jump processes and to apply it to study optimal control problems. Based on this analysis, a hybrid control algorithm is proposed and discussed in
Section~\ref{sec-hybrid}.
We present several numerical examples in Section~\ref{sec-examples}, and a technical lemma is recorded in Appendix~\ref{app-1}.

\section{Mathematical Setup}
\label{sec-setup}
In this section, we will first introduce our problem, the notations used, and finally sketch
two concrete situations in which the problem is relevant.

\subsection{Controlled Markov jump processes}
\label{sub-setup-mjp}

Let $\mathbb{X}$ be a discrete lattice in $\mathbb{R}^n$ and consider the Markov jump
process $x(t)$ on it.
Suppose that at time $t \ge 0$ and given $x(t) = x \in \mathbb{X}$,
the probability for making a transition from $x$ to $x + l$
within the infinitesimal time interval $[t, t + ds)$ is $f(x, l)\,ds$, $l \in \mathbb{X}$.
Denoting $\tau$ the waiting time
\begin{align}
  \tau = \inf_{s > t}\big\{s-t ~;~ x(s) \neq x(t)\big\}\,,
  \label{waiting-time}
\end{align}
it is known that $\tau$ follows an exponential distribution with the rate $\lambda(x) = \sum_{l
\in \mathbb{X}} f(x, l)$, i.e.~$\tau \sim \mbox{Exp}(\lambda(x))$.

\subsubsection*{Jump rates}
In this work, we suppose that the jump process $x(t)$ depends on both a parameter $N \gg
1$ and the control $\nu \in \mathcal{A}$, where $\mathcal{A}$ is the \textit{control set}.
In applications, $N$ may be related to system's volume or the magnitude of particle
numbers, while the control $\nu$ may affect the jump rates $f$.
To indicate these dependencies, we denote the jump process as $x^{\nu,N}$ and
also introduce the normalized process $z^{\nu,N}(t) = N^{-1} x^{\nu,N}(t)$. It is convenient to think of the normalized variable $z$
as a particle density, which is why we will sometimes refer to $z^{\nu,N}(t)$
as the \emph{normalized density process}.
Notice that $z^{\nu,N}$ is a Markov jump process on the \textit{scaled lattice} and, due
to its importance in our analysis, we use the notation $\mathbb{X}_N$ and
$f^{\nu,N}_{\textcolor{revisecolor}{d}} : \mathbb{X}_N \times \mathbb{X}_N
\rightarrow \mathbb{R}^+$ for its state space and jump rates, respectively, 
where $\mathbb{R}^+$ is the set consisting of non-negative real numbers.
$\mathbb{X}$ and $f^{\nu,N}_{\textcolor{revisecolor}{o}} : \mathbb{X} \times \mathbb{X} \rightarrow \mathbb{R}^+$
will be reserved for the \textit{original process} $x^{\nu,N}$.
\textcolor{revisecolor}{Notice that the jump rates of the original process may depend on $N$.
The subscripts ``d'' and ``o'' which appear in the rate functions simply indicate
that they refer to either the normalized density process or the original process.}
Specifically,
we have $\mathbb{X}_N = \{\frac{x}{N}\, |\, x \in \mathbb{X}\}$ and $f^{\nu,
N}_d(z,l) = f^{\nu,N}_o(Nz, Nl)$ for $z, l \in \mathbb{X}_N$. 

\subsubsection*{Controls}
We will discuss the control policies and the controlled Markov jump process in detail.
For the sake of simplicity, we will refer to the normalized process $z^{\nu,
N}$ only, stressing that all considerations are transferable to the process $x^{\nu,N}$.
Suppose that on the time interval $[0, T]$, $K+1$ time points $0 = t_0 < t_1 < \cdots < t_j < t_{j+1} < \cdots < t_K = T$ are given and
fixed. At each time $t_j$, $0 \le j < K$, called \textit{control stage}, we
are allowed to select some
control $\nu_j \in \mathcal{A}$ and apply it to the jump process in order to
influence its jump rates.
Once a control $\nu_j$ is selected at time $t_j$, it will persistently take
effect during  the time interval $[t_j, t_{j+1})$.
  When the selection of controls $\nu_j$ is allowed to depend on the
  system's current states at time $t_j$, the control policy is called feedback
  control policy and otherwise it is called open loop control policy.
  More generally, we introduce the sets of open loop and feedback control
  policies on time $[t_k, T]$ for $0 \le k < K$ :
  \begin{align}
    \begin{split}
    \mathcal{U}_{o,k} =& \big\{(\nu_k, \nu_{k+1}, \cdots, \nu_{K-1}) ~|~ \nu_j \in
    \mathcal{A}\,,~k\le j < K\big\}\,, \\
    \mathcal{U}_{f,k} =& \big\{(\nu_k, \nu_{k+1}, \cdots, \nu_{K-1}) ~|~ \nu_j :
    \mathbb{X}_N \rightarrow \mathcal{A}\,,~k\le j < K\big\}\,.
  \end{split}
  \label{control-policy-set}
  \end{align}
  \textcolor{revisecolor}{Notice that in the feedback case, while each policy $\nu_j$ is a function of
  the state, the same notation will be used to denote its value (i.e.~the control selected
at $t_j$) when no ambiguity exists.}
  For further simplification, let $\sigma$ denote either `o' or `f' and we
  will write $\mathcal{U}_{\sigma, k}$ to refer to either open loop or
  feedback control policy set.

Given a control policy $u \in \mathcal{U}_{\sigma,k}$,  we express the
  corresponding controlled process in the time interval $[t_k, T]$ as $z^{u,N}(t)$,
  i.e.~the control $\nu_j$ is applied during time $t \in [t_j, t_{j+1})$, $k \le j < K$.
    The notation $z^{u,N}(t\,;\,z)$ will be used to emphasize that the
 process starts from a fixed initial state $z \in \mathbb{X}_N$ at time $t_k$ (the starting time may be nonzero).
Specifically, for a fixed control policy
$$u = (\nu_0, \nu_1, \nu_2, \cdots, \nu_{K-1}) \in \mathcal{U}_{\sigma, 0},$$
$z^{u,N}(t), t \ge 0$ is a Markov jump
process with the property that the
  probability for system's state to jump from $z^{u,N}(t)=z$ to $z + l$ within
the infinitesimal time interval $[t, t+ds)$ at $t \in [t_{j}, t_{j+1})$, is
    $f^{\nu_j,N}_{\textcolor{revisecolor}{d}}(z, l)\,ds$ for $l \in \mathbb{X}_N$.
That is, application of controls changes the jump rates of the Markov jump process. 
With the notation 
     \begin{align}
       j(t) := i\,, \quad \mbox{if}\, \quad t \in [t_{i}, t_{i+1})\,, 
	 \label{j-t}
    \end{align}
      we can denote the control policy which is applied to the process $z^{u,N}(t)$ at
      time $t$ as $\nu_{j(t)}$.
\textcolor{revisecolor}{Finally, 
the notation $z^{\nu,N}(t)$ will also be used,
when we emphasize that the current control policy at time $t$ is $\nu \in
\mathcal{A}$, 
or when we consider the controlled process on a single stage $[t_j,
  t_{j+1})$, in which case only the control policy $\nu$ applied at time
$t_j$ is relevant.}

\subsubsection*{Cost functional}

For a control policy $u = (\nu_0, \nu_1, \cdots,
\nu_{K-1}) \in \mathcal{U}_{\sigma, 0}$
and the process $z^{u, N}$, we define the cost functional
\begin{align}
  J_N(z, u) = \mathbf{E}_{z}^u \Big[\sum_{j=0}^{K-1}
    \Big(
r\big(z^{u,N}(t_j), \nu_j\big) +
    \int_{t_j}^{t_{j+1}}
  \phi\big(z^{u,N}(s), \nu_j\big)\, ds\Big) +
\psi\big(z^{u,N}(T)\big) \Big]\,
  \label{cost-j-general}
\end{align}
where $\mathbf{E}_{z}^u$ denotes the expectation over all realizations of $z^{u,N}$ starting at $z^{u,N}(0) = z$ and evolving under the control policy $u$.
We emphasize that, in the feedback case $u\in \mathcal{U}_{f,0}$, 
we have adopted the convention discussed before and 
$\nu_j$ in (\ref{cost-j-general}) should be interpreted as $\nu_j=\nu_j(z^{u,N}(t_j))$. 
Functions $r, \phi : \mathbb{R}^n \times \mathcal{A} \rightarrow
\mathbb{R}$ and $\psi : \mathbb{R}^n \rightarrow \mathbb{R}$ correspond to
the costs at each control stage $t_j$, the running cost and the terminal cost,
respectively.

\subsection{Limiting process and underlying assumptions}
\label{subsec-assumptions}

Our analysis in the course of the paper is based on
Kurtz's limit theorems for jump processes \cite{limit_thm_ode1971, ode_kurtz_limit,
Kurtz1978, kurtz_limit_thm_diffusion_approx}, which state that, for $u \in \mathcal{U}_{o, 0}$, the normalized density process
$z^{u,N}$ converges to a deterministic limiting process  $\tilde{z}^u$
\textcolor{revisecolor}{under certain assumptions} and is governed by the ordinary differential equation (ODE)
\begin{align}
  \frac{d\tilde{z}^u(t)}{dt} = F^{\nu_{j(t)}}(\tilde{z}^u(t)) \,,
  \label{limit-ode}
\end{align}
or, in integral form,
\begin{align}
  \tilde{z}^u(t) = \tilde{z}^u(0) + \int_0^t F^{\nu_{j(s)}}
  (\tilde{z}^{u}(s))\, ds\,.
  \label{integral-ode}
\end{align}
Here the vector field $F^{\nu}$ is defined as the limit of
\begin{align}
  F^{\nu, N}(z) = \sum_{l \in \mathbb{X}_N} l \, f^{\nu,
  N}_{\textcolor{revisecolor}{d}}(z, l)\,, \quad z \in \mathbb{X}_N\,,
  \label{F-nu-N-general}
\end{align}
as $N\to\infty$ (see Assumption~\ref{assump-2}), \textcolor{revisecolor}{and we have used the notation
$j(\cdot)$ which is defined in (\ref{j-t})}. Convergence of $z^{u,N}$ to $\tilde{z}^u$ will be established below in Theorem~\ref{thm-2}.

\subsubsection*{Limiting control value}

 We are interested in substituting the optimal control policy for the jump
 process with an optimal open loop control
 $u_{0} \in \mathcal{U}_{o,0}$  of the limiting process, such that

\begin{align}
  J_N(z, u_{0}) \approx U_N(z) \triangleq \inf_{u \in \mathcal{U}_{\sigma, 0}} J_N(z,u)\,,
  \label{target-fun}
\end{align}
i.e. the infimum (minimum) cost is approximated under the policy $u_{0}$.

The function $U_N$ is
called the \emph{value function} or \emph{control value} of the underlying stochastic control problem.
\textcolor{revisecolor}{It is known that an optimal control $u^{N}_{opt, \sigma}=\argmin_u J_N(z, u)$ exists when $\mathcal{A}$ is a finite set; see
 \cite{Puterman1994} for more details and possible relaxations of the
 assumptions on the set of admissible controls.}

For the related deterministic limiting process  $\tilde{z}^u$ satisfying (\ref{limit-ode}) under some open loop policy $u\in \mathcal{U}_{o, 0}$,
we define the cost functional by
  \begin{align}
    \widetilde{J}(z, u) = \sum_{j=0}^{K-1} \Big[r\big(\tilde{z}^u(t_j),
    \nu_j\big) + \int_{t_j}^{t_{j+1}}
    \phi\big(\tilde{z}^u(s), \nu_j\big)\, ds \Big] +
    \psi\big(\tilde{z}^u(T)\big)\,,
  \label{cost-j-ode}
  \end{align}
  and the corresponding value function $\widetilde{U}(z) = \inf_{u \in \mathcal{U}_{o,0}} \widetilde{J}(z,u)$.
  Note that when $\mathcal{A}$ is a finite set, the minimizer exists since the number of possible open loop control policies $u$ is finite and equal to
    $|\mathcal{A}|^{K}$, i.e.~$|\mathcal{U}_{o, 0}| = |\mathcal{A}|^{K}$. Convergence of the value function $U_{N}\to\widetilde{U}$ will be established in the course of the paper.

\subsubsection*{Standing assumptions}
\textcolor{revisecolor}{Let $\Omega$ be a fixed open subset of the space $\mathbb{R}^n$}. The subsequent analysis rests on the following assumptions :

\begin{assumption}
For some fixed $1 < \alpha \le 2$, we assume that
  \begin{align}
    M_{N, \textcolor{revisecolor}{\alpha}} := \sup_{\nu \in \mathcal{A}}
    \sup_{z \in \mathbb{X}_N \textcolor{revisecolor}{\cap \Omega}}
  \Big(\sum_{l \in \mathbb{X}_N}
  |l|^\alpha f^{\nu, N}_{\textcolor{revisecolor}{d}}(z, l)\Big) < \infty \,,
  \label{M-eps}
\end{align}
and satisfies
  \[
  \lim\limits_{N\rightarrow \infty}
  M_{N,\textcolor{revisecolor}{\alpha}}= 0\,.
  \]
  \label{assump-1}
\end{assumption}

\begin{assumption}
There exist functions $F^{\nu} : \Omega \rightarrow
\mathbb{R}^n$, such that
\begin{align}
  \omega_N := \sup\limits_{z \in \mathbb{X}_N \textcolor{revisecolor}{\cap \Omega},\, \nu \in \mathcal{A}} \big|F^{\nu,
  N}(z) - F^{\nu}(z)\big|
  \label{omega-f-diff}
\end{align}
satisfies
\[
 \lim\limits_{N\rightarrow \infty}
  \omega_N= 0\,.
  \]
  \label{assump-2}
\end{assumption}
\begin{assumption}
  There exists a constant $L_F \ge 0$, \textcolor{revisecolor}{which may depend on
  the subset $\Omega$}, such that
  \[|F^{\nu}(z') - F^{\nu}(z)| \le L_F |z' - z|\,, \quad\forall z, z' \in
    \textcolor{revisecolor}{\Omega}\,,\, \nu \in \mathcal{A}\,.
  \]
  \label{assump-3}
\end{assumption}

Finally, for the functions related to the cost functional (\ref{cost-j-general}) of the optimal control problem, we suppose
\begin{assumption}
  There exist constants $L_r, L_\phi, L_\psi, M_r, M_\phi ,  M_\psi \ge 0$,
  \textcolor{revisecolor}{which may depend on the subset $\Omega$}, such that
  \begin{align*}
    &
|r(z_1, \nu) - r(z_2, \nu)| \le L_r|z_1 - z_2|\,,\quad
     |\phi(z_1, \nu) - \phi(z_2, \nu)| \le L_\phi|z_1 - z_2| ,\\
& |\psi(z_1) - \psi(z_2)| \le L_\psi|z_1 - z_2|\,,
\end{align*}
$\forall z_1, z_2 \in \textcolor{revisecolor}{\Omega}\,,\; \nu \in \mathcal{A}$.
Moreover, $|r(z, \nu)| \le M_r$, $|\phi(z, \nu)| \le M_\phi$, $|\psi(z)| \le
M_\psi$, $\forall z \in \textcolor{revisecolor}{\mathbb{R}^n}\,,\, \nu \in \mathcal{A}$.
\label{assump-4}
\end{assumption}
\textcolor{revisecolor}{
\begin{remark}
  We make some remarks on the above assumptions.
  \begin{enumerate}
\item 
  Although the constants in Assumptions~\ref{assump-1}-\ref{assump-4} may
  depend on the subset $\Omega$, we will omit the dependence, since $\Omega$ is fixed throughout this paper.
    \item
  Instead of utilizing the jump rate function of the density jump process
  $z^{u,N}$, the quantity in Assumption~\ref{assump-1} can also be
  expressed in terms of the original jump process $x^{u,N}$. In fact, using
  the relation between the functions $f^{\nu, N}_{d}$ and $f^{\nu, N}_o$,  (\ref{M-eps}) is equivalent to
  \begin{align}
    M_{N, \alpha} = 
    N^{-\alpha}\,\sup_{\nu \in \mathcal{A}}
    \sup_{z \in \mathbb{X}_N \cap \Omega}
    \Big(\sum_{l \in \mathbb{X}}
  |l|^\alpha f^{\nu, N}_{o}(Nz, l)\Big) < \infty \,.
\end{align}
    \item
Assumption~\ref{assump-2} states that $F^{\nu, N}(z)$ converges to $F^{\nu}(z)$
  uniformly for all $\nu \in \mathcal{A}$ on the subset $\Omega$,
  while Assumption~\ref{assump-3} states that the family of the limiting vector fields $F^{\nu}(z)$ are (local) Lipschitz functions with Lipschitz constant
  $L_F$ on the set $\Omega$, uniformly for $\nu \in \mathcal{A}$.
  Similarly, Assumption~\ref{assump-4} assures that the functions $r, \phi, \psi$ are
  Lipschitz on $\Omega$ and are bounded on $\mathbb{R}^n$, uniformly for $\nu \in \mathcal{A}$.
  \end{enumerate}
  \label{rmk-assump}
\end{remark}
}

\subsection{Applications}
\label{sub-concrete}
Here we consider two prototypical examples of Markov jump processes, which appear relevant in the context of optimal control and to which our results can be applied.

\subsubsection*{Density dependent Markov chain}
\textcolor{revisecolor}{The first example is the \emph{density dependent Markov chain}
\cite{limit_thm_ode1971}, where the jump rates of the original process depend
on the density of the system's states.}
Specifically, following the notations of Subsection~\ref{sub-setup-mjp} and denoting the density
dependent Markov chain as $x^{\nu, N}(\cdot)$, it holds that the rate of
jumping from state $x$ to $x + l$ under the control $\nu \in \mathcal{A}$
is given by $f^{\nu,N}_{\textcolor{revisecolor}{o}}(x,l) = N{\textcolor{revisecolor}{\eta}}^{\nu}(x/N, l)$ for $x, l \in \mathbb{X}$,
\textcolor{revisecolor}{where $\eta^{\nu} : \mathbb{R}^n \times \mathbb{X} \rightarrow
\mathbb{R}^+$ is a function independent of $N$}. As a consequence,
$$f^{\nu, N}_{\textcolor{revisecolor}{d}}(z, l/N) =
f^{\nu,N}_{\textcolor{revisecolor}{o}}(Nz, l) = N {\textcolor{revisecolor}{\eta}}^{\nu}(z, l)$$ 
is the rate at which the normalized density process $z^{\nu,N}(\cdot) =
N^{-1}x^{\nu,N}(\cdot)$ jumps from $z = x/N$ to $z + l/N = (x + l)/N$.
Concrete models of density dependent Markov chains include the predator-prey model, elementary chemical reactions such as \ce{B + C <=> D} or epidemic
models \cite{limit_thm_ode1971, ode_kurtz_limit}.

Notice that if we assume 
\begin{align}
M_\alpha = \sup_{\nu \in \mathcal{A}}
\sup_{z \in \textcolor{revisecolor}{\Omega}}
   \Big(\sum_{l \in \mathbb{X}} |l|^{\alpha}
   {\textcolor{revisecolor}{\eta}}^{\nu}(z, l)\Big) < \infty\,,
  \label{M-alpha}
\end{align}
then Assumption~\ref{assump-1} holds, since $M_{N,\textcolor{revisecolor}{\alpha}} = N^{1-\alpha} M_\alpha$ with $\alpha > 1$.  Furthermore, if we define
  \begin{align}
    F^\nu(z) = \sum_{l \in \mathbb{X}} l\,{\textcolor{revisecolor}{\eta}}^{\nu}(z,l)\,,\quad  \forall z \in
    \mathbb{R}^n\,,
    \label{density-dep-f-nu}
  \end{align}
  then (\ref{F-nu-N-general}) becomes
  \begin{align*}
    F^{\nu, N}(z) =
     \sum_{l \in
     \mathbb{X}_N} l\,f^{\nu, N}_{\textcolor{revisecolor}{d}}(z, l) =
    \sum_{l \in
    \mathbb{X}} \frac{l}{N}\cdot N{\textcolor{revisecolor}{\eta}}^{\nu}(z, l) =
    F^{\nu}(z)\,, \quad z \in \mathbb{X}_N\,,
  \end{align*}
  \textcolor{revisecolor}{where the function $F^\nu(z)$ is independent of $N$}. This implies that Assumption~\ref{assump-2} trivially holds with $\omega_N
  \equiv 0$. \\

\subsubsection*{Chemical reactions}
As a second example, we mention systems of chemical reactions. Consider a reaction network
consisting of $n$ chemical species that can undergo $m$ different
chemical reactions:
\begin{align}
  \sum_{i = 1}^{n} v_{ki}\, S_i \xrightarrow{\kappa_k} \sum_{i = 1}^{n}
  v_{ki}'\, S_i\,, \quad
  k = 1, \cdots, m\,.
  \label{reaction-network}
\end{align}
Here $S_i$ are the different chemical species, $\kappa_k$ is the rate constant
of the $k$-th reaction, $v_{ki}$, $v_{ki}'$ are the
molecule numbers of species $S_i$ consumed or generated when the $k$-th
reaction fires. Now let $x^{(i)}(t)$ be the number of molecules of species $S_i$ at time $t$ and define
\begin{align}
  x(t) = \big(x^{(1)}(t), x^{(2)}(t), \cdots, x^{(n)}(t)\big)^T \in \mathbb{N}^n\,
\end{align}
to be the state of the chemical system at time $t$. When the $k$-th reaction
fires at time $t > 0$, the system's state
jumps from $x(t)$ to $x(t) + (v'_k - v_k)$ where
\begin{align}
   v_k = (v_{k1}, v_{k2}, \cdots, v_{kn})^T \in \mathbb{N}^n, \quad v'_k = (v'_{k1}, v'_{k2}, \cdots, v'_{kn})^T \in \mathbb{N}^n\,.
\end{align}
In order to fully describe the system as a Markov jump process, we
still need to specify the Poisson intensity of each reaction (propensity
function). Let $\lambda$ denote a generic propensity function. For simplicity, we
will restrict ourself to \textcolor{revisecolor}{at most binary} reactions, which
\textcolor{revisecolor}{consume} at most two molecules :
\begin{enumerate}
  \item
    \ce{$\emptyset$ ->[\kappa] {\color{revisecolor}product}}\,, \;  $\lambda =
     \kappa N$\,
  \item
    \ce{$S_i$ ->[\kappa] {\color{revisecolor}product}}\,, \;  $\lambda = \kappa
    x^{(i)}$\,
  \item
    \ce{2$S_i$ ->[\kappa] {\color{revisecolor}product}}\,, \; $\lambda =
    \frac{\kappa}{N} x^{(i)}(x^{(i)} - 1)$\,
  \item
    \ce{$S_i$ + $S_j$ ->[\kappa] {\color{revisecolor}product}}\,, \;$\lambda =
    \frac{\kappa}{N} x^{(i)}x^{(j)}$\,,
\end{enumerate}
where $x = (x^{(1)}, \cdots, x^{(n)})$ is the system's state and $N$ is a constant
related to the volume of the system (e.g., the total number of molecules or a test tube volume).
In the above reactions $1-4$, $\kappa$ is a constant of order one and
\textcolor{revisecolor}{the scaling of $\lambda$ with respect to $N$ corresponds to the ``classical scaling'' considered in~\cite{kang2013,ball06}. We also refer to \cite{ssck_gillespie} for
further discussions on the propensity functions.}
Note that, in general the propensity function is a function of the system state. 

For the reaction network described in (\ref{reaction-network}), denoting the
propensity functions $\lambda_k(x)$ when the system is at state $x$, then the dynamics of $x(t)$ can be written as
\begin{align}
  x(t) = x(0) + \sum_{k=1}^{m} (v_k' - v_k)\, Y_k\left(\int_0^t
  \lambda_k\big(x(s)\big)\,  ds\right)
  \label{poisson-rep}
\end{align}
where $Y_k(\cdot), 1 \le k \le m$ are independent Poisson processes with
unit intensity. For the system of controlled chemical reactions, we use the
notation $\lambda^{\nu,N}_{k}(x)$ to indicate that the propensities not only
depend on $N$, but also on the control $\nu\in \mathcal{A}$ via the rate
constants $\kappa=\kappa_k(\nu)$. From the definition of the reaction events,
it is clear that the jump rates introduced before and the propensity functions are related by
\[
  f^{\nu,\textcolor{revisecolor}{N}}_{\textcolor{revisecolor}{o}}(x, l) = \sum_{\substack{1
  \le k \le m\,,\\  v_k' - v_k=l}}\lambda_k^{\nu,N}(x) \,.
 \]
Notice that if only reactions of type $1$, $2$ or $4$ are involved, the process
defined by $f^{\nu,\textcolor{revisecolor}{N}}_{\textcolor{revisecolor}{o}}$ is an instance of the aforementioned
density dependent Markov chain; when reactions of type $3$ are involved, then the
limiting vector field $F^\nu$ can be computed from $F^{\nu,N}$ in
\textcolor{revisecolor}{(\ref{F-nu-N-general})} by exploiting that
$f^{\nu, N}_{\textcolor{revisecolor}{d}}(z, l) = N\kappa z^{(i)}(z^{(i)}- N^{-1})$ for $z, l
\in \mathbb{X}_N$, where
$\kappa = \kappa_k(\nu)$, if $\textcolor{revisecolor}{N}l = v_k' - v_k$ for some $1
\le k \le m$ (for simplicity suppose only one such index $k$ exists) and $f^{\nu, N}_{\textcolor{revisecolor}{d}}(z, l)=0$ otherwise.

\section{Asymptotic analysis of the optimal control problem}
\label{sec-analysis-ocp}
In this section, we study optimal control problems
in the large number regime based on Kurtz's limit theorem \cite{limit_thm_ode1971, ode_kurtz_limit,
Kurtz1978, kurtz_limit_thm_diffusion_approx}.

As a first step, given an open loop control $u \in \mathcal{U}_{o,0}$, we establish the approximation result of the Markov jump
process $z^{u,N}$ by the ODE limit~(\ref{limit-ode}). \textcolor{revisecolor}{The proof is adapted
from Kurtz's argument, especially~\cite{limit_thm_ode1971}. However, for completeness we feel
it is necessary to present the proof in detail. } 
As a second step, we confine our attention to the open loop control problem
which is a direct application of Kurtz's theorem, given that
the Assumptions in Subsection~\ref{subsec-assumptions} hold. Specifically, we show
that $J_N(z_N,u) \to \widetilde{J}(z_0, u)$ for $u \in \mathcal{U}_{o,0}, z_N
\in \mathbb{X}_N, z_N \rightarrow z_0 \in \Omega$ as $N \to \infty$ (Theorem~\ref{thm-3}). Then, as a
third step, we consider the feedback control problem and prove that $U_N(z_N) \to
\widetilde{U}(z)$ if $z_N \rightarrow z$,
and, especially, if $u_0 \in \mathcal{U}_{o,0}$ and $\widetilde{J}(z,u_0) =
\widetilde{U}(z)$, then $|J_N(z_N, u_0) - U_N(z_N)| \to 0$ as $N\to\infty$ (Theorem~\ref{thm-4}).
As we will discuss in detail, an important consequence of Theorem~\ref{thm-4} is that the optimal (open
loop) control policy for the limiting ODE system is almost optimal for the
Markov jump process if $N \gg 1$, i.e., it  is asymptotically optimal among
all \textit{feedback} control policies in $\mathcal{U}_{f, 0}$. Finally, we extend the analysis of the finite time-horizon case to
   discounted optimal control problems on an infinite time-horizon (Theorem~\ref{thm-5}).

\subsection{ODE approximation of the normalized Markov jump process}
\label{sub-ode-appro}
Let $u \in \mathcal{U}_{o, 0}$ be some open loop control policy and $z^{u, N}(t)=N^{-1}x^{u,N}(t)$ denote the
normalized density Markov jump process. Recall that $\Omega$ is the open subset of $\mathbb{R}^n$ introduced in Subsection~\ref{subsec-assumptions}.
The convergence of the normalized density process as $N\to\infty$ is described by the following theorem.
\begin{theorem}
  Let $z^{u, N}(t)$ be the normalized density jump process under the open loop policy $u \in
  \mathcal{U}_{o, 0}$ and 
  \textcolor{revisecolor}{
  suppose the ODE (\ref{limit-ode}) has a unique solution $\tilde{z}^u(t)$ 
  on $t \in [0, T]$ starting from $z_0 \in \Omega$. 
  Furthermore, $\exists \gamma>0$, s.t. 
  \begin{align}
    \Omega^u_{\gamma,z_0, [0, T]} := \Big\{z' \in \mathbb{R}^n~\Big|~\inf_{0 \le t \le
    T}|z'-\tilde{z}^u(t)| \le \gamma\Big\} \subseteq \Omega \,.
    \label{omega-gamma-set}
  \end{align}
  Let $\tau^u_N$ be the stopping time for the jump process
    $z^{u,N}$ to leave the set $\Omega^u_{\gamma, z_0,
  [0,T]}$, i.e.  
\begin{align}
  \tau^{u}_N := \inf_{s \ge 0} \Big\{s\,\big|~z^{u,N}(s) \not\in\,
  \Omega^u_{\gamma, z_0, [0,T]}\Big\}\,.
  \label{tau-u-n}
\end{align}
}
\begin{enumerate}
  \item
Suppose Assumption~\ref{assump-3} holds. We have 
\begin{align}
  \mathbf{E}^u\!\left[\sup_{0 \le s \le t \textcolor{revisecolor}{\wedge \tau^{u}_N}} \big|z^{u, N}(s) - \tilde{z}^u(s)\big|\right] \le
  \Big[\mathbf{E}\big|z^{u, N}(0) - z_0\big| + C_{T,N} \Big]\,e^{L_F t}\,,
\label{thm-2-result-1}
\end{align}
for $0 \le t \le T$, where the constant 
\begin{align}
  C_{T,N} = T\omega_N + \frac{\alpha}{2(\alpha -
  1)}\Big(\frac{4TM_{N,\textcolor{revisecolor}{\alpha}}}{\alpha - 1}\Big)^{\frac{1}{\alpha}}\,,
  \label{c-t-n}
\end{align}
with $\alpha \in (1,2]$, and $\omega_N$, $M_{N,\textcolor{revisecolor}{\alpha}}$ are defined in
(\ref{omega-f-diff}) and (\ref{M-eps}), respectively. 
\item
  Suppose Assumptions~\ref{assump-1}-\ref{assump-3} are satisfied
\textcolor{revisecolor}{with constant $\alpha \in (1,2]$} and that $$\lim\limits_{N\rightarrow \infty} \mathbf{E}|z^{u, N}(0) - z_0|
  =0.$$ Then for any control policy $u \in \mathcal{U}_{o,0}$, we have 
\begin{align}
  \lim_{N\rightarrow\infty} \mathbf{E}^u\!\left[\sup_{0 \le s \le t
    \textcolor{revisecolor}{\wedge \tau^u_N}} |z^{u, N}(s) -
\tilde{z}^u(s)| \right] = 0 \,.
\label{thm-2-result-2}
\end{align}

\textcolor{revisecolor}{Furthermore, let $\rho > 0$ be given such that $\rho \le \frac{1}{3} \gamma e^{-L_F T}$. Then $\exists N_0
> 0$ which may depend on $\rho$, such that 
\begin{align}
  \mathbf{P}(\tau^u_N < T) \le
  \rho^{-1}\bigg[\mathbf{E}|z^{u,N}(0)-z_0| + 
\frac{\alpha}{2(\alpha -
  1)}\Big(\frac{4TM_{N,\textcolor{revisecolor}{\alpha}}}{\alpha -
1}\Big)^{\frac{1}{\alpha}}\bigg]\,,
\label{thm-2-result-3}
\end{align}
whenever $N \ge N_0$, where $\mathbf{P}$ is the probability with respect to the process $z^{u,N}$ under the control $u$.
Especially, we have 
\begin{align*}
  \lim_{N\rightarrow\infty} \mathbf{P} (\tau^u_N < T) = 0\,.
\end{align*}
}
  \end{enumerate}
\label{thm-2}
\end{theorem}
\begin{proof}
  \begin{enumerate}
    \item
Let $w^{u,N}$ be the martingale 
\begin{align}
  w^{u, N}(t) = z^{u, N}(t) - z^{u, N}(0)
- \int_0^t F^{\nu_{j(s)}, N}(z^{u, N}(s))\, ds \,,
\label{martingale-w}
\end{align}
and consider the coupled Markov process $(z^{u, N}(t), w^{u, N}(t))$. For a differentiable function
$\varphi$ of $w$, Dynkin's formula~\cite{markov_process_dynkin, oksendalSDE} entails
\begin{align*}
  & \mathbf{E}^u\!\left[\varphi\big(w^{u, N}(t \textcolor{revisecolor}{\wedge \tau^u_N})\big)\right] - \mathbf{E}^u\!\left[\varphi\big(w^{u,
  N}(0)\big)\right] \\
  = &
  \mathbf{E}^u \bigg\{\int_0^{t \textcolor{revisecolor}{\wedge \tau^u_N}}\Big[
  \sum_{l \in \mathbb{X}_N} \Big(\varphi\big(l + w^{u,
  N}(s)\big) - \varphi\big(w^{u, N}(s)\big) -l \cdot \nabla \varphi\big(w^{u,
  N}(s)\big) \Big)\\
  & \quad \times f^{\nu_{j(s)}, N}_{\textcolor{revisecolor}{d}}(z^{u, N}(s), l)
  \Big]\, ds\bigg\}\,.
\end{align*}
In particular, setting $\varphi(z) = |z|^\alpha$, \textcolor{revisecolor}{where $\alpha \in (1,2]$ is
the constant in Assumption~\ref{assump-1}}, and using Lemma~\ref{lemma-1} from Appendix~\ref{app-1}, we obtain
\begin{align*}
  & \mathbf{E}^u\big|w^{u, N}(t \textcolor{revisecolor}{\wedge \tau^u_N})\big|^\alpha \le
  \frac{4 t}{2^\alpha(\alpha - 1)} \sup_{\nu \in \mathcal{A}}
  \sup_{z \in \mathbb{X}_N \textcolor{revisecolor}{\cap \Omega}}
  \Big(\sum_{l \in \mathbb{X}_N} |l|^\alpha f^{\nu,
  N}_{\textcolor{revisecolor}{d}}(z, l)\Big) = \frac{4tM_{N,\textcolor{revisecolor}{\alpha}}}{2^\alpha(\alpha-1)}  \,,
\end{align*}
which, by H\"older's inequality and Doob's maximal inequality, implies that
\begin{equation}  \label{w-ineq}
\begin{aligned}
  \mathbf{E}^u\!\left[\sup_{0 \le s \le t} \big|w^{u, N}(s \textcolor{revisecolor}{\wedge
  \tau^u_N})\big|\right] & \le \!\left[\mathbf{E}^u\big(\sup_{0 \le s \le t}
      \big|w^{u, N}(s \textcolor{revisecolor}{\wedge \tau^u_N})\big|^\alpha\big)\right]^{\frac{1}{\alpha}}\\
  & \le \frac{\alpha}{\alpha - 1}
\Big[\mathbf{E}^u\,\big|w^{u, N}(t\textcolor{revisecolor}{\wedge \tau^u_N})\big|^\alpha\Big]^{\frac{1}{\alpha}}\\
  & \le \frac{\alpha}{2(\alpha - 1)}\left(\frac{4 t
M_{N,\textcolor{revisecolor}{\alpha}}}{\alpha -
  1}\right)^{\frac{1}{\alpha}} \,.
\end{aligned}
\end{equation}
Combining (\ref{martingale-w}) and (\ref{integral-ode}) and taking Assumption~\ref{assump-3} into consideration, it follows that
\begin{align*}
  & \big|z^{u, N}(t\textcolor{revisecolor}{\wedge \tau^u_N}) -
  \tilde{z}^u(t\textcolor{revisecolor}{\wedge \tau^u_N})\big| \\
   \le &
\big|z^{u, N}(0) - z_0\big| +
L_F\int_0^{t \textcolor{revisecolor}{\wedge \tau^u_N}} \big|z^{u, N}(s) - \tilde{z}^u(s)\big| ds \\
&+ \int_0^{t\textcolor{revisecolor}{\wedge \tau^u_N}}
  \big|F^{\nu_{j(s)}, N}\big(z^{u, N}(s)\big) - F^{\nu_{j(s)}}\big(z^{u,
  N}(s)\big)\big| ds
   + \big|w^{u, N}(t\textcolor{revisecolor}{\wedge \tau^u_N})\big|\\
   \le &
\big|z^{u, N}(0) - z_0\big| + L_F\int_0^{t\textcolor{revisecolor}{\wedge \tau^u_N}} \big|z^{u, N}(s) - \tilde{z}^u(s)\big| ds
   + t \,\omega_N + \big|w^{u, N}(t\textcolor{revisecolor}{\wedge \tau^u_N})\big|\,.
\end{align*}
Now let $y^{u, N}(t) = \sup\limits_{0 \le s \le t\textcolor{revisecolor}{\wedge \tau^u_N}} \big|z^{u, N}(s) - \tilde{z}^u(s)\big|$. Then
\begin{align*}
  y^{u, N}(t) \le
  y^{u, N}(0) +
  L_F \int_0^t y^{u, N}(s) \,ds + T\omega_N + \sup_{0 \le s \le T} \big|w^{u,N}(s\textcolor{revisecolor}{\wedge \tau^u_N})\big|\,,
\end{align*}
and Gronwall's inequality implies
\begin{align}
  y^{u, N}(t) \le \left[ y^{u, N}(0) + T\omega_N + \sup_{0 \le s \le T} \big|w^{u,N}(s\textcolor{revisecolor}{\wedge \tau^u_N})\big|\right] e^{L_F t}\,.
  \label{thm-2-estimate-y}
\end{align}
The estimate (\ref{thm-2-result-1}) follows by taking expectations on both
sides of the above inequality and using (\ref{w-ineq}). 
\item
  {\color{revisecolor} The assertion (\ref{thm-2-result-2}) follows directly from
    (\ref{thm-2-result-1}) by taking the limit $N\to\infty$ and applying
    Assumptions~\ref{assump-1} and \ref{assump-2}.
  To prove the assertion (\ref{thm-2-result-3}), we first choose $N_0 > 0$ such that
  $T\omega_N \le \rho$ whenever $N > N_0$. This is possible due to Assumption~\ref{assump-2}.

  From the definitions of $y^{u,N}(t)$, the subset $Q^u_{\gamma, z_0, [0,T]}$ in
  (\ref{omega-gamma-set}) and the inequality (\ref{thm-2-estimate-y}), we can
  deduce that
  \begin{align*}
    & y^{u,N}(0) \le \rho \quad \mbox{and}\,\quad \sup_{0 \le s \le T}
    |w^{u,N}(s\wedge \tau^u_N)| \le \rho \\
    \Longrightarrow & \quad y^{u,N}(T) \le 3\rho e^{L_F T} \le \gamma 
    \quad \Longrightarrow \quad \tau^u_N \ge T\,,
  \end{align*}
  and therefore 
  \begin{align*}
    \mathbf{P}(\tau^u_N < T) \le &\, \mathbf{P}\big(y^{u,N}(0) > \rho \big) +
    \mathbf{P}\Big(\sup_{0 \le s \le T} |w^{u,N}(s\wedge \tau^u_N)| >\rho \Big)
    \\
    \le&\, \rho^{-1}\bigg[\mathbf{E}|z^{u,N}(0)-z_0| + 
\frac{\alpha}{2(\alpha -
  1)}\Big(\frac{4TM_{N,\alpha}}{\alpha -
1}\Big)^{\frac{1}{\alpha}}\bigg]\,,
  \end{align*}
  where we have used the fact that $y^{u,N}(0) = |z^{u,N}(0) - z_0|$, the
  inequality (\ref{w-ineq}) and the Chebyshev's inequality.
}
\end{enumerate}
\end{proof}

We conclude this subsection with the following remarks.

\begin{remark}
 From the proof, it is straightforward to see that, when $z^{u,N}(0)$ is
 deterministic and $|z^{u,N}(0) - z_0| \le \rho \le \frac{1}{3} \gamma
 e^{-L_FT}$, estimate (\ref{thm-2-result-3}) can be improved as
\begin{align}
  \mathbf{P}(\tau^u_N < T) \le
  \rho^{-1}\bigg[\frac{\alpha}{2(\alpha -
  1)}\Big(\frac{4TM_{N,\textcolor{revisecolor}{\alpha}}}{\alpha -
1}\Big)^{\frac{1}{\alpha}}\bigg] \le \rho^{-1} C_{T,N}\,.
\label{thm-2-result-3-improved}
\end{align}
\label{rmk-kurtz-1}
\end{remark}

  \begin{remark}
  \label{rmk-kurtz-density-depend}
For the density dependent Markov chain introduced in Subsection~\ref{sub-concrete}, it
holds that $\omega_{N}=0$ and
$M_{N,\textcolor{revisecolor}{\alpha}} = N^{1-\alpha} M_\alpha$, where $M_\alpha$ is
given in (\ref{M-alpha}) with $\alpha \in (1,2]$.
Therefore, the constant in (\ref{c-t-n}) satisfies
  \begin{align}
    C_{T,N} = \mathcal{O}\left(N^{\frac{1}{\alpha} - 1}\right)\,.
 \label{const-n-order}
  \end{align}
  \textcolor{revisecolor}{
    Assuming $\mathbf{E}|z^{u,N}(0) - z_0|\rightarrow 0$ fast enough as $N\rightarrow \infty$,
the above implies that the convergence speed in both (\ref{thm-2-result-1}) and
(\ref{thm-2-result-3}) is explicitly of order $N^{\frac{1}{\alpha} - 1}$.}

  \textcolor{revisecolor}{The simplest case is when $z^{u,N}$ is a one-dimensional
    process and the control set $\mathcal{A}$ is a singleton}. For simplicity, we will omit the control $u$ in the notations in
  the remainder of this paragraph. Suppose that $\textcolor{revisecolor}{\eta}(z, 1)
  = 1$ and $\textcolor{revisecolor}{\eta}(z,l) = 0$
for $l \neq 1$, $z \ge 0$. Then 
(\ref{density-dep-f-nu}) implies that $F(z) \equiv 1$, which is Lipschitz continuous with Lipschitz constant $L_F = 0$. For the initial value $z_0 = 0$, equation
    (\ref{integral-ode}) yields $\tilde{z}(t) = t$ and $z^{N}(t) =
    N^{-1} P(Nt)$, where $P(\cdot)$ is a Poisson process with unit intensity.
    \textcolor{revisecolor}{We can also choose the subsets $\Omega_{\gamma, 0, [0,T]}=\Omega=\mathbb{R}^n$.}
    Further note that Assumption~\ref{assump-1} holds with $\alpha = 2$ and $M_{\alpha} = 1$, so that Theorem~\ref{thm-2} entails
  \begin{align*}
      \mathbf{E}\!\left[\sup_{0 \le s \le T} \Big|\frac{P(Ns)}{N} - s\Big|\right] \le
    \left(\frac{4T}{N}\right)^{1/2} \,.
  \end{align*}
\end{remark}

\subsection{Optimal control on finite time-horizon}
\label{sub-finite-time}
In this subsection, we apply the previous approximation result to study both open
and closed loop optimal control on a finite time-horizon.

\subsubsection*{Open loop control}
As a straight consequence of Theorem~\ref{thm-2} and
Assumptions~\ref{assump-3}--\ref{assump-4}, we have the following result
for the open loop control problem.
  \begin{theorem}
    Suppose that Assumptions~\ref{assump-1}-\ref{assump-4} hold true.
      Let $z_0 \in \Omega$ and $u \in \mathcal{U}_{o, 0}$ be any open loop control policy of the form $u = (\nu_0, \nu_1, \cdots, \nu_{K-1})$
    with $\nu_{j} \in \mathcal{A}$, $0 \le j < K$. 
    Suppose the ODE (\ref{limit-ode}) has a unique solution on $[0, T]$ and 
    furthermore the condition (\ref{omega-gamma-set}) is satisfied for some
    $\gamma > 0$.  Recall that the cost functionals $J_N$ and $\widetilde{J}$ are defined in (\ref{cost-j-general}),
    (\ref{cost-j-ode}), respectively. 
  Let $z_N \in \mathbb{X}_N \textcolor{revisecolor}{\cap \Omega}$ and $z_N
  \rightarrow z_0$ as $N\rightarrow +\infty$.  Then $\exists N_0 > 0$, s.t. for $N > N_0$, we have 
\begin{align}
\begin{split}
  \big|J_N(z_N, u) - \widetilde{J}(z_0, u)\big| 
  \le &
  \Big(|z_N - z_0| + C_{T,N}\Big)
  \Big[L_\phi \frac{e^{L_F T} - 1}{L_F}
  +\big(KL_r + L_\psi + \textcolor{revisecolor}{ \overline{M}}\big) e^{L_FT} \Big] \,,
  \label{cost-approx-thm-3}
\end{split}
\end{align}
with the convention
$\frac{e^{L_F T} - 1}{L_F}=T$,if $L_{F}=0$,
and the constant 
\begin{align}
  \textcolor{revisecolor}{ \overline{M} := 6 \gamma^{-1} \big(KM_r + TM_\phi + M_\psi\big)\,.}
  \label{m-overline}
\end{align}
The constant $C_{T,N}$ is defined in (\ref{c-t-n}) and the other constants are given in Assumptions~\ref{assump-3}--\ref{assump-4}.
Especially, when the condition (\ref{omega-gamma-set}) is satisfied for all $u \in
\mathcal{U}_{o,0}$ for some common $\gamma>0$, we have
\begin{align}
  \lim_{N\rightarrow \infty} |J_N(z_N, u) - \widetilde{J}(z_0, u)| = 0\,,
\end{align}
 uniformly for all control policies $u\in \mathcal{U}_{o,0}$.
\label{thm-3}
\end{theorem}
\begin{proof}
  First of all, let us define the quantity 
  \begin{align*}
    I =& \sum_{j = 0}^{K-1} \Big[
  r(z^{u, N}(t_j), \nu_j) - r(\tilde{z}^u(t_j), \nu_j)
    + \int_{t_j}^{t_{j+1}} \Big(\phi(z^{u, N}(s), \nu_{j}) -
   \phi(\tilde{z}^{u}(s), \nu_{j})\Big) \,ds\Big] \\
   & + \psi(z^{u,N}(T)) - \psi\big(\tilde{z}^u(T)\big)\,.
  \end{align*}
  Then the boundedness conditions in Assumption~\ref{assump-4}
  immediately imply $|I| \le 2\big(KM_r + TM_\phi + M_\psi\big)$. 
  Recalling the stopping time $\tau^u_N$ in (\ref{tau-u-n}) and the Lipschitz
  conditions in Assumption~\ref{assump-4}, we also have 
  \begin{align}
    \begin{split}
    |I| \le & 
     \sum_{j = 0}^{K-1} \Big\{L_r\big|z^{u, N}(t_j) -
  \tilde{z}^{u}(t_j)\big| + L_\phi \int_{t_j}^{t_{j+1}} \big|z^{u, N}(s) - \tilde{z}^{u}(s)
 \big| ds \Big\} \\
 & + L_\psi \big|z^{u, N}(T) - \tilde{z}^{u}(T)\big|\,, 
 \end{split}
 \label{thm3-inproof-1}
  \end{align}
  as long as $\tau^u_N \ge T$.  Therefore, using the definitions of the cost functions $J_N$,
  $\widetilde{J}$,  we have 
  \textcolor{revisecolor}{
\begin{align*}
  & \big|J_N(z_N, u) - \widetilde{J}(z_0, u)\big| =  \big|\mathbf{E}^u_{z_N}\,I| \\
  \le& 
  \big|\mathbf{E}^u_{z_N}(I\cdot\mathbbm{1}_{\{\tau^{u}_N \ge T\}})\big| +
  \big|\mathbf{E}^u_{z_N}(I\cdot\mathbbm{1}_{\{\tau^{u}_N < T\}})\big| \\
  \le & \mathbf{E}^u_{z_N}\big(|I|\cdot\mathbbm{1}_{\{\tau^{u}_N \ge T\}}\big) +
2\big(KM_r + TM_\phi + M_\psi\big)\, \mathbf{P}\big(\tau^{u}_N < T)\,,
\end{align*}
}
where $\mathbbm{1}$ denotes the indicator function. 
For the first term above, noticing the fact
$$\mathbf{E}^u_{z_N}\!\left[\Big(\sup_{0 \le s \le t} \big|z^{u, N}(s) - \tilde{z}^u(s)\big|\Big)
\mathbbm{1}_{\{\tau^{u}_N \ge T\}}
\right]
\le \mathbf{E}^u_{z_N}\Big[\sup_{0 \le s \le t \wedge \tau^{u}_N} \big|z^{u, N}(s)
  - \tilde{z}^u(s)\big|\Big],$$
using (\ref{thm3-inproof-1}) and applying Theorem~\ref{thm-2}, we obtain
  \begin{align*}
    & \textcolor{revisecolor}{\mathbf{E}^u_{z_N}\big(|I|\cdot\mathbbm{1}_{\{\tau^{u}_N \ge
T\}}\big)} \\
  \le & \Big(|z_N - z_0| + C_{T,N}\Big)\Big[\sum_{j = 0}^{K-1}\Big(L_\phi
\int_{t_j}^{t_{j+1}} \, e^{L_F\,s} ds + L_r e^{L_F\,t_j}\Big)+ L_\psi e^{L_F T}\Big] \\
  \le & \left\{L_\phi \frac{e^{L_F T} - 1}{L_F} + \big(KL_r + L_\psi\big) e^{L_FT}\right\}\,\Big(|z_N - z_0| + C_{T,N}\Big)\,.
\end{align*}
\textcolor{revisecolor}{
Now fix the constant $\rho= \frac{1}{3}\gamma e^{-L_FT}$ and choose $N_0$ such
that $|z_N-z_0| \le \rho$ when $N>N_0$. The assertion (\ref{cost-approx-thm-3}) then follows after we estimate
$\mathbf{P}(\tau^u_N < T)$ by applying Theorem~\ref{thm-2}. See (\ref{thm-2-result-3-improved}) in Remark~\ref{rmk-kurtz-1}.}
The convergence of the cost function $J_N$ to $\widetilde{J}$ follows from
(\ref{cost-approx-thm-3}) directly. 
\end{proof}

\subsubsection*{Feedback control}
Now we consider the case of a feedback control problem.
In accordance with (\ref{cost-j-general}), we define the cost functional for $u \in \mathcal{U}_{f, k}$,
$z \in \mathbb{X}_N \textcolor{revisecolor}{\cap \Omega}$, $0 \le k < K$ and the corresponding value function as
\begin{align}
  \begin{split}
  J_N(z,u, k) & = \mathbf{E}^u_{t_{k},z}\!\left[\sum_{j=k}^{K-1} \Big(
r(z^{u,N}(t_j), \nu_j) +  \int_{t_j}^{t_{j+1}}
\phi(z^{u,N}(s), \nu_j) \,ds\Big) +  \psi\big(z^{u,N}(T)\big)\right], \\
U_N(z,k) & = \inf_{u \in \mathcal{U}_{f,k}} J_N(z,u, k)\,,
\end{split}
\label{fb-j-u-finite}
\end{align}
with the shorthand $ \mathbf{E}^u_{t_{k},z}[\cdot]=
\mathbf{E}^u[\,\cdot\,|\,z^{u,N}(t_{k})=z]$ for the conditional expectation
over all realizations of the controlled process starting at
$z^{u,N}(t_{k})=z$. 
\textcolor{revisecolor}{Notice that, following the convention in
  Subsection~\ref{sub-setup-mjp}, we have used the same
notation $\nu_j$ to denote both the control policy function which depends on
system's state, and the value of the control selected at $t_j$, i.e.
we have $\nu_j=\nu_j(z^{u,N}(t_j))$ in (\ref{fb-j-u-finite}). See the discussion after
(\ref{control-policy-set}).}
By definition, the value function $U_N$, also called
the \emph{optimal cost-to-go}, is the minimum cost value from time $t_k$ to $T$ as a
function of the initial data $(z,t_{k})$. In particular, it holds that $U_N(z,K) =
\psi(z)$.

Then in complete analogy with the above definitions, we define
\begin{align*}
  \widetilde{J}(z,u,k) =& \sum_{j=k}^{K-1} \Big(
r(\tilde{z}^{u}(t_j), \nu_j)  +
\int_{t_j}^{t_{j+1}}
  \phi(\tilde{z}^{u}(s), \nu_j)\, ds\Big) +
  \psi\big(\tilde{z}^{u}(T)\big) \,,\quad u \in \mathcal{U}_{o, k}\,,\\
  \widetilde{U}(z,k) =& \inf_{u \in \mathcal{U}_{o,k}} \widetilde{J}(z,u,k)\,,
\end{align*}
for $z \in \Omega$,
to be the cost functional and the value function of the deterministic limiting process.
In what follows, we will omit the dependence of $J_N$,
$\widetilde{J}$ and $U_N$, $\widetilde{U}$ on
$k$ when $k = 0$ so that the notations are consistent with
(\ref{cost-j-general}) and (\ref{cost-j-ode}).

By the dynamic programming principle \cite{Puterman1994}, the necessary conditions for optimality are given in terms of Bellman's equations for the two value functions :
\begin{align}
  \begin{split}
  U_N(z_N,k) &= \inf_{\nu \in \mathcal{A}} \mathbf{E}^\nu\!\left[r(z_N,\nu) +
    \int_{t_k}^{t_{k+1}} \phi(z^{\nu, N}(s) ,\nu)\, ds +
  U_N(z^{\nu, N}(t_{k+1}), k+1) \right],  \\
  \widetilde{U}(z,k) &= \inf_{\nu \in \mathcal{A}} \left\{r(z,\nu) +
  \int_{t_k}^{t_{k+1}} \phi(\tilde{z}^\nu(s) ,\nu)\, ds +
\widetilde{U}(\tilde{z}^\nu(t_{k+1}), k+1) \right\},
\end{split}
\label{optimality-eqn}
\end{align}
with $0 \le k \le K - 1$,
where $z^{\nu, N}(t_{k}) = z_N \in \mathbb{X}_N$, $\tilde{z}^{\nu}(t_{k})=z
\in \textcolor{revisecolor}{\Omega}$
and the terminal conditions
\begin{align}
  U_N(z_N,K) = \psi(z_N)\,,\qquad  \widetilde{U}(z,K) = \psi(z)\,.
  \label{terminal-value-fun}
\end{align}
\textcolor{revisecolor}{Notice that in (\ref{optimality-eqn}), we have used the notation
$\mathbf{E}^\nu=\mathbf{E}^\nu_{t_{k},z_N}$ for the conditional expectation and $z^{\nu, N}(t)$,
$\tilde{z}^\nu(t)$ for the processes, since the involved quantities and
processes only depend on the control $\nu$ selected at $t_k$, rather than the whole control policy.
}

Before we proceed, we shall first introduce some
constants in order to simplify the analysis later on.
Let $h= \max\big\{|t_{j+1} - t_j| \colon 0\le j \le K-1 \big\}$. In accordance with (\ref{c-t-n}),
we set
\begin{equation}\label{c-h-n}
  C_{h,N} = h\,\omega_N + \frac{\alpha}{2(\alpha - 1)}
  \left(\frac{4h\,M_{N,\textcolor{revisecolor}{\alpha}}}{\alpha-1}\right)^{\frac{1}{\alpha}}\,.
\end{equation}
We also introduce the sequences of numbers $a_k, b_k$, $0 \le k \le K$, satisfying the recursive
relations
\begin{align}
  \begin{split}
    a_k =& L_r + L_\phi e^{L_Fh}h + \textcolor{revisecolor}{\overline{M} e^{L_Fh}} + a_{k+1} e^{L_Fh}\,,\\
    b_k =& L_\phi C_{h,N} e^{L_Fh} (t_{k+1} - t_k) + \textcolor{revisecolor}{2\overline{M} C_{h,N}
  e^{L_F h}} + a_{k+1}
    C_{h,N} e^{L_Fh} + b_{k+1}\,,
  \end{split}
  \label{ak-bk}
\end{align}
for $0 \le k \le K - 1$ and $a_K = L_\psi$, $b_K = 0$, where $\overline{M}$ is
defined in (\ref{m-overline}). The last two expressions can be made more explicit :
\begin{align}
  \begin{split}
    a_k = & \Big(L_r + L_\phi e^{L_Fh}h + \textcolor{revisecolor}{\overline{M}
  e^{L_Fh}} \Big) \frac{e^{L_Fh(K-k)} -
  1}{e^{L_Fh} - 1} + L_\psi e^{(K-k)L_Fh}\,, \\
  b_k  = & C_{h,N} e^{L_Fh}\left\{L_\phi (T - t_k) +
  \left[\frac{L_r + L_\phi e^{L_Fh}h + \textcolor{revisecolor}{\overline{M}e^{L_Fh}}}{e^{L_Fh} - 1}
  \left(\frac{e^{L_Fh(K-k)} - 1}{e^{L_Fh} - 1} - (K-k)\right) \right.\right.\\
  & + \left.\left. L_\psi
\frac{e^{L_Fh(K-k)} - 1}{e^{L_Fh} - 1} \right] +
\textcolor{revisecolor}{2\overline{M} (K-k)} \right\}\,,
\end{split}
\label{ak-bk-explicit}
\end{align}
for $0 \le k \le K$.
Notice that under Assumptions~\ref{assump-1} and \ref{assump-2}, both $C_{h, N}$
and $b_k$ go to zero as $N\rightarrow\infty$.

\textcolor{revisecolor}{Similar to (\ref{omega-gamma-set}), we also introduce the set 
  \begin{align}
    \Omega^u_{\gamma,z, [t_i, t_{j}]} := \Big\{z' \in \mathbb{R}^n~\Big|~\inf_{t_i \le t \le t_j}|z'-\tilde{z}^u(t)| \le \gamma\Big\} \,,
    \label{omega-gamma-set-i-j}
  \end{align}
  between two control stages $t_i < t_j$ where $\tilde{z}^u(t_i) = z$, $u \in \mathcal{U}_{o, i}$.
  Especially, the notation $\Omega^\nu_{\gamma, z, [t_i, t_{i+1}]}$ will be
      used when only the control policy $\nu \in \mathcal{A}$ at the control stage
    $t_i$ is relevant.} We have the following approximation result of the value functions.
\begin{theorem}
  Suppose Assumptions~\ref{assump-1}-\ref{assump-4} hold.
Given $0 \le k \le K$ and $z \in \Omega$, s.t.
the ODE (\ref{limit-ode})  has a unique solution $\tilde{z}^u$ on $[t_k, T]$ 
for all $u \in \mathcal{U}_{o, k}$ and furthermore, $\exists \gamma > 0$, s.t.
    $\Omega^u_{\gamma,z, [t_k, T]} \subseteq \Omega$.
  for all $u \in \mathcal{U}_{o, k}$.  Let $z_N \in \mathbb{X}_N
  \textcolor{revisecolor}{\cap \Omega}$ be random with $\mathbf{E}|z_N - z| < \infty$.
  Then $\exists N_k>0$, s.t. 
  \begin{align}
    \mathbf{E} |U_N(z_N,k) - \widetilde{U}(z,k)| \le a_k \mathbf{E}|z_N - z| +
    b_k\,,
    \label{value-approx}
\end{align}
for $N>N_k$, with $a_k, b_k$ as given by (\ref{ak-bk}) or (\ref{ak-bk-explicit}).
 Further suppose that $u_0 \in \mathcal{U}_{o,0}$ is the optimal (open loop) control policy for the process
$\tilde{z}^{u}$, i.e.~$\widetilde{J}(z, u_0) = \widetilde{U}(z)$, and $z_N
\in \mathbb{X}_N \textcolor{revisecolor}{\cap \Omega}$ is deterministic satisfying $z_N \rightarrow z$ as $N \rightarrow \infty$. Then $\exists N_0 > 0$, s.t. when $N>N_0$, 
\begin{align}
  \begin{split}
  &|J_N(z_N, u_0) - U_N(z_N)| \\
  \le& b_0 + a_0 |z_N - z| + \left[L_\phi \frac{e^{L_F T} - 1}{L_F}
  +\big(KL_r + L_\psi + \textcolor{revisecolor}{\overline{M}}\big)e^{L_FT} \right]\,\Big(C_{T,N} + |z_N -z|\Big)\,.
\end{split}
  \label{asymp-optimality}
\end{align}
Especially, it holds that
\begin{align*}
  \lim_{N \rightarrow \infty} |J_N(z_N, u_0) - U_N(z_N)| = 0.
\end{align*}
\label{thm-4}
\end{theorem}

\begin{proof}
  We first prove (\ref{value-approx}) by backward induction from $k=K$ to $k=0$.
  Let $\mathbf{E}$ denote the expectation with respect to
  the random variable $z_N
  \in \mathbb{X}_N \textcolor{revisecolor}{\cap \Omega}$
  and recall $\mathbf{E}^\nu$ is the shorthand of the conditional expectation $\mathbf{E}^\nu_{t_{k},z_N}$.
  For $k=K$, since $z, z_N \in \Omega$, the terminal condition (\ref{terminal-value-fun}) and the Lipschitz continuity of the terminal cost $\psi$ in Assumption~\ref{assump-4} imply that
  \begin{align*}
    \mathbf{E}|U_N(z_N,K) - \widetilde{U}(z, K)| = \mathbf{E}|\psi(z_N)
    - \psi(z)| \le L_\psi \mathbf{E}|z_N-z|\,,
  \end{align*}
  therefore  (\ref{value-approx}) holds with $a_K = L_\psi$, $b_K = 0$ and for
  any $N_K>0$. 
  
  Now suppose (\ref{value-approx})
  is true for $k+1\le K$. First notice that we have the simple estimate $$|U_N(z_N,
  k) - \widetilde{U}(z,k)| \le 2\big[(K-k) M_r + (T-t_k) M_\phi +
  M_\psi\big]$$
  under Assumption~\ref{assump-4}.  Then, \textcolor{revisecolor}{fixing the constant
  $\rho=\frac{1}{3}\gamma e^{-L_Fh}$}
  and using the Bellman equation (\ref{optimality-eqn}) for the value
  function, we can estimate
\begin{align}
  & \mathbf{E}|U_N(z_N,k) - \widetilde{U}(z,k)| \notag \\
  = &
  \mathbf{E}\Big[|U_N(z_N,k) - \widetilde{U}(z,k)|\cdot \textcolor{revisecolor}{\mathbbm{1}_{\{|z_N-z|
  \le \rho\}}}\Big]
  + \mathbf{E}\Big[|U_N(z_N,k) - \widetilde{U}(z,k)|\cdot \textcolor{revisecolor}{\mathbbm{1}_{\{|z_N-z| >
  \rho\}}}\Big] \notag \\
  \le &
  \mathbf{E}\Big[|U_N(z_N,k) - \widetilde{U}(z,k)|\cdot \textcolor{revisecolor}{\mathbbm{1}_{\{|z_N-z|
  \le \rho\}}}\Big] \notag \\
  &  + \textcolor{revisecolor}{6\gamma^{-1} e^{L_Fh} \big[(K-k) M_r + (T-t_k) M_\phi +
M_\psi\big]\,\,\mathbf{E}|z_N-z|} \notag \\
  \le &
  \mathbf{E}\bigg[\bigg(\sup_{\nu \in \mathcal{A}} \Big\{\big|r(z_N,\nu) - r(z,\nu)\big|  
   +  \mathbf{E}^\nu\left|\int_{t_k}^{t_{k+1}} \Big(\phi(\tilde{z}^\nu(s),\nu)
   - \phi(z^{\nu,N}(s),\nu)\Big)\, ds\right| \notag \\
  & + \mathbf{E}^\nu\left| U_N\big(z^{\nu,N}(t_{k+1}),k+1\big) -
  \widetilde{U}\big(\tilde{z}^\nu(t_{k+1}), k+1\big) \right|\Big\}\bigg)\cdot
\textcolor{revisecolor}{\mathbbm{1}_{\{|z_N-z| \le \rho\}}}\bigg] \notag \\
& + \textcolor{revisecolor}{\overline{M} e^{L_Fh} \mathbf{E}|z_N-z|} \,,
  \label{thm-4-tmp-1}
 \end{align}
 where Chebyshev's inequality has been used and we recall that the constant
 $\overline{M}$ is defined in (\ref{m-overline}). 

 \textcolor{revisecolor}{In the following, let us consider a fixed $z_N \in \mathbb{X}_N$ such that $|z_N - z| \le \rho$. We consider the process $z^{\nu, N}(s)$ on $[t_k, t_{k+1}]$ with $z^{\nu,N}(t_k) = z_N$
 and, similar to (\ref{tau-u-n}), we define the stopping time $$\tau^{\nu}_N = \inf\limits_{s \ge t_{k}}
 \Big\{s\,\big|\,z^{\nu, N}(s) \not\in \Omega^\nu_{\gamma, z, [t_k, t_{k+1}]}\Big\}.$$
 For the notation, see the paragraph following (\ref{omega-gamma-set-i-j}). 
 Since $\Omega^u_{\gamma, z, [t_k, T]} \subseteq \Omega$ for $\forall u\in
 \mathcal{U}_{o,k}$ trivially implies
 $\Omega^\nu_{\gamma, z, [t_k, t_{k+1}]} \subseteq \Omega$,
 Theorem~\ref{thm-2} when considered on the time interval $[t_k, t_{k+1}]$
 guarantees that $\exists N' > 0$, s.t. when $N \ge N'$ we have 
 \begin{align}
   \begin{split}
   & \mathbf{E}^\nu\Big[\sup_{t_k \le s \le t \wedge \tau^{\nu}_N} \big|z^{u,
   N}(s) - \tilde{z}^u(s)\big|\Big] \le
\Big(\big|z_N - z\big| + C_{h,N} \Big)\,e^{L_F (t-t_k)}\,, \quad t
\in [t_k, t_{k+1}]\,, \\
   & \mathbf{P}(\tau^{\nu}_N < t_{k+1}) \le 3 \gamma^{-1} e^{L_Fh} C_{h,N}\,,
\end{split}
\label{tmp-4-tmp-2}
 \end{align}
 where the second inequality follows from (\ref{thm-2-result-3-improved}) in Remark~\ref{rmk-kurtz-1}.  }

 We continue to estimate each of the three terms within the supremum in (\ref{thm-4-tmp-1}).
 For the first term, noticing that $|z_N - z| \le \rho < \gamma$ implies $z_N \in
 \Omega$, and the function $r$ is Lipschitz in $\Omega$, 
 \begin{align*}
   |r(z_N, \nu)-r(z,\nu)| \le L_r|z_N-z|\,.
 \end{align*}
 For the second term, using a similar argument as in the proof of
 Theorem~\ref{thm-3} and the estimate (\ref{tmp-4-tmp-2}), we can obtain, 
 for $N > N'$,
 \begin{align*}
   &\mathbf{E}^\nu\Big|\int_{t_k}^{t_{k+1}} \big(\phi(\tilde{z}^\nu(s),\nu) -
   \phi(z^{\nu,N}(s),\nu)\big)\, ds\Big| \\
 \le 
 &\mathbf{E}^\nu\!\left[\left(\int_{t_k}^{t_{k+1}}
  \big|\phi(\tilde{z}^\nu(s),\nu) - \phi(z^{\nu,N}(s),\nu)\big|\,
ds\right)\textcolor{revisecolor}{\mathbbm{1}_{\{\tau^{\nu}_N \ge t_{k+1}\}}}\right] +
\textcolor{revisecolor}{2hM_\phi
\mathbf{P}(\tau^{\nu}_N < t_{k+1})}  \\
\le& L_\phi\Big(|z_N - z| + C_{h,N}\Big)e^{L_Fh} (t_{k+1} - t_k) +
\textcolor{revisecolor}{6h\gamma^{-1} M_\phi e^{L_Fh} C_{h,N}}\,. 
 \end{align*}
 For the third term, we notice the simple fact that $\Omega^{u}_{\gamma, z, [t_k, T]}
 \subseteq \Omega$ for $\forall u \in \mathcal{U}_{o, k}$ implies
 $\Omega^u_{\gamma, z', [t_{k+1}, T]} \subseteq \Omega$ for $\forall u \in \mathcal{U}_{o,
 k+1}$, where $z'=\tilde{z}^\nu(t_{k+1})$. And also that $\tau^{\nu}_N \ge t_{k+1}$
 implies $z^{\nu,N}(t_{k+1}) \in \Omega$. We have
 \begin{align*}
   &\mathbf{E}^\nu\Big| U_N(z^{\nu,N}(t_{k+1}),k+1) -
  \widetilde{U}(\tilde{z}^\nu(t_{k+1}), k+1) \Big| \\
\le 
& \mathbf{E}^\nu\Big[\big| U_N(z^{\nu,N}(t_{k+1}),k+1) -
  \widetilde{U}(\tilde{z}^\nu(t_{k+1}), k+1) \big|\,\cdot
\textcolor{revisecolor}{\mathbbm{1}_{\{\tau^{\nu}_N \ge t_{k+1}\}}}\Big]  \\
&+ \textcolor{revisecolor}{2\Big[(K-k-1) M_r + (T-t_{k+1}) M_\phi + M_\psi\Big]
\mathbf{P}(\tau^{\nu}_N < t_{k+1})}\\
  \le
& \mathbf{E}^\nu\Big[\big| U_N(z^{\nu,N}(t_{k+1}),k+1) -
\widetilde{U}(\tilde{z}^\nu(t_{k+1}), k+1)
\big|\,\Big|\,\textcolor{revisecolor}{\tau^{\nu}_N \ge t_{k+1}}\Big]
\textcolor{revisecolor}{\mathbf{P}(\tau^{\nu}_N \ge t_{k+1})} \\
&+ \textcolor{revisecolor}{6\gamma^{-1}\Big[(K-k-1) M_r + (T-t_{k+1}) M_\phi + M_\psi\Big]
e^{L_Fh}\,C_{h,N}} \\
  \le & a_{k+1} \mathbf{E}^\nu
\Big[|z^{\nu,N}(t_{k+1}) - \tilde{z}^\nu(t_{k+1})| \cdot
\textcolor{revisecolor}{\mathbbm{1}_{\{\tau^{\nu}_N \ge t_{k+1}\}}}\Big]
+ b_{k+1} + \textcolor{revisecolor}{\overline{M} e^{L_Fh}\,C_{h,N}} \\
\le& a_{k+1} \big(|z_N-z| + C_{h,N}\big) e^{L_Fh} + b_{k+1} + 
\textcolor{revisecolor}{\overline{M} e^{L_Fh}\,C_{h,N}}\,,
 \end{align*}
 for $N > \max\big\{N', N_{k+1}\big\}$.
 In the above, we have used the conclusion for $k+1$ to the conditional
 expectation $\mathbf{E}^\nu(\cdot\,|\,\tau^\nu_N \ge t_{k+1})$. 

 Substituting the above estimates into (\ref{thm-4-tmp-1}), we conclude
\begin{align*}
  & \mathbf{E} |U_N(z_N,k) - \widetilde{U}(z,k)| \\
  \le & \mathbf{E}\Big[L_r |z_N -
  z| + 
  L_\phi\big(|z_N - z| + C_{h,N}\big)e^{L_Fh} (t_{k+1} - t_k) +
  \textcolor{revisecolor}{6h\gamma^{-1} M_\phi e^{L_Fh} C_{h,N}}\,\\
  & + a_{k+1} \big(|z_N-z| + C_{h,N}\big) e^{L_Fh} + b_{k+1} + 
\textcolor{revisecolor}{\overline{M} e^{L_Fh}\,C_{h,N}}\Big] +
\textcolor{revisecolor}{e^{L_Fh} \overline{M}\,\mathbf{E}|z_N-z|}  \\
\le & \big(L_r + L_\phi e^{L_Fh}h+ a_{k+1} e^{L_Fh} +
\textcolor{revisecolor}{\overline{M} e^{L_Fh} \big)\,\mathbf{E}|z_N - z|} \\
& + L_\phi C_{h,N} e^{L_Fh} (t_{k+1} - t_k) + \textcolor{revisecolor}{2\overline{M}
C_{h,N} e^{L_Fh}} + a_{k+1} C_{h,N} e^{L_Fh} + b_{k+1} \\ = & a_k \mathbf{E}\,|z_N - z| + b_k\,,
\end{align*}
where the recursive relation (\ref{ak-bk}) has been used in the last equation.
This proves (\ref{value-approx}) for $k$ with
$N_k=\max\big\{N',N_{k+1}\big\}$. 

Equation (\ref{asymp-optimality}) now follows from (\ref{value-approx})   and Theorem~\ref{thm-3}, using the
triangle inequality: $\exists N_0 > 0$, s.t. $N>N_0$, we have 
\begin{align*}
  & |J_N(z_N, u_0) - U_N(z_N)|\\
  \le & |J_N(z_N, u_0) - \widetilde{J}(z, u_0)| +
  |\widetilde{U}(z) - U_N(z_N)| \\
  \le & b_0 + a_0|z_N - z| + \left(L_\phi \frac{e^{L_F T} - 1}{L_F} +\big(KL_r
  + L_\psi + \textcolor{revisecolor}{\overline{M}}\big) e^{L_FT}\right)\,\Big(C_{T,N} + |z_N - z|\Big) \,.
\end{align*}
Convergence $|J_N(z_N, u_0) - U_N(z_N)|\to 0$ as $N\to\infty$  readily
follows from Assumptions~\ref{assump-1} and \ref{assump-2}.
\end{proof}
\begin{remark}
  As discussed in Remark \ref{rmk-kurtz-density-depend}, we have $C_{T,N}
  =\mathcal{O}(N^{\frac{1}{\alpha} - 1})$ and thus $b_0 =
  \mathcal{O}(N^{\frac{1}{\alpha} - 1})$ for the density dependent Markov
  chain introduced in Subsection~\ref{sub-concrete}.  As a consequence, in
  this case we can explicitly compute the order of convergence in Theorems~\ref{thm-3} and
  \ref{thm-4}. That is, $\exists N_0 > 0$, s.t. when $N>N_0$,  
\begin{align*}
  &|J_N(z_N, u) - \widetilde{J}(z_0, u)|
  \le
  C N^{\frac{1}{\alpha} - 1} \,, \quad u \in \mathcal{U}_{o,0}\,,
\end{align*}
and
\begin{align*}
  |J_N(z_N, u_{0}) - U_N(z_N)| \le
  C N^{\frac{1}{\alpha} - 1} \,,
\end{align*}
with $C>0$ being a generic constant, $u_{0}$ being the optimal open loop policy for the limiting process
$\tilde{z}^u$, and $U_N$ being the value function of the stochastic feedback optimal control problem.
\end{remark}

\subsection{Feedback optimal control on infinite time-horizon with discounted cost}
\label{sub-discount}

As a final step of our analysis, 
we consider the discounted optimal control problem on an infinite time-horizon. 
\textcolor{revisecolor}{While the open loop control problem on a
  finite time horizon that is addressed in Theorem~\ref{thm-3} will be useful later
  on in Sections~\ref{sec-hybrid} and~\ref{sec-examples}, open loop control on an infinite time-horizon
for stochastic processes seems to be less relevant in applications. }
Therefore, in the following, we consider the \textcolor{revisecolor}{feedback} optimal control problem with cost functional
 \begin{align}
   J_N(z, u) = \mathbf{E}_{z}^u\!\left[\sum_{j=0}^{\infty} e^{-\beta t_j} \left(
       r(z^{u,N}(t_j), \nu_j) + \int_{t_j}^{t_{j+1}} \phi(z^{u,N}(s), \nu_j)\, ds
   \right)\right]\,,
  \label{cost-j-general-infty-discount}
\end{align}
where $\beta > 0$ is a discount factor, $u \in \mathcal{U}_f$ with  
\begin{align}
  \mathcal{U}_{f} =  \Big\{(\nu_0, \nu_{1}, \cdots) ~|~ \nu_j :
    \mathbb{X}_N \rightarrow \mathcal{A}\,, \quad 0\le j < \infty\Big\}\,,
\end{align}
and again the shorthand $\nu_j=\nu_j(z^{u,N}(t_j))$ has been used in (\ref{cost-j-general-infty-discount}).

We assume that the control set $\mathcal{A}$ is finite, \textcolor{revisecolor}{which guarantees the
  existence of the optimal control policy and will simplify the proof of
  Theorem~\ref{thm-5} (see below). We emphasize that this assumption is not
  essential and can be relaxed since we will only consider $\epsilon$-optimal
control policies in Theorem~\ref{thm-5}. Also see the related discussions in
Subsection~\ref{subsec-assumptions}}. Furthermore, we only focus on the
case when the time stages at which the controls can be changed are uniformly
distributed, i.e. $t_j = jh$ for some $h > 0$. \textcolor{revisecolor}{This
  uniformity in time allows us to define value functions which only depend on system's states and will simplify the discussions below. }

It is known (e.g. \cite{Puterman1994}) that the value function $U_N(z) = \inf\limits_{u \in \mathcal{U}_f} J_N(z,u)$
solves the Bellman equation
\begin{align}
  U_N(z) = \min_{\nu \in \mathcal{A}} \mathbf{E}_{z}^\nu\!\left[r(z,\nu) +
  \int_0^h \phi(z^{\nu, N}(s), \nu) ds + \lambda U_N(z^{\nu, N}(h))\right]\,,
  \label{optimality-discount-N}
\end{align}
where $\lambda = e^{-\beta h} < 1$. Moreover it is known \cite{Puterman1994}
that there is a map
$\pi_N \colon \mathbb{X}_N \rightarrow \mathcal{A},$ such that  $u_{opt}=(\pi_N, \pi_N,
\cdots) \in \mathcal{U}_f$ is an optimal feedback policy that satisfies $U_N(z) = J_N(z, u_{opt})$ and can be determined by the dynamic programming (i.e.~Bellman) equation via
\begin{align*}
  \pi_N(z) \in \argmin_{\nu \in \mathcal{A}} \left\{r(z,\nu) +
  \mathbf{E}^\nu_z\Big[ \int_0^h
\phi(z^{\nu,N}(s), \nu)\, ds + \lambda U_N(z^{\nu,N}(h))\Big]\right\}, \quad z \in
  \mathbb{X}_N\,.
\end{align*}

In correspondence with the stochastic control problem, we also consider the optimal
control of the deterministic limit dynamics $\tilde{z}^u(\cdot)$ which
satisfies ODE (\ref{limit-ode}), where $$u
\in \mathcal{U}_{o} :=  \Big\{\,(\nu_0, \nu_{1}, \cdots) ~\Big|~ \nu_j \in
\mathcal{A}\,,~0\le j < \infty\Big\}.$$ 
{
\color{revisecolor}In this context, it is necessary that the solution $\tilde{z}^u(\cdot)$ exists on $[0, +\infty)$. 
    Recalling the set defined in (\ref{omega-gamma-set}), in the following we
    consider the subset $\Omega_g \subseteq \Omega$ with the property that,
    \begin{enumerate}
      \item
	$z \in \Omega_g$ $ \Longrightarrow$ $\tilde{z}^u(t) \in \Omega_g$,
	$\forall 0 \le t < \infty$, $\forall u \in \mathcal{U}_o$.
      \item
    for all $T>0$, we can find $\gamma>0$, such that $\Omega^u_{\gamma, z, [0, T]}
    \subseteq \Omega$ holds for all $u \in \mathcal{U}_o$, $\forall z \in \Omega_g$.
\end{enumerate}

      We emphasize that this (nonempty) subset $\Omega_g$ can be easily
      constructed as long as $\Omega$ is large enough and it doesn't have to
      be unique. In fact, when the solution $\tilde{z}^u$ of the ODE (\ref{limit-ode}) starting from
      $\tilde{z}^u(0) = z$ exists on $[0, +\infty)$ and stays in $\Omega$
	for all time (without approaching its boundary) for any $u \in \mathcal{U}_o$, it is easy to see that the set
	$\Omega_g:=\{\tilde{z}^u(t)~|~ t\ge 0,\, u \in\mathcal{U}_o\}$
	satisfies the above two conditions. 
}

  The natural candidate for the deterministic cost functional reads
\begin{align}
  \widetilde{J}(z, u) = \sum_{j=0}^{\infty} e^{-\beta jh} \left(r(\tilde{z}^u(jh), \nu_j) + \int_{jh}^{(j+1)h}
  \phi(\tilde{z}^u(s), \nu_j)\, ds \right)\,,
  \label{J-discount-infty}
\end{align}
where $z \in \Omega_{g}$. 
  \textcolor{revisecolor}{Notice that again, following the convention in Subsection~\ref{sub-setup-mjp}, we use the same
notation $\nu_j$ to denote both the control policy function which depends on
system's state, and the value of the control selected at $t_j$. See the discussion after
(\ref{control-policy-set}).}

  By the dynamic programming principle, the
  corresponding value function $\widetilde{U}(z) = \inf\limits_{u\in
  \mathcal{U}_o} \widetilde{J}(z,u)$ satisfies
\begin{align}
  \widetilde{U}(z) = \min_{\nu \in \mathcal{A}} \left\{r(z,\nu) + \int_0^h
  \phi(\tilde{z}^\nu(s), \nu)\, ds + \lambda \widetilde{U}(\tilde{z}^{\nu}(h))\right\}\,,
  \label{optimality-discount-infty}
\end{align}
where $\tilde{z}^\nu(0) = z \in \Omega_g$.  We will assume that a map $\pi_\infty :
\textcolor{revisecolor}{\Omega_g} \rightarrow \mathcal{A}$ exists such that
\begin{align}
  \pi_\infty(z) \in \argmin_{\nu \in \mathcal{A}} \left\{r(z,\nu) + \int_0^h
  \phi(\tilde{z}^{\nu}(s), \nu)\, ds + \lambda
\widetilde{U}(\tilde{z}^{\nu}(h))\right\} \,,
  \label{optimality-condition-infty}
\end{align}
where $\tilde{z}^{\nu}(0) = \textcolor{revisecolor}{z \in \Omega_g}
$.

Assumption~\ref{assump-4} implies that
  \begin{align}
    \widetilde{J}(z,u)
    \le \sum_{j = 0}^{\infty} e^{-\beta jh} \left(M_r +
    \int_{jh}^{(j+1)h} M_\phi\, ds\right)
    = \frac{M_r + M_\phi h}{1 - e^{-\beta h}} =:
    M_J\,. \label{MJ}
  \end{align}
  Similarly, $J_N(z,u) \le M_J$ and therefore the same upper bound
  applies to $\widetilde{U}(z)$ and $U_N(z)$. 

  The next theorem provides the relations between the stochastic optimal control
  problem and the optimal control problem of the limiting ODE.
    \begin{theorem}
    Let \textcolor{revisecolor}{the nonempty subset $\Omega_g \subseteq \Omega$ be given}. 
    \begin{enumerate}
      \item
	Suppose that Assumptions~\ref{assump-3}-\ref{assump-4} hold.
	For every $\epsilon > 0$, there exists $C_\epsilon > 0$, such that
	\begin{align*}
	  \sup_{\substack{z,z' \in \Omega_g \\|z - z'| \le R}}
	  |\widetilde{U}(z) - \widetilde{U}(z')| \le C_\epsilon R + \epsilon
	  \,,\quad \forall R > 0\,.
      \end{align*}
    \item
      Suppose that Assumptions~\ref{assump-1}-\ref{assump-4} hold. Then for
      all $\epsilon > 0$, there exists $\delta > 0$ and $N' \in{\mathbb N}$,
      such that when $N\ge N'$, it holds
      \begin{align}
	|U_N(z_N) - \widetilde{U}(z)| \le \epsilon\,, 
	\label{thm-5-ineq-1}
      \end{align}
      for all $z_N \in \mathbb{X}_N \cap \Omega$ , $z \in \Omega_g$, and $|z_N-z|\le \delta$.
    \item
Suppose that Assumptions~\ref{assump-1}-\ref{assump-4} hold. Given $0 <
\epsilon' < \epsilon$, 
$z \in \Omega_g$, and 
      an $\epsilon'$-optimal open loop policy $u =(\nu_0, \nu_1, \cdots) \in
      \mathcal{U}_o$ of the limiting ODE system, which satisfies
      $$\widetilde{U}(z) \le \widetilde{J}(z, u) \le \widetilde{U}(z) +
      \epsilon'.$$
      There exist constants $N'\in{\mathbb N}$ and $\delta > 0$, depending on
      $\epsilon, \epsilon'$ and $z$, such that
      for $N > N'$,  we have
    \begin{align*}
      J_N(z_N, u) \le U_N(z_N) + \epsilon\,,
    \end{align*}
    for all $z_N \in \mathbb{X}_N \cap \Omega$ and $|z_N - z| \le \delta$.
    That is, $u$ is an $\epsilon$-optimal control policy for the feedback optimal control problem (\ref{cost-j-general-infty-discount}).
    \end{enumerate}
    \label{thm-5}
  \end{theorem}
  \begin{proof}
    \begin{enumerate}
      \item
	Consider two starting points $z, z' \in \Omega_g$ and let $\nu =
	\pi_\infty(z)$. 
	Let $\tilde{z}^\nu(s;z)$, $\tilde{z}^\nu(s;z')$ be the solutions of
	the ODE (\ref{integral-ode}) on the time interval $[0, h]$ starting from $z, z'$ at $s=0$, respectively. And notice that $z,z' \in \Omega_g$ implies both solutions stay in $\Omega$ all time.

	By the Lipschitz continuity of the cost functions in Assumption~\ref{assump-4}, and (\ref{optimality-discount-infty})--(\ref{optimality-condition-infty}), we have
 \begin{equation}\label{u-continuity-1-side}
   \begin{aligned}
   \widetilde{U}(z') - \widetilde{U}(z)  \le& r(z', \nu) + \int_0^h \phi(\tilde{z}^{\nu}(s; z'), \nu) \, ds + \lambda
    \widetilde{U}(\tilde{z}^{\nu}(h; z'))\\
    & - r(z,\nu) - \int_0^h \phi(\tilde{z}^{\nu}(s; z), \nu)\, ds - \lambda
 \widetilde{U}(\tilde{z}^{\nu}(h; z))\\
 \le& L_r |z - z'| + \int_0^h L_\phi\big|\tilde{z}^{\nu}(s; z') -
 \tilde{z}^{\nu}(s; z)\big| \, ds \\
 & + \lambda \big|\widetilde{U}(\tilde{z}^{\nu}(h; z)) -
 \widetilde{U}(\tilde{z}^{\nu}(h; z'))\big| \,.
  \end{aligned}
  \end{equation}
  Using Assumption~\ref{assump-3}, the standard ODE theory implies 
  \begin{equation}\label{ode-diff-init-value}
    \big|\tilde{z}^\nu(t;z) - \tilde{z}^\nu(t;z')\big| \le e^{L_F t} \big|z -
    z'\big|\,,\quad 0 \le t \le h\,.
  \end{equation}
  Now for all $R \ge 0$, we define the function
  \begin{align}
    G_1(R) = \sup_{\substack{z_1, z_2\in \Omega_g\,,\\ |z_1 - z_2| \le R}} \big|\widetilde{U}(z_1) - \widetilde{U}(z_2)\big|\,,
    \label{G-R}
\end{align}
and it follows from (\ref{MJ}) that $G_1(R) \le 2M_J$, $\forall R\ge 0$.
    Combining (\ref{u-continuity-1-side}) and (\ref{ode-diff-init-value}), we find
    \begin{align*}
      G_1(R) \le \left(L_r + L_\phi e^{L_F h}h \right) R + \lambda G_1\!\left(e^{L_F h} R\right)\,,
    \end{align*}
    which, upon iterating the above inequality $k$ times, leads to
    \begin{align}
      G_1(R) \le \Big(L_r + L_\phi e^{L_F h} h\Big) \frac{1 - \lambda^k e^{L_Fkh}}{1 -
      \lambda e^{L_Fh}} R + 2\lambda^k M_J\,.
      \label{G-R-bound}
    \end{align}
    The first conclusion follows by noticing that $\lambda < 1$.
  \item
    Given $\epsilon > 0$ and since $\lambda < 1$, we could first choose $k>0$ such that $2\lambda^{k}
    M_J \le \frac{\epsilon}{3}$.
    From the definition of the subset $\Omega_g$, 
    we know $\exists \gamma > 0$, s.t. $\Omega^u_{\gamma, z, [0, kh]}
    \subseteq \Omega$ is satisfied for all $z \in \Omega_g$ and $u \in
    \mathcal{U}_o$.
    Let the constant $0 < \delta < \frac{\gamma}{3} e^{-L_F h}$ and $z_N \in \mathbb{X}_N\textcolor{revisecolor}{\cap
    \Omega}$, such that $|z-z_N| \le \delta$. 
    \textcolor{revisecolor}{Given $\nu \in \mathcal{A}$, we consider the stopping time 
    \begin{align}
      \tau^{\nu}_N = \inf_{s \ge 0} \Big\{s\,\Big|\,
      |z^{\nu,N}(s)-\tilde{z}^\nu(s)| > 3\delta e^{L_F h}\Big\} \wedge h \,,
      \label{thm-5-stopping-time}
    \end{align}
  where $z^{\nu, N}(0) = z_N$, $\tilde{z}^\nu(0) = z$, respectively.}  In
  fact, under Assumptions~\ref{assump-1}-\ref{assump-3} and
  using the fact that $\Omega^{\nu}_{\gamma, z, [0,h]} \subseteq \Omega$ (see the discussion
  before Theorem~\ref{thm-4} on the notations), the same argument in
  Theorem~\ref{thm-2} on the time interval $[0, h]$
implies that $\exists N' > 0$, s.t. when $N \ge N'$, we have 
\textcolor{revisecolor}{
 \begin{align}
   \mathbf{P}(\tau^{\nu}_N < h) \le \delta^{-1} C_{h,N}\,,
\label{tmp-5-tmp-2}
 \end{align}
 }
 where the constant $C_{h,N}$ is defined in (\ref{c-h-n}). Also see (\ref{thm-2-result-3-improved}) in Remark~\ref{rmk-kurtz-1}.  

More generally, for $R \ge 0$, we define the function
    $$G_2(R) = \sup_{\substack{z' \in \mathbb{X}_N \cap \Omega, z \in
    \Omega_g\\ |z'-z|\le R}} |U_N(z') - \widetilde{U}(z)|\,,$$
    and notice that Assumption~\ref{assump-4} implies $|G_2(R)| \le 2M_J$, $\forall R\ge 0$.

    Letting $\nu = \pi_\infty(z) \in \mathcal{A}$, using the dynamic programming equations (\ref{optimality-discount-N}),
    (\ref{optimality-discount-infty}) and the estimate (\ref{tmp-5-tmp-2}), we can obtain
       \begin{align*}
	 & U_N(z_N) - \widetilde{U}(z) \\
	 \le & \mathbf{E}^\nu_{\textcolor{revisecolor}{z_N}}\!\left[\int_0^h \phi(z^{\nu, N}(s), \nu)\, ds + \lambda
U_N(z^{\nu, N}(h))\right] 
 - \int_0^h \phi(\tilde{z}^\nu(s), \nu)\, ds - \lambda \widetilde{U}(\tilde{z}^{\nu}(h)) \\
	 \le & \mathbf{E}^\nu_{\textcolor{revisecolor}{z_N}}\!\left[\Big(\int_0^h \big|\phi(z^{\nu, N}(s), \nu) 
 -  \phi(\tilde{z}^\nu(s), \nu)\big|\, ds\Big)
   \cdot
 \textcolor{revisecolor}{\mathbbm{1}_{\{\tau^{\nu}_N \ge h\}}} \right]\\
 & + \lambda 
 \mathbf{E}^\nu_{\textcolor{revisecolor}{z_N}}
 \Big[\big|U_N(z^{\nu, N}(h)) - \widetilde{U}(\tilde{z}^{\nu}(h))\big|
   \cdot
 \textcolor{revisecolor}{\mathbbm{1}_{\{\tau^{\nu}_N \ge h\}}} \Big]
  + 2\big(h M_\phi + \lambda M_J\big) \textcolor{revisecolor}{\mathbf{P}(\tau^{\nu}_N < h)} \\
   \le & \mathbf{E}^\nu_{\textcolor{revisecolor}{z_N}}\!\left[\Big(\int_0^h L_\phi\big|z^{\nu, N}(s) -
   \tilde{z}^\nu(s)\big|\, ds + \lambda \big| U_N(z^{\nu, N}(h)) -
   \widetilde{U}(\tilde{z}^{\nu}(h))\big|\Big)\cdot
 \textcolor{revisecolor}{\mathbbm{1}_{\{\tau^{\nu}_N \ge h\}}} \right] \\
 & + 2\big(h M_\phi + \lambda M_J\big) \textcolor{revisecolor}{\mathbf{P}(\tau^{\nu}_N < h)} \\
       \le & 3 L_\phi \delta e^{L_F h} h + \lambda G_2(3\delta e^{L_Fh}) +
 2\delta^{-1} \big(h M_\phi + \lambda M_J\big) C_{h,N}\,.
    \end{align*}
    In the above, we have used the facts that 
    \textcolor{revisecolor}{
    \begin{align*}
      & z \in \Omega_g~\Longrightarrow ~\tilde{z}^\nu(h) \in \Omega_g\,, \\
      & \tau^{\nu}_N \ge h ~\Longrightarrow ~\sup_{0 \le s \le h} |z^{\nu,N}(s) -
      \tilde{z}^\nu(s)| \le 3\delta e^{L_Fh} \le \gamma ~\Longrightarrow~
      z^{\nu,N}(h) \in \Omega^\nu_{\gamma, z, [0,h]} \subseteq \Omega\,.
    \end{align*}
  }
    Since the same upper bound holds for $\widetilde{U}(z) - U_N(z_N)$ as well, taking the supremum over $z_N \in \mathbb{X}_N\cap \Omega$, $z \in
    \Omega_g$, such that $|z_N-z| < \delta$, we obtain
    \begin{align*}
      G_2(\delta) \le 3 L_\phi \delta e^{L_F h}h  + \lambda
      G_2\big(3\delta e^{L_Fh}\big) + 2\delta^{-1} (h M_\phi + \lambda M_J) C_{h,N}\,,
    \end{align*}
    as long as $\delta \le \frac{\gamma}{3} e^{-L_F h}$, $N>N'$. 

    \textcolor{revisecolor}{Notice that 
    $\Omega^u_{\gamma, z, [0, kh]} \subseteq \Omega$ implies $\Omega^u_{\gamma,
    z', [ih, kh]} \subseteq \Omega$ for the same $\gamma>0$, where $z' =
  \tilde{z}^u(ih)$, $0 < i < k$, $\forall u \in \mathcal{U}_{o}$.}
    Therefore, iterating the above inequality for $k$ times and using the inequality $G_2 \le 2M_J$,  it gives 
    \begin{align*}
      G_2(\delta) \le&  3 L_\phi \delta e^{L_F h}h 
      \frac{(3\lambda e^{L_Fh})^k-1}{3\lambda e^{L_Fh}-1} + 2\lambda^k M_J \\
      & + 2\delta^{-1} \big(h M_\phi + \lambda M_J\big) 
      C_{h,N}\frac{3^{-k}\lambda^k e^{-kL_Fh}-1}{3^{-1}\lambda e^{-L_Fh}-1} \\
      \le & 3 L_\phi \delta e^{L_F h}h 
      \frac{(3\lambda e^{L_Fh})^k-1}{3\lambda e^{L_Fh}-1} + \frac{\epsilon}{3} \\
      & + 2\delta^{-1} \big(h M_\phi + \lambda M_J\big) 
      C_{h,N}\frac{3^{-k}\lambda^k e^{-kL_Fh}-1}{3^{-1}\lambda e^{-L_Fh}-1} \\
    \end{align*}
    for $\delta \le 3^{-k} e^{-kL_F h} \gamma$, $N>N'$. 

    Since Assumptions $\ref{assump-1}-\ref{assump-2}$ imply that $C_{h,N} \rightarrow 0$ as $N\rightarrow \infty$, 
    we can first choose $\delta$ and then $N'$ such that $G_2(\delta) \le \epsilon$ when $N \ge N'$. The conclusion follows readily.
  \item
    We estimate the cost using the definition (\ref{cost-j-general-infty-discount}).
    Notice that the constant $\lambda = e^{-\beta h}<1$ and that the open loop
    control $u$ is $\epsilon'$-optimal for the deterministic optimal control problem (\ref{J-discount-infty}). For any $k \ge 1$, recalling the
    stopping time in (\ref{tau-u-n}) and Assumption~\ref{assump-4},  we obtain
    \begin{align*}
      & J_N(z_N, u) \\
      \le&
      \mathbf{E}_{z_N}^u\!\left[\sum_{j=0}^{k} \lambda^j \left( r(z^{u,N}(t_j), \nu_j) + \int_{t_j}^{t_{j+1}} \phi(z^{u,N}(s), \nu_j)\, ds \right)\right] 
       + \sum_{j=k+1}^\infty \lambda^j (M_r + h M_{\phi})  \\
   \le& \sum_{j=0}^{k} \lambda^j \left(r(\tilde{z}^{u}(t_j), \nu_j) + \int_{t_j}^{t_{j+1}} \phi(\tilde{z}^{u}(s), \nu_j)\, ds \right)
       + \sum_{j=k+1}^\infty \lambda^j (M_r + h M_{\phi}) \\
      &+  \mathbf{E}_{z_N}^u\bigg[\sum_{j=0}^{k} \lambda^j \big|r(z^{u,N}(t_j), \nu_j) - r(\tilde{z}^{u}(t_j), \nu_j)\big|\bigg]\\
      &+  \mathbf{E}_{z_N}^u\bigg[\sum_{j=0}^{k} \lambda^j \int_{t_j}^{t_{j+1}} \big|\phi(z^{u,N}(s), \nu_j) - \phi(\tilde{z}^{u}(s), \nu_j)\big|ds\bigg] \\
   \le& \sum_{j=0}^{\infty} \lambda^j \left(r(\tilde{z}^{u}(t_j), \nu_j) + \int_{t_j}^{t_{j+1}} \phi(\tilde{z}^{u}(s), \nu_j)\, ds \right)
       + 2\sum_{j=k+1}^\infty \lambda^j (M_r + h M_{\phi}) \\
      &+  \mathbf{E}_{z_N}^u\bigg[\Big(\sum_{j=0}^{k} \lambda^j
     \big|r(z^{u,N}(t_j), \nu_j) - r(\tilde{z}^{u}(t_j),
     \nu_j)\big|\Big)\textcolor{revisecolor}{\cdot \mathbbm{1}_{\{\tau^{u}_N \ge kh\}}}\bigg]\\
     &+  \mathbf{E}_{z_N}^u\bigg[\Big(\sum_{j=0}^{k} \lambda^j \int_{t_j}^{t_{j+1}} \big|\phi(z^{u,N}(s), \nu_j) - \phi(\tilde{z}^{u}(s), \nu_j)\big|ds
     \Big)\textcolor{revisecolor}{\cdot \mathbbm{1}_{\{\tau^{u}_N \ge kh\}}}\bigg] \\
     & + 2\big(M_r+hM_{\phi}\big) \textcolor{revisecolor}{\mathbf{P}(\tau^{u}_N < kh)} \sum_{j=0}^{k} \lambda^{j} \\
     \le   & \mathbf{E}_{z_N}^u\left[\Big(\sup_{0 \le s \le kh} |z^{u,N}(s) -
      \tilde{z}^u(s)|\Big)\cdot \textcolor{revisecolor}{\mathbbm{1}_{\{\tau^{u}_N \ge kh\}}}\right] (L_r + h L_\phi) \sum_{j=0}^{k} \lambda^j \\
      & + \widetilde{J}(z, u) +
      2\frac{M_r+hM_{\phi}}{1-\lambda}\Big(\lambda^{k+1} + \textcolor{revisecolor}{\mathbf{P}(\tau^{u}_N < kh)}\Big) \\
   \le& \widetilde{U}(z) + \epsilon' + \mathbf{E}_{z_N}^u\left[\Big(\sup_{0 \le
   s \le kh} |z^{u,N}(s) - \tilde{z}^u(s)|\Big)\cdot
 \textcolor{revisecolor}{\mathbbm{1}_{\{\tau^{u}_N \ge kh\}}}\right] 
   \frac{L_r + h L_\phi}{1-\lambda} \\
   & + 2M_J \bigg(\lambda^{k+1}+ \textcolor{revisecolor}{\mathbf{P}(\tau^{u}_N < kh)}\bigg)\\ 
   \le& U_N(z_N) + |U_N(z_N) - \widetilde{U}(z)|  + 2 M_J \bigg(\lambda^{k+1}+
   \textcolor{revisecolor}{\mathbf{P}(\tau^{u}_N < kh)}\bigg) + \epsilon' \\
   & +\mathbf{E}_{z_N}^u\left[\Big(\sup_{0 \le s \le kh} |z^{u,N}(s) -
  \tilde{z}^u(s)|\Big)\cdot \textcolor{revisecolor}{\mathbbm{1}_{\{\tau^{u}_N \ge kh\}}}\right] 
   \frac{L_r + h L_\phi}{1-\lambda}\,. 
 \end{align*}
 Now for $\epsilon > \epsilon'$, we can first choose $k > 0$ and then
 obtain $\gamma>0$ using the property of the subset $\Omega_g$ with $T=kh$. Applying 
 Theorem~\ref{thm-2} on the time interval $[0, kh]$, (\ref{thm-5-ineq-1}) and
 Assumptions~\ref{assump-1}-\ref{assump-2}, we can find $N' \in \mathbb{N}$
 and $\delta>0$,
 such that 
 \begin{align*}
   J_N(z_N,u) \le U_N(z_N) + \epsilon \,,
 \end{align*}
 if $z \in \Omega_g$, $z_N \in \mathbb{X}_N \cap \Omega$, and $|z-z_N|<
 \delta$.
The conclusion follows immediately.
  \end{enumerate}
  \end{proof}

\section{Algorithms}
\label{sec-hybrid}
In this section, we discuss some numerical aspects of the
control problems studied in this paper.
\textcolor{revisecolor}{The main motivation is that, although our previous analysis suggested that the
optimal open loop control of the limiting ODE system is a reasonable approximation
whenever $N$ is sufficiently large, in applications it is often difficult to
verify how large $N$ should be such that the approximation is satisfactory.
On the other hand, the optimal feedback control becomes increasingly difficult
to compute due to the rapid growth of the state space when $N$ is large. The main purpose of this
section is to construct an algorithm which further improves the
optimal open loop policy by utilizing the information of the system state
(i.e.~by adding feedback), while avoiding the curse of dimensionality that is
inherent to  the dynamic programming approach.}

In contrast to the previous sections, this part involves some heuristics, and
we confine ourselves to the optimal control problem for a Markov jump
process on a finite time-horizon $[0, T]$ with a finite control set $\mathcal{A}$. To this end, we assume that the parameter $N$ is large, and we remind the reader again that $x^{u,N}$
denotes the original Markov jump process with a control policy $u$ and $z^{u,N} =
N^{-1 }x^{u, N}$ stands for the normalized density process. The state spaces on which $x^{u,N}$ and $z^{u,N}$ live are denoted by
$\mathbb{X}$ and $\mathbb{X}_N$, respectively.

\subsection{Tau-leaping method}
\label{sub-tau-leaping}
In order to compute the optimal control policy, it is necessary to simulate
trajectories of the underlying Markov jump process and to estimate the corresponding cost.
The stochastic simulation algorithm (SSA) \cite{Gillespie1976_ssa, Gillespie1977_ssa, ssck_gillespie} is a typical
Monte Carlo method: At each time step, it determines the waiting
time in (\ref{waiting-time}) as well as the next state according to the jump rates between
the current state and the next possible states. When $N$ is large, however, the system becomes numerically stiff
 because a large number of jump events occur within a short time interval.
Since SSA traces every single jump event of the system, the effective step size of the method decreases rapidly, which renders the SSA  inefficient.

As a remedy to this problem, the tau-leaping method \cite{gillespie2001, stiffness_implicit_tau,tau_size_selection,randon_correction_tiejun,ssck_gillespie} aims at increasing the effective step size by updating the state vector according to the transitions that may occur within a given time interval. Roughly speaking, instead of computing
the waiting time and the next jump, the idea of the tau-leaping method
is to answer ``how many times will each type of jumps occur within a given time
interval'' and then update the state vector accordingly. With a proper and carefully chosen step size \cite{tau_size_selection}, the tau-leaping method can
approximate the SSA quite well and meanwhile reduce the simulation time up to
$1$ or $2$ orders of magnitude. In our implementation (see the numerical examples in Section~\ref{sec-examples}), we use the explicit
tau-leaping method where the leaping time step sizes are determined according to \cite{tau_size_selection}.

\subsection{State space truncation}
\label{sub-feedback}

The computational complexity for solving the feedback optimal control problem
is proportional to the number of states in $\mathbb{X}$ considered (which
is of order $N^n$, with $n$ being the number of species).
Therefore, truncating the state space $\mathbb{X}$
is necessary before numerically solving the optimal feedback control.
One such approach to truncate the state space is to consider only states
$x=(x^{(1)}, x^{(2)}, \cdots, x^{(n)}) \in
\mathbb{X}$ that lie within a hypercube defined by $x^{(i)} \in [c_{i} N,
c'_{i}N]$, $1 \le i \le n$, where $0 \le c_{i} < c_{i}'$ are estimations of
the lower and upper bounds of the average densities per species. The cut-off values $c_{i},\, c_{i}'$ could, e.g., be determined by launching independent simulations of the jump process controlled by candidate open loop control policies.

Once a truncated state space $\mathbb{X}_{cut}$ has been constructed, then a simple algorithm (Algorithm~\ref{algo-feedback}) to
compute the optimal feedback control policy can be based on the necessary optimality
condition (\ref{optimality-eqn}) with the terminal condition
$U_N(\cdot,K) = \psi$ where the expectation value in (\ref{optimality-eqn}) is estimated by a Monte Carlo average. If $T$ is the total simulation time, $\Delta t>0$ is the average time step size used to generate trajectories (e.g.~by SSA or tau-leaping) and we
use $M$ independent realizations for each starting state to approximate the expectation value, the overall computational cost of
Algorithm~\ref{algo-feedback} is $\mathcal{O}\big(M\cdot |\mathcal{A}|
\cdot|\mathbb{X}_{cut}|\cdot\lceil T/\Delta t \rceil\big)$.
\begin{algorithm}[h]
  \caption{Compute the optimal feedback control policy on truncated state
  space\label{algo-feedback}}
  \begin{algorithmic}[1]
    \State
    Set $U_N(\cdot, K) = \psi$.
    \For{$k \gets K-1 \textrm{ to } 0$}
    \For{each $x \in \mathbb{X}_{cut}$}
    \For{each $\nu \in \mathcal{A}$}
    \State
    Starting from $x$ at time $t_k$, generate $M$ trajectories $x^{\nu, N}_i$
    till time $t_{k+1}$, such that $x^{\nu,N}_i(t_{k+1}) \in \mathbb{X}_{cut}$
    (generate new realization if $x^{\nu,N}_i(t_{k+1}) \notin\mathbb{X}_{cut}$).
    \State
    Let $z = x/N$, $z^{\nu,N}_i = x^{\nu,N}_i/N$, compute
    \begin{align*}
      \textcolor{revisecolor}{Q}(\nu) = \frac{1}{M} \sum_{i=1}^M
      \left(r(z, \nu) + \int_{t_k}^{t_{k+1}} \phi(z^{\nu, N}_i(s), \nu) \,ds +
      U_N(z^{\nu,N}_i(t_{k+1}), k+1)\right).
    \end{align*}
    \EndFor
\State
Set $\nu_k(z) = \argmin\limits_{\nu \in \mathcal{A}} \textcolor{revisecolor}{Q}(\nu)$ and $U_N(z,k)
= \min\limits_{\nu \in \mathcal{A}} \textcolor{revisecolor}{Q}(\nu)$.
    \EndFor
    \EndFor
  \end{algorithmic}
\end{algorithm}

\subsection{Hybrid control}
\label{sub-hybrid}

Solving the feedback control problem may be computationally infeasible even
after truncation of the state space. As already mentioned at the beginning of
this section, we will utilize an adaptive state space truncation strategy
which exploits information from the (optimal) open loop control policies.
The key idea is to assume that the typical states visited
by the jump process when an optimal open loop policy is applied are also
important states for computing a sufficiently accurate feedback control policy. To
this end, the following algorithm generates states (for each control stage)
whose densities are
scattered around the density values of the system controlled by reasonable open loop control policies.

\subsubsection*{Adaptive truncation strategy}

Let $\mathcal{S}_j \subset \mathbb{X}$
denote the finite state set at the $j$-th control stage after truncation, $0 \le j
< K$. We construct sets $\mathcal{S}_j$ using the following steps.

\begin{enumerate}
  \item
    \ul{Compute ``good'' (open loop) candidate policies for the Markov jump
      process on time {\color{revisecolor}$[0, T]$}.}\label{goodpolicies}
    A control policy $u_k \in \mathcal{U}_{o, 0}$ is called ``good'' if $k <
    n_{ol}$ and $J_N(u_k) \le (1 + \epsilon_{ol}) J_N(u_0)$
    for appropriately chosen $n_{ol} \in \mathbb{N}$, $n_{ol}\ge1$ and
    $\epsilon_{ol} \ge 0$ (especially, $u_0$ is the optimal open loop control
    policy for the jump process).
    \textcolor{revisecolor}{Sort all ``good'' control policies $u_k \in \mathcal{U}_{o,0}$
    by their costs in non-decreasing order.}
  \item
\ul{Compute statistics of the controlled jump processes under ``good'' policies.}
For each ``good'' open loop policy $u_k$, record the average densities $z_{k,j} \in
\mathbb{R}^n$ and the standard deviations $\sigma_{k,j} \in \mathbb{R}^n$ of
the controlled normalized density process at each stage $j$, $0 \le j < K$.
\item
\ul{Compute the truncated sets $\mathcal{S}_j$.}
For each ``good'' open loop policy $u_k$, generate $M_{ol}$ trajectories and
add the states $x \in \mathbb{X}$ of each trajectory at stage $j$ to the set $\mathcal{S}_j$ if
\begin{align}
  x^{(i)}/N
  \in \big[z^{(i)}_{k,j} - \zeta\sigma^{(i)}_{k,j}, z^{(i)}_{k,j} + \zeta\sigma^{(i)}_{k,j}\big]\,, \quad \forall i\in\{1, \ldots,n\}
\label{closeness}
\end{align}
where $\zeta>0$ is a pre-selected constant, and $x^{(i)}, z^{(i)}_{k,j},
\sigma^{(i)}_{k,j}$ are the
$i$th components of $x$, $z_{k,j}$, $\sigma_{k,j} \in \mathbb{R}^n$ .
\end{enumerate}

\begin{remark}
A few remarks about the above algorithm are in order.
  \begin{enumerate}
    \item
      In the case that the jump process starts from a fixed initial value $x$,
      $\mathcal{S}_0=\{x\}$ is a singleton containing only the initial state.
    \item
      Step $1$ can be accomplished by enumerating all possible (finite) $u_k \in
\mathcal{U}_{o,0}$ and computing the cost $J(u_k)$ by simulating trajectories
using SSA or the tau-leaping method. \textcolor{revisecolor}{Parameters $n_{ol}$ and $\epsilon_{ol}$
are introduced in order to determine the number of ``good'' open loop policies which might carry
important information and will be used to construct the truncated state sets $\mathcal{S}_j$ in
Steps 2, 3 above}. By the central limit theorem (see
\cite{kurtz_limit_thm_diffusion_approx}), the state distributions of
the jump process under ``good'' open loop policies is approximately Gaussian whenever $N$ is
large, hence the standard statistical estimators for the means and standard
deviations computed in Step~$2$ can capture the distributions to a good approximation.
\item
Ideally, for every ``good'' control policy $u_k$ and every control
stage $j$, we would like to record all possible (i.e.~reachable) discrete states that satisfy
(\ref{closeness}).
However, this set may be very large. Therefore, we sample these reachable
states in Step $3$ with a tunable parameter $M_{ol}$, which can control the number of states in $\mathcal{S}_j$.
The drawback is that important states may be missing when they are
not visited by the $M_{ol}$ trajectories (see below for a patch).
  \end{enumerate}
  \label{rmk-hybrid}
\end{remark}

\subsubsection*{Hybrid control policy}
Having the state sets $\mathcal{S}_j$ at hand, the task of computing a feedback control policy is
to determine maps $\nu_j : \mathcal{S}_j \rightarrow \mathcal{A}$, $0 \le j <
K$, according to a modification of Algorithm~\ref{algo-feedback}.
Keeping in mind that the sets $\mathcal{S}_j$ may be only partially sampled, it is quite possible that, at some control stage
$j$, the system fails to reach $\mathcal{S}_j$ under control $\nu_{j-1}$. To remedy this defect, we propose the following strategy :
Denote the best available open loop policy as $u_0 = (\nu^0_0,
\nu^0_1, \cdots, \nu^0_{K-1})$, and consider the $j$-th control stage,
$0 \le j < K$ where we suppose that the system has ended up in a state $x
\notin\mathcal{S}_j$. Further let $x'$ be one of the nearest states to $x$ among all states in $\mathcal{S}_j$, i.e.~$x' \in \argmin_{x' \in \mathcal{S}_j} |x-x'|$. Then we apply the control $\nu_j(x')$ if $|x-x'|/N \le
\epsilon_{near}$, where $\epsilon_{near}$ is a cut-off parameter, and otherwise we use $\nu_j^0$. In other words, we replace the original candidate control by the modified control policy
$u=(\bar{\nu}_0, \bar{\nu}_1, \cdots, \bar{\nu}_{K-1}) \in \mathcal{U}_{f,0}$ with
\begin{align}
 \bar{\nu}_j \colon \mathbb{X} \rightarrow \mathcal{A}\,,\quad  \bar{\nu}_j(x) =
  \begin{cases}
    \nu_j(x')\,, & \mbox{if}~|x' - x|/N < \epsilon_{near} \\
      \nu_j^0\,, & otherwise\,.
    \end{cases}
    \label{def-hybrid-map}
\end{align}
In the following, we keep using $\nu_j$ instead of $\bar{\nu}_j$
when no ambiguity exists. This strategy can prevent problems that arise
when the feedback policy $\nu_{j}$ at stage $j$ cannot be computed because some rare, but important states are missing due
to the insufficient sampling when constructing the set $\mathcal{S}_j$. Notice that the algorithmic modification can be easily switched off by setting $\epsilon_{near} =
0$. In this case, the feedback policy is applied only when the states belong to $\mathcal{S}_j$, while open loop policies are applied otherwise.
In agreement with the notation used in Sections~\ref{sec-intro}--\ref{sec-analysis-ocp}, we define
\begin{align}
  \mathcal{U}_{h,k} =& \{(\bar{\nu}_k, \bar{\nu}_{k+1}, \cdots, \bar{\nu}_{K-1}) ~|~ \nu_j :
    \mathcal{S}_j \rightarrow \mathcal{A}\,,\, k\le j < K\}\,,\quad 0 \le k < K\,,
\end{align}
as the set of all hybrid control policies, where $\bar{\nu}_j$ is defined as
in (\ref{def-hybrid-map}). The algorithmic task now boils down to finding the
optimal hybrid control policy $u \in \mathcal{U}_{h, 0}$. In order to solve this task, we consider
the cost function $J_N(z,u, k)$ as in (\ref{fb-j-u-finite}) and define a modified value function as
\begin{align}
U_N(z,k) =& \inf_{u \in \mathcal{U}_{h,k}} J_N(z,u, k)\,, \quad Nz \in
  \mathcal{S}_k\,.
\end{align}

By definition, the value function satisfies the terminal condition $U_N(z,K) = \psi(z)$ and a modified Bellman equation as a necessary optimality condition :
\begin{align}
  \begin{split}
    U_N(z,k) = &\min_{\nu \in \mathcal{A}} \mathbf{E}^\nu\!\left[
  \sum_{j=k}^{\tau-1} \Big(r\big(z^{u,N}(t_j), \nu_j(z^{u,N}(t_j))\big)\right. \\
  &+ \left.\int_{t_{j}}^{t_{j+1}} \phi\big(z^{u,N}(s), \nu_j(z^{u,N}(t_j))\big)\,
  ds\Big) +
U_N\big(z^{u,N}(t_\tau), \tau\big) \right],\quad Nz \in \mathcal{S}_k\,.
\end{split}
\label{hybrid-optimality}
\end{align}
where $z^{u, N}(t_k) = z$, $u=(\nu_k, \nu_{k+1}, \cdots, \nu_{K-1})$ with $\nu_k=\nu$ and $(\nu_{k+1}, \cdots,
\nu_{K-1}) \in \mathcal{U}_{h, k+1}$ is the optimal hybrid control policy
starting from stage $k+1$. The terminal index $\tau$ is a stopping time, depending on the particular realization, and is either the
smallest stage index such that $k < \tau < K$ and $Nz^{u,N}(t_{\tau}) \in
\mathcal{S}_{\tau}$, or $\tau=K$ otherwise.
Notice that in (\ref{hybrid-optimality}), only values of $U_N(z, k)$ at states
$z$ such that $Nz \in \mathcal{S}_k$ are involved. Based on it, we can compute the optimal hybrid
control policy by backward iterations in Algorithm~\ref{algo-hybrid} below.
\begin{algorithm}[H]
  \caption{Compute the optimal hybrid control policy\label{algo-hybrid}}
  \begin{algorithmic}[1]
    \State
    Set $U_N(\cdot, K) = \psi$.
    \For{$k \gets K-1 \textrm{ to } 0$}
    \For{each $x \in \mathcal{S}_{k}$}
    \For{each $\nu \in \mathcal{A}$}
    \State
    Set $u=(\nu, \nu_{k+1}, \cdots, \nu_{K-1})$, where $\nu_j$ is the optimal
    policy function on the $j$-th stage, $k< j< K-1$ (already computed).
    \State
    Generate $M$ trajectories $x^{u, N}_i$ from time $t_k$ to $t_{\tau_i}$ where
    $k<\tau_i$ and $t_{\tau_i}$ is either the first time when $x^{u,
    N}_i(t_{\tau_i}) \in \mathcal{S}_{\tau_i}$ or $\tau_i=K$, $1 \le i \le M$.
    \State
    Let $z = x / N$, $z^{u,N}_i=x^{u,N}_i/N$, compute
    \begin{align*}
      \textcolor{revisecolor}{Q}(\nu) = &\frac{1}{M} \sum_{i=1}^M \Big\{\sum_{j=k}^{\tau_i - 1}
	\Big[r(z^{u,N}_i, \nu_j(z^{u,N}_i(t_j))) + \int_{t_j}^{t_{j+1}}
	  \phi\big(z^{u,
	N}_i(s), \nu_j(z^{u,N}_i(t_j))\big) \,ds\Big] \\
      &+ U_N(z^{u,N}_i(t_{\tau_i}), \tau_i)\Big\}.
    \end{align*}
    \EndFor
\State
Set $\nu_k(z) = \argmin\limits_{\nu \in \mathcal{A}} \textcolor{revisecolor}{Q}(\nu)$ and $U_N(z,k)
= \min\limits_{\nu \in \mathcal{A}} \textcolor{revisecolor}{Q}(\nu)$.
    \EndFor
    \EndFor
  \end{algorithmic}
\end{algorithm}

A computational bottleneck in computing the hybrid control policy for
$\epsilon_{near} > 0$ is the solution of the minimization problem $\argmin_{x'
\in \mathcal{S}_j} |x-x'|$, i.e. to find the nearest neighbor of $x$ in
$\mathcal{S}_j$. The computational complexity of a direct minimization based on a pairwise comparison is $\mathcal{O}\big(|\mathcal{S}_j|\big)$, which would increase the computational
cost of Algorithm~\ref{algo-hybrid} to
$\mathcal{O}\big(M \cdot |\mathcal{A}|\cdot |\mathcal{S}_j|^2 \cdot \lceil
T/\Delta t \rceil \big)$ (assuming
$\tau=k+1$ and $|\mathcal{S}_j|$ are constant for simplicity). However, by
employing the so-called $k$-$d$ tree data structure \cite{Bentley1975} to store the states in $\mathcal{S}_j$, the computational complexity of
finding the nearest neighbor can be reduced to $\mathcal{O}\big(\ln |\mathcal{S}_j|\big)$,
by which the total computational cost is of the order
\[
\mathcal{O}\big(M\cdot|\mathcal{A}|\cdot |\mathcal{S}_j| \cdot \ln
|\mathcal{S}_j| \cdot \lceil T /\Delta t \rceil\big)\,.
\]
In the numerical examples in Section~\ref{sec-examples} below, our implementation uses the ANN (Approximate Nearest Neighbor) library \cite{ANNmanual}, which provides operations
on $k$-$d$ trees and efficient algorithms for finding
the first $k$-th ($k=1$ in our case) nearest neighbors.

\section{Numerical examples}
In this section, we consider two numerical examples in order to
demonstrate the analysis and the algorithms discussed in the previous
sections.
\label{sec-examples}
\subsection{Birth-death process}
First, we consider the one-dimensional birth-death process which can be described as
  \begin{align}
    x-1 \xleftarrow{\kappa_-} x \xrightarrow{\kappa_+} x+1\,,
    \label{birth-death-process}
\end{align}
where $x \in \mathbb{N}^+$. We suppose that the process has a density dependent birth rate
which is $x\cdot\kappa_+ $ when the current state is $x$ and, similarly, $x \cdot \kappa_-$ for the
death rate.
We fix $T=3.0$ and $K=3$, i.e. the control can be switched at time $t=0.0,
1.0, 2.0$. Two control/parameterization sets $\mathcal{A}_1$, $\mathcal{A}_2$ shown in
Table~\ref{tab-birth-death-policy} are considered. Each set contains two controls
$\nu^{(0)}$, $\nu^{(1)}$ that affect the jump rates $\kappa_-$ and $\kappa_+$. For the optimal control problem,
let $x^{u,N}(t)$ be system's state at $t \in [0, T]$ with control $u \in
\mathcal{U}_{\sigma, 0}$
and set $r(z, \nu) = \psi(z) = 0$, $\phi(z, \nu) = |z-1.0|$ for $\nu \in
\mathcal{A}_i$, $i=1,2$, leading to the cost function
\begin{align}
  J_N(z_0, u) = \mathbf{E}_{z_0}^u\!\left[\int_0^{3} |z^{u,N}(t) - 1.0|\, dt\right]  \,,\quad
  u \in \mathcal{U}_{\sigma, 0}\,,
  \label{birth-death-cost-J}
\end{align}
with $z^{u,N}(t) = x^{u,N}(t)/N$ and $z^{u,N}(0) = z_0$. \textcolor{revisecolor}{Minimizing a cost as in
(\ref{birth-death-cost-J}) may arise when one wishes to keep the density
of the system (\ref{birth-death-process}) not far away from $1.0$ by
controlling the jump rates (depending on system's states)}. Fixing $z_0 = 1.2$ and one
of two control sets $\mathcal{A}_1$, $\mathcal{A}_2$, we shall compare the optimal open loop and feedback control policies
for the jump process as $N$ increases, as well as the optimal (open loop) control policy for the
related deterministic ODE
\begin{align}
  \frac{d\tilde{z}^u(t)}{dt} = (\kappa_+ - \kappa_-) \tilde{z}^u(t), \qquad
  \tilde{z}^u(0) = 1.2\,.
  \label{birth-death-ode}
\end{align}

\subsubsection*{Open loop control}
In the case of open loop control, there are $|\mathcal{U}_{o, 0}|=2^3 = 8$
different control policies in total for both the jump process
(\ref{birth-death-process}) and the deterministic ODE (\ref{birth-death-ode}),
regardless of the value $N$, since one of the two
controls $\nu^{(0)}, \nu^{(1)}$ can be selected at any of the three control stages. The optimal control
is obtained by simply comparing the costs of all $8$ possible policies. In
Figure~\ref{ol-birth-death-fig}, the evolutions of the means and the standard
deviations of the density $z^{u,N}(t)$ are
shown for different $N$. For both control sets $\mathcal{A}_1$,
$\mathcal{A}_2$, it is observed that the standard deviations decrease and the means get
closer to that of the ODE controlled by the optimal control policy as $N$ grows larger.
For the control set $\mathcal{A}_2$, we observe that the suboptimal policy
$u_2=(1,1,0)$ leads to a cost which is close to the optimal cost (that is determined by choosing the optimal policy $u_1=(1,0,1)$) of the ODE system. (For the ease of notation,
we use the index of the control action to denote the control policy, e.g.
$(1,0,1)$ means $(\nu^{(1)}, \nu^{(0)}, \nu^{(1)})$.) For
the jump processes with $N=40$ or $N=100$, $u_2$ performs even better than $u_1$; cf.~Figure~\ref{cost-birth-death-fig}.

\subsubsection*{Feedback control}
Now we turn to the feedback control problem, in which case
the optimal control policy can be obtained
by iterating the dynamic programming equations (\ref{optimality-eqn})--(\ref{terminal-value-fun})
by backward iterations. As the state space $\mathbb{X} = \mathbb{N}^+$ is
infinite, finite state truncation is necessary for
Algorithm~\ref{algo-feedback} to work.
\textcolor{revisecolor}{Based on a rough estimation of the solution of ODE (\ref{birth-death-ode}),
and taking account of the form of the cost functional (\ref{birth-death-cost-J}), the initial condition
$z_0 = 1.2$, as well as the jump rates $\kappa_+, \kappa_-$, we truncate the space into the finite subset
$\mathbb{X}_{cut}=\{N/2, N/2+1, \cdots, 2N\} \subset \mathbb{N}^+$ (see discussions in Subsection~\ref{sub-feedback}).}

Figure~\ref{fb-birth-death-fig} shows the means and the standard
deviations of $z^{u, N}(t)$ under the optimal feedback control policy as a function of time for increasing $N$.
Generally, for both control sets $\mathcal{A}_1$ and $\mathcal{A}_2$, the optimal feedback control policies lead
to smaller costs as compared to the optimal open loop controls (Figure~\ref{cost-birth-death-fig}).
Specifically, we observe in Figure~\ref{fb-birth-death-fig}(a) that, for the
control set $\mathcal{A}_1$, the standard deviations
decrease and the means converge to the densities of the optimally controlled ODE system (by
$u_2$) as $N$ increases. For the control set $\mathcal{A}_2$, due to the
existence of the competing
policy $u_2$ (in this case, $u_1=(1,0,1)$ is optimal for the ODE system), some states with density close to $z=1.0$ may select the control
$\nu^{(1)}$ at stage $t=1.0$, while others select $\nu^{(0)}$ (see Table~\ref{tab-birth-death-policy}), which leads
to a significant rise in the standard deviation at the next control stage $t=2.0$ (see Figure~\ref{fb-birth-death-fig}(b));
we moreover notice that the convergence of the empirical means of the controlled jump process at time $t=2.0$ to
the ODE solution is slower than in case of the control set
$\mathcal{A}_1$ as $N$ increases. The last observation is in agreement with
Figure~\ref{set-2-bd-fig}(a) which shows the bimodal
probability density function of the optimally controlled process at time $t=2.0$ that becomes
even more pronounced for larger values of $N$. Nevertheless, Figure~\ref{cost-birth-death-fig} clearly shows the convergence of the
cost values of both open loop and feedback control policies as $N$ increases, in line with the theoretical prediction. 
Also notice that, in Figure~\ref{cost-birth-death-fig}(b), the optimal costs using feedback and
hybrid policies for finite $N$ can be smaller than the optimal cost of the limiting ODE system, i.e. the convergence may be not monotonically decreasing
from above. As a final demonstration, Figures~\ref{cost-birth-death-fig}(a)
and \ref{set-2-bd-fig}(b) show a comparison of the SSA and the tau-leaping methods,
with the clear indication that the results using the tau-leaping method are close to the SSA prediction, but  at much lower computational cost.
\begin{table}[pt]
  \centering
  \begin{tabular}{c|c|cl|ll}
    \hline
    \hline
    \multirow{2}{*}{No.} & \multirow{2}{*}{control} & \multicolumn{2}{c|}{$\mathcal{A}_1$} &
    \multicolumn{2}{c}{$\mathcal{A}_2$} \\
    & & $\kappa_-$ & $\kappa_+$ & $\kappa_-$ & $\kappa_+$ \\
    \hline
    $0$ & $\nu^{(0)}$ & \underline{$0.6$} & $1.0$ &  \underline{$0.8$} & $1.0$ \\
    $1$ & $\nu^{(1)}$  & $1.0$ & $0.8$ & $1.0$ & $0.8$ \\
    \hline
    \hline
  \end{tabular}
  \caption{Two different control sets $\mathcal{A}_1$, $\mathcal{A}_2$ for the birth-death jump process. Each
  set contains two controls where the underlined entries indicate different control actions in $\mathcal{A}_1$ and $\mathcal{A}_2$. \label{tab-birth-death-policy} }
\end{table}
\begin{figure}[h]
\centering
  \begin{tabular}{cc}
    \subfigure[Control set
    $\mathcal{A}_1$]{\includegraphics[width=6.0cm]{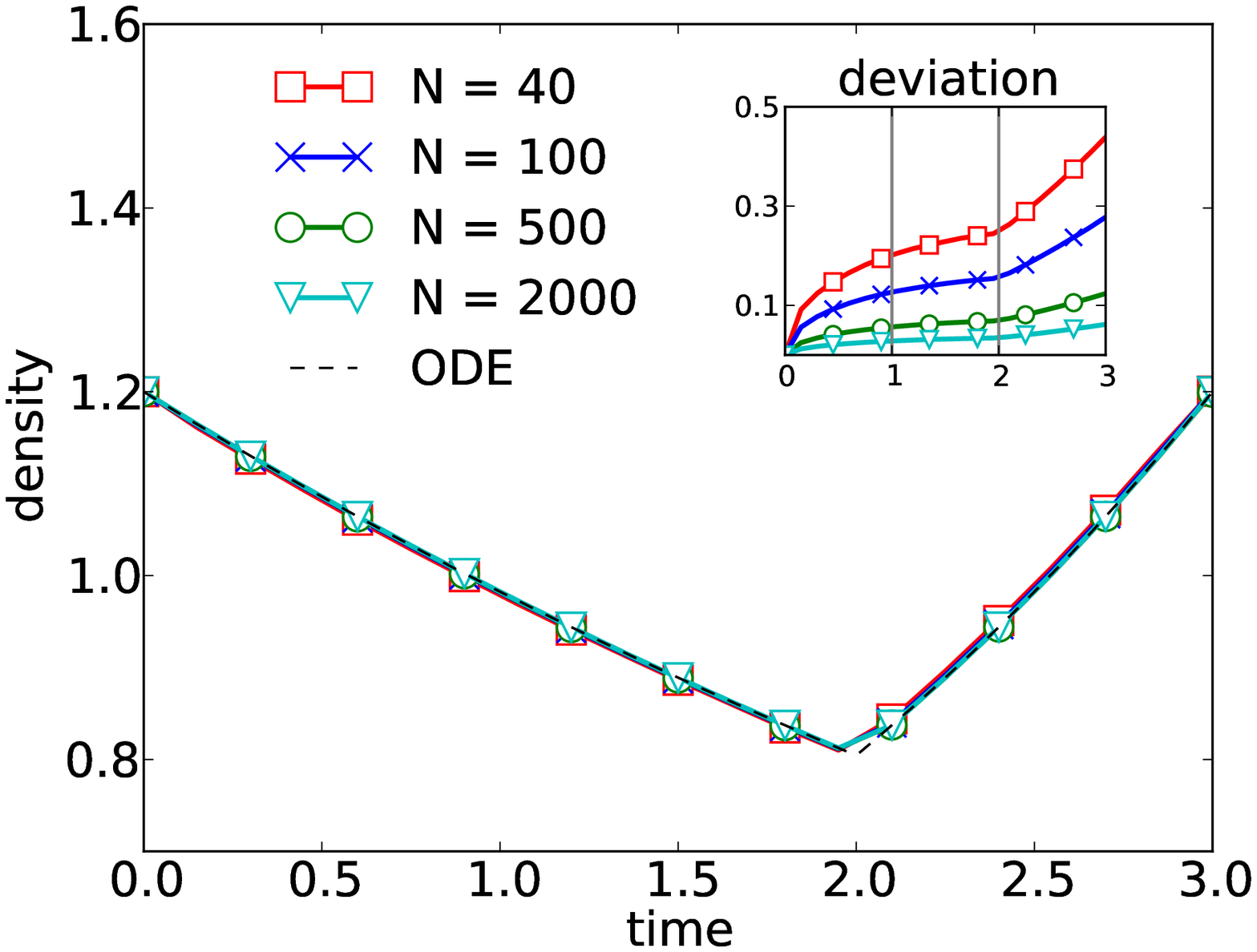}} &
    \subfigure[Control set $\mathcal{A}_2$]{\includegraphics[width=6.0cm]{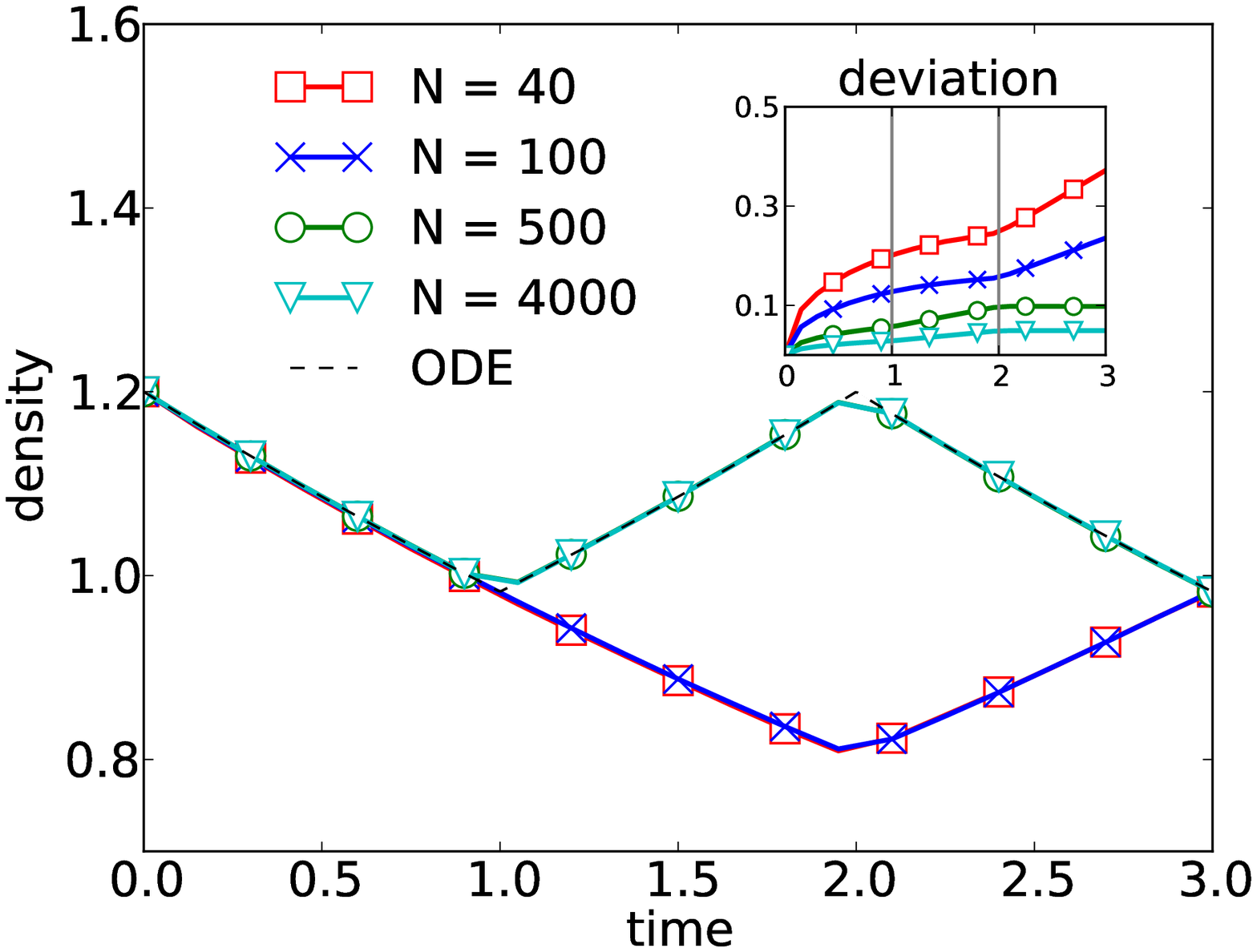}}
\end{tabular}
\caption{Birth-death process. Evolution of the empirical means and the standard
  deviations (inset plot) of the normalized density process $z^{u,N}$
  under the optimal open loop control policies in comparison with the ODE solutions.
  Here $N$ is the scaling number and controls are switched at times $t = 0.0,
  1.0, 2.0$. (a) Control Set $\mathcal{A}_1$. The optimal policy is
$u_2 = (1,1,0)$ both for the jump process for all $N$ and for the ODE system. (b)
Control Set $\mathcal{A}_2$. The optimal
policy is $u_2=(1,1,0)$ for the jump process with $N = 40, 100$, but it
becomes $u_1=(1,0,1)$ for $N=500, 4000$ and the ODE system. \label{ol-birth-death-fig}}
\end{figure}
\begin{figure}[h]
\centering
  \begin{tabular}{cc}
    \subfigure[Control set
    $\mathcal{A}_1$]{\includegraphics[width=6.0cm]{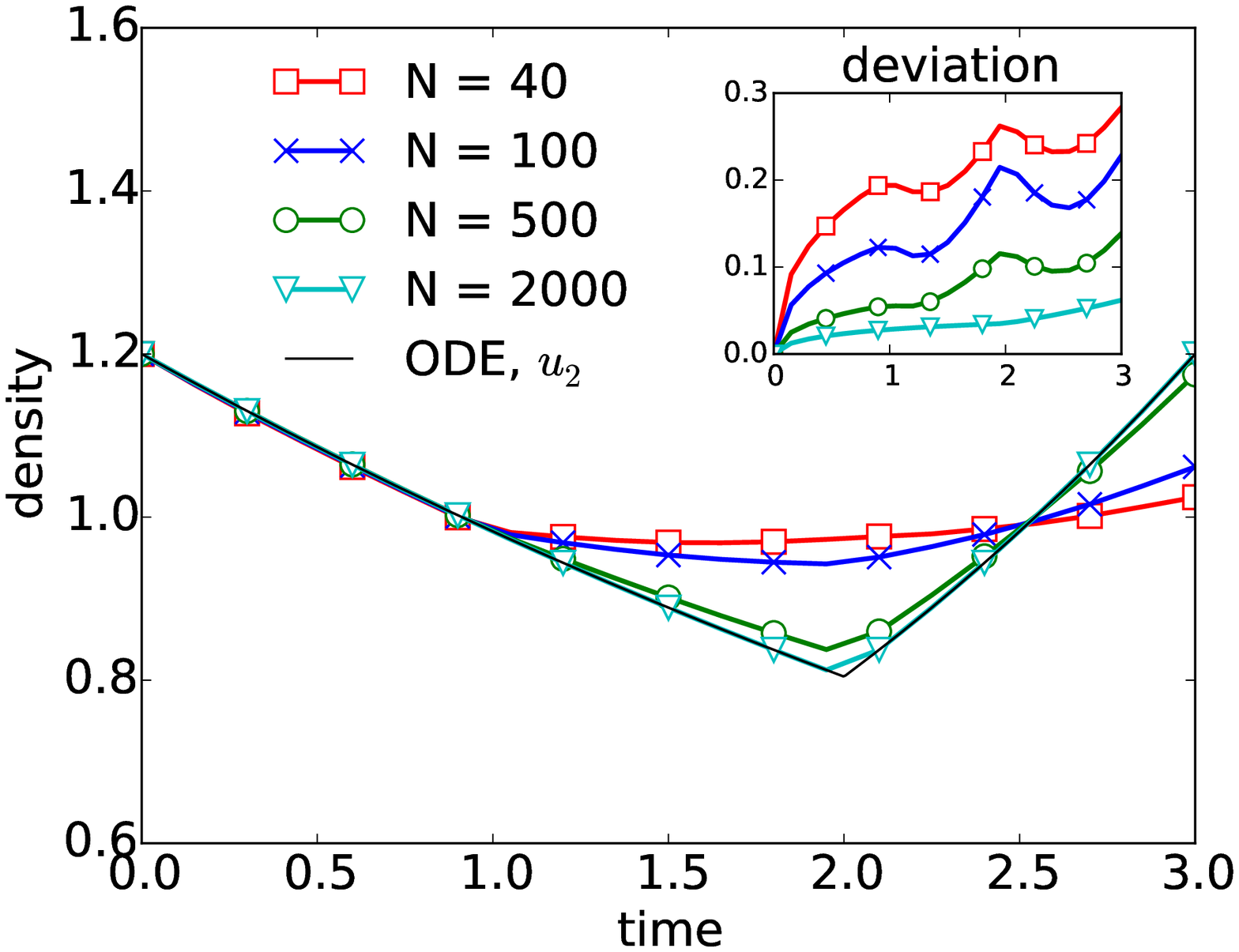}} &
    \subfigure[Control set
    $\mathcal{A}_2$]{\includegraphics[width=6.0cm]{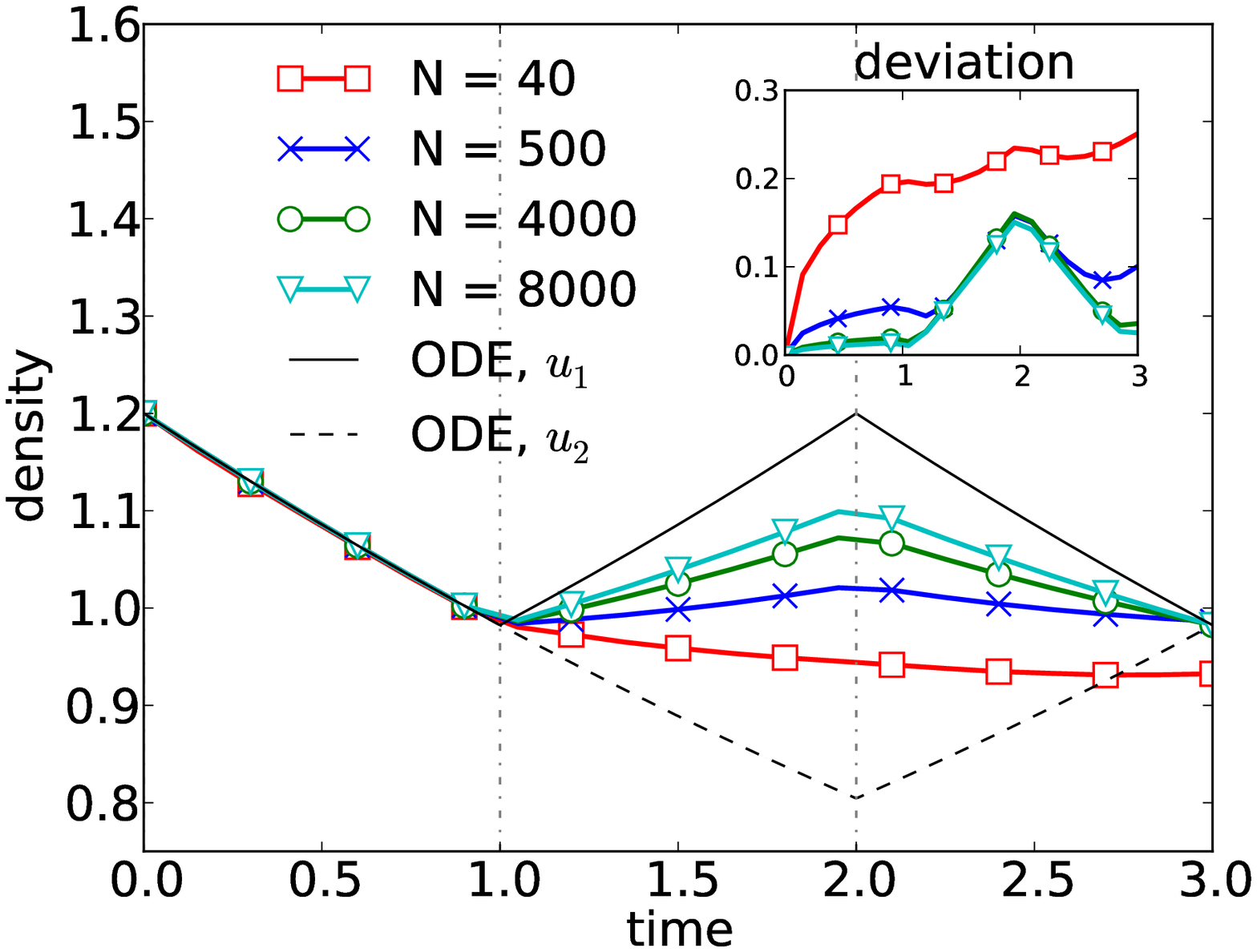}}
\end{tabular}
\caption{Birth-death process. Evolution of the empirical means and the
  standard deviations (inset plot) of
  the normalized density process $z^{u,N}$ under the optimal feedback
  control policies in comparison with the ODE solutions. $N$ is the
scaling number and controls are switched at times $t = 0.0, 1.0, 2.0$. (a)
Control set $\mathcal{A}_1$: as $N$ increases, the standard deviations
decrease and the empirical means get closer to the ODE solution under the optimal
policy $u_2 = (1,1,0)$. (b) Control set $\mathcal{A}_2$: the policies
$u_1=(1,0,1)$ and $u_2=(1,1,0)$ are the dominant (sub)optimal control policies for the ODE system.
\label{fb-birth-death-fig}}
\end{figure}
\begin{figure}[h]
\centering
  \begin{tabular}{cc}
    \subfigure[Control set
    $\mathcal{A}_1$]{\includegraphics[width=6.0cm]{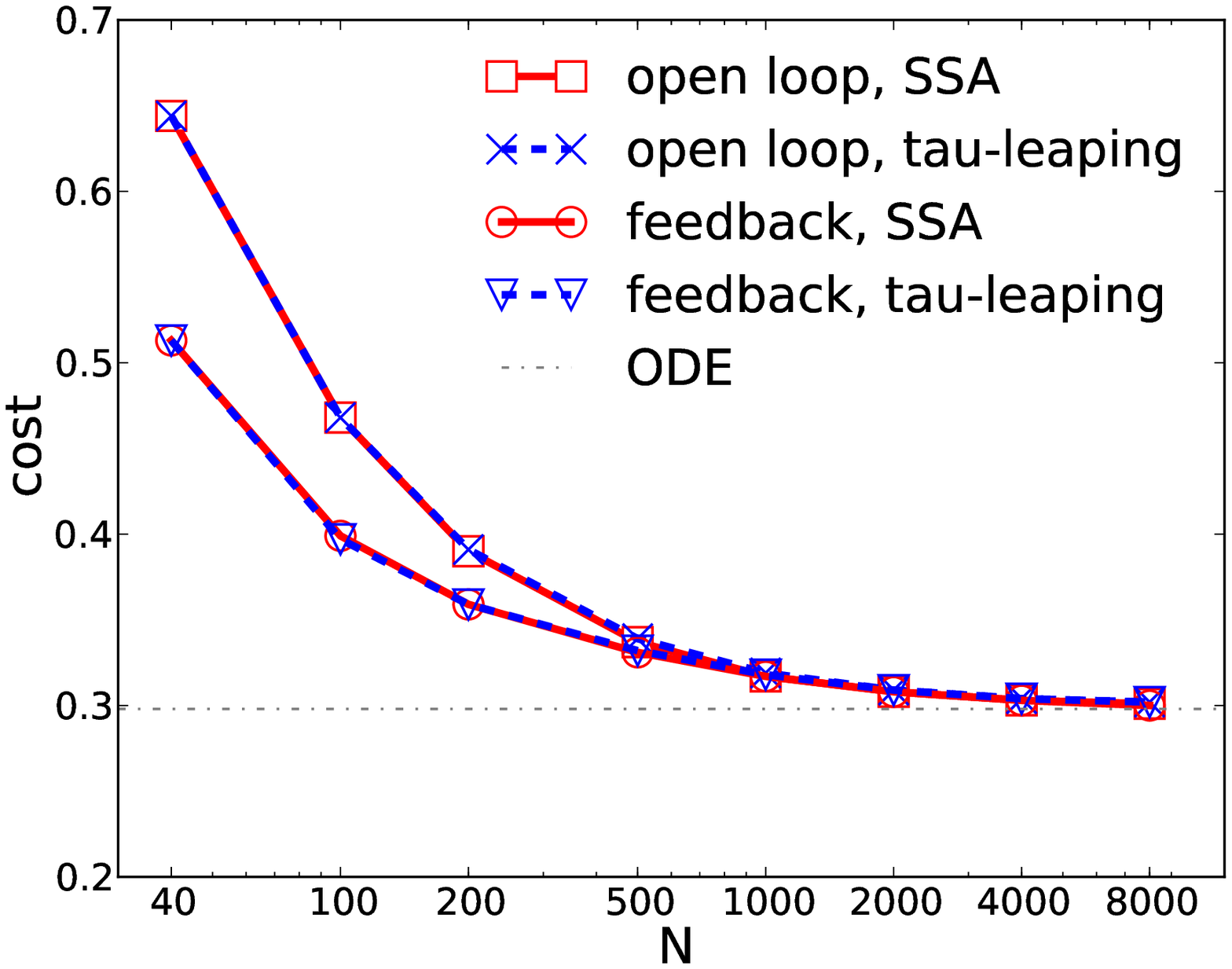}} &
    \subfigure[Control set
    $\mathcal{A}_2$]{\includegraphics[width=6.0cm]{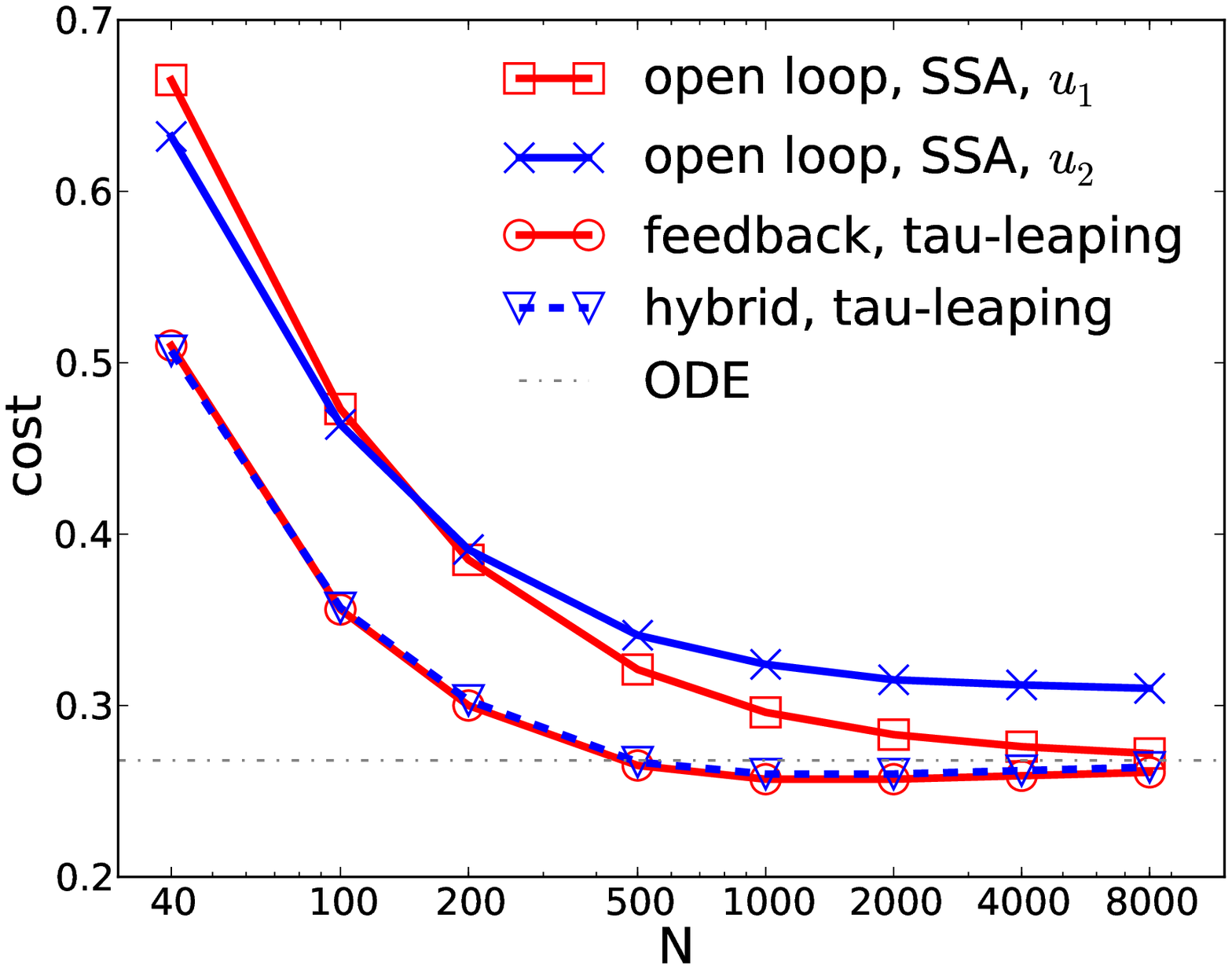}}
\end{tabular}
\caption{Birth-death process. Cost values for the jump processes with different scaling number $N$.
  Both SSA and tau-leaping methods are used to
  sample trajectories.  For the control set $\mathcal{A}_2$, $u_1 =
  (1,0,1)$, $u_2 = (1,1,0)$ are the best two open loop policies.  \label{cost-birth-death-fig}
}
\end{figure}
\begin{figure}[h]
\centering
  \begin{tabular}{ccc}
    \subfigure[Probability density at $t = 2.0$]{\includegraphics[width=0.301\textwidth]{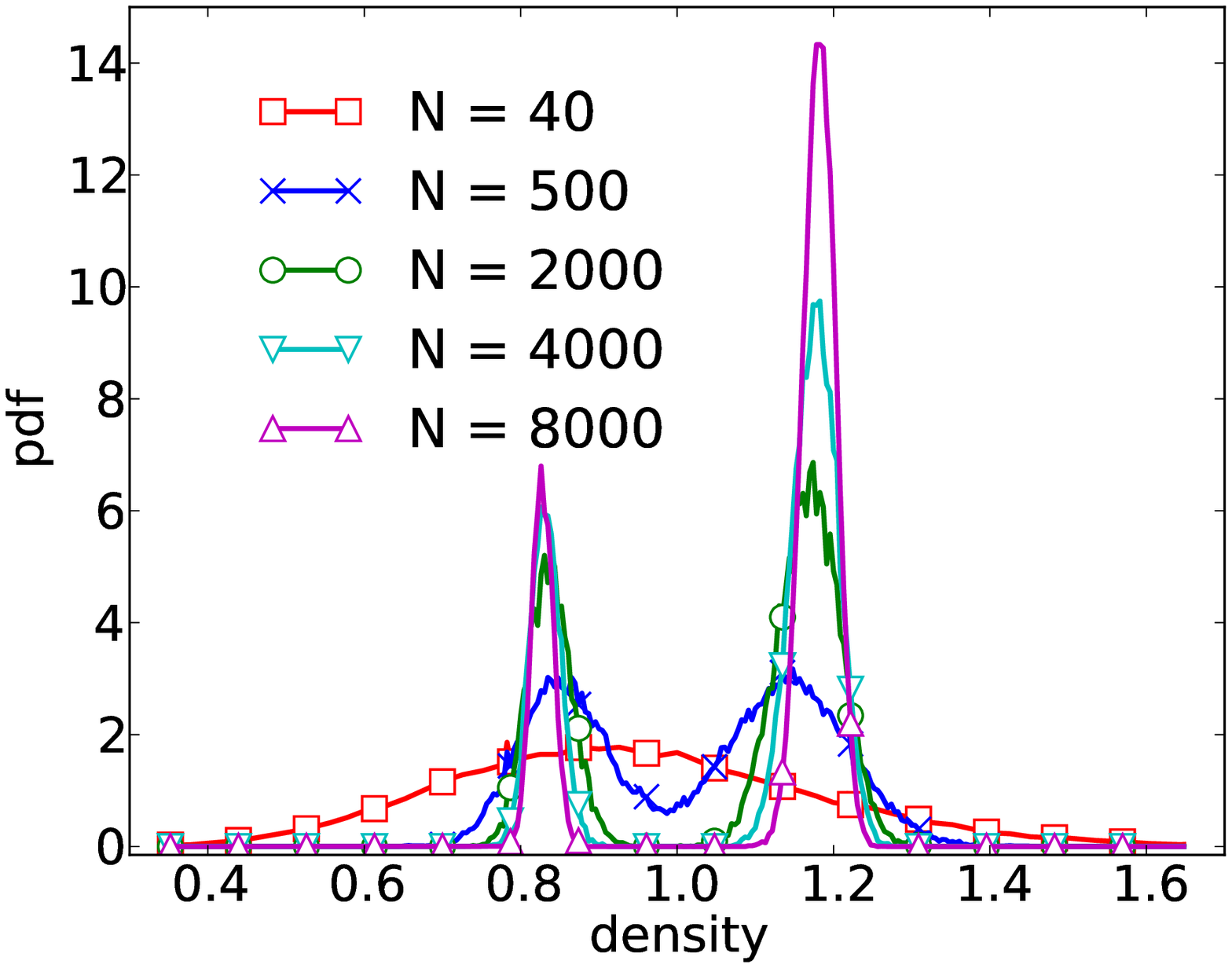}} &
    \subfigure[Run time]{\includegraphics[width=0.301\textwidth]{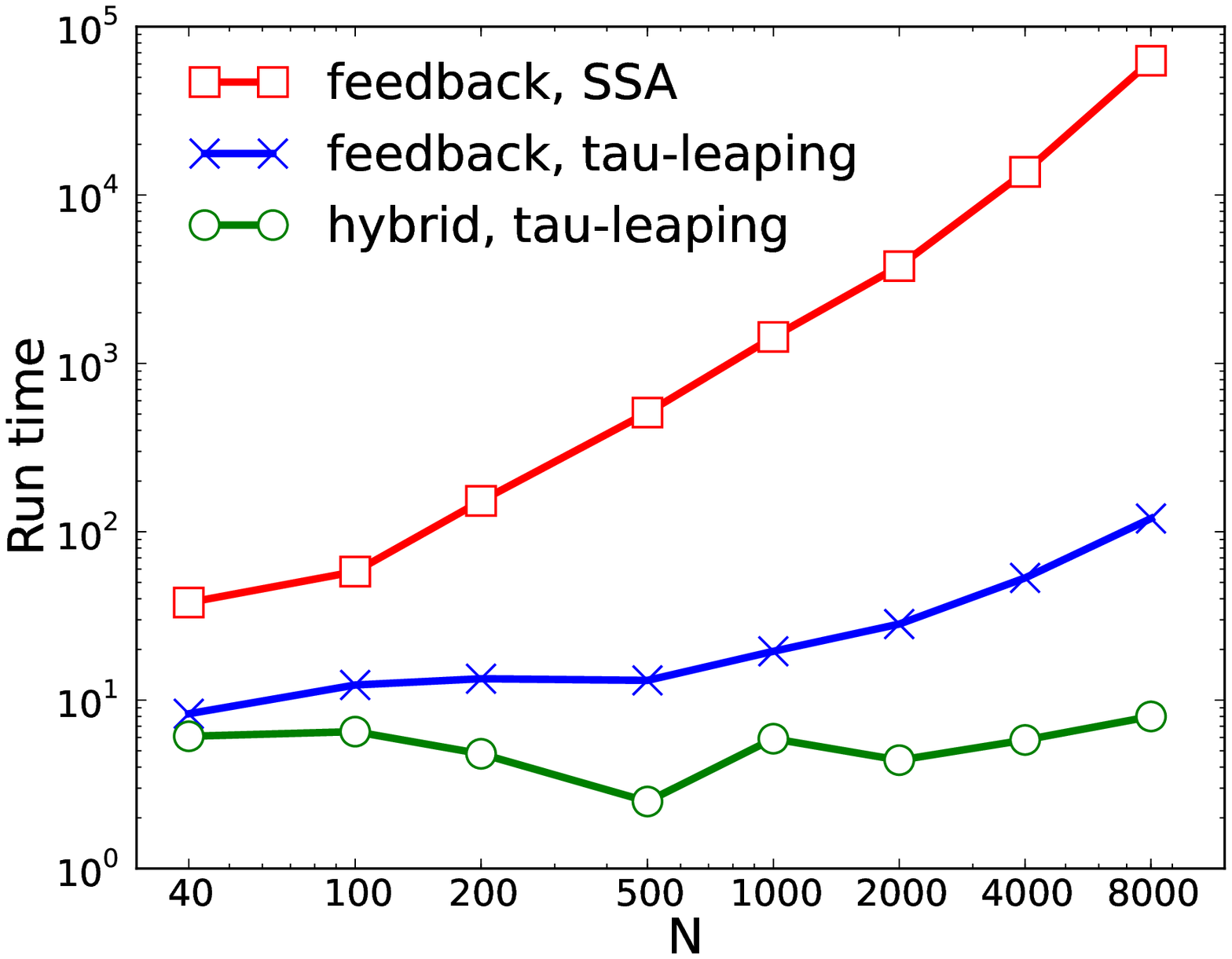}} &
    \subfigure[Number of states]{\includegraphics[width=0.301\textwidth]{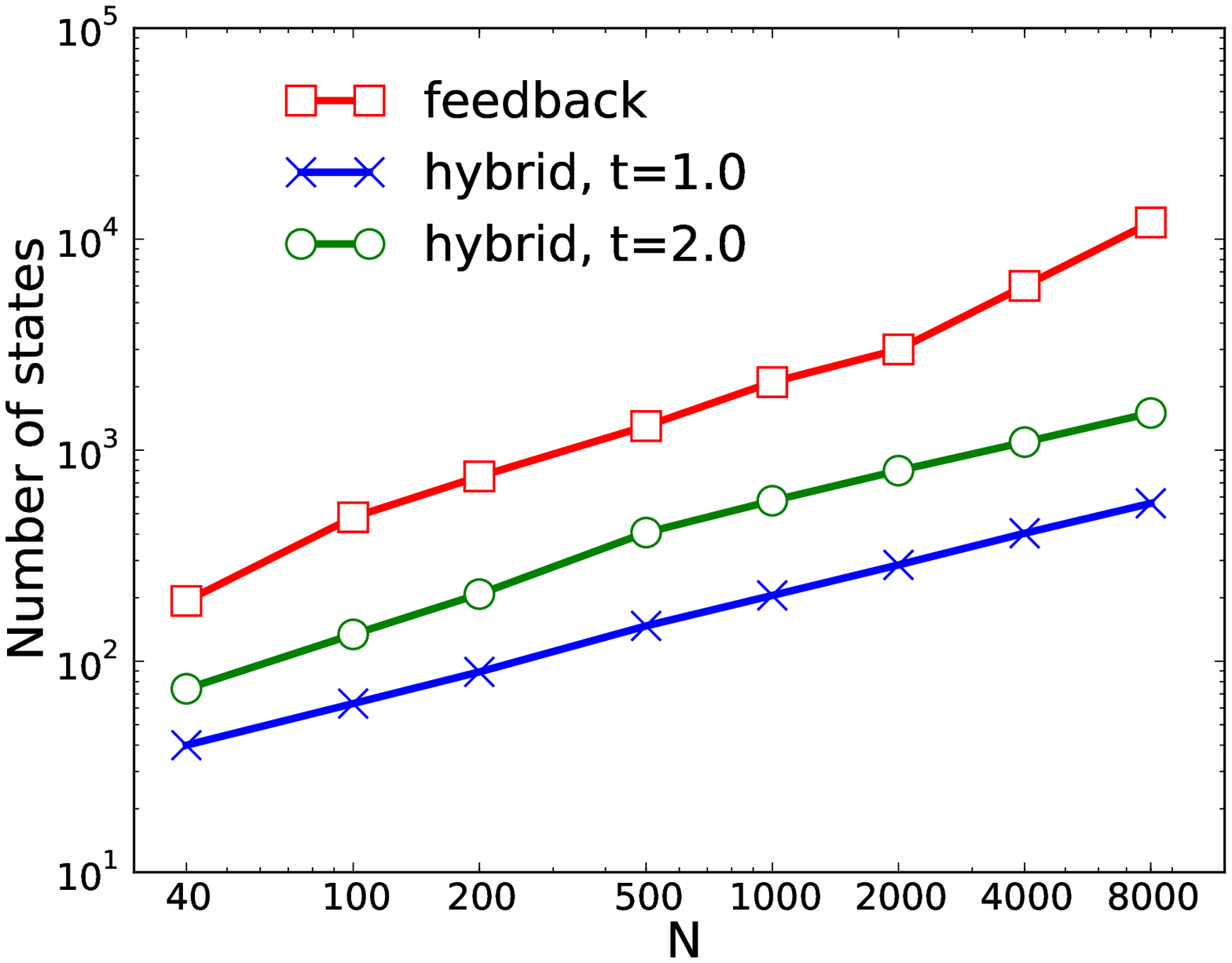}}
\end{tabular}
\caption{
Birth-death process. Dynamics under the control set $\mathcal{A}_2$.
(a) The probability distribution of states of the jump process at time $t=2.0$ under the optimal feedback control.
(b) The CPU run time (in seconds) for different values of $N$, where the
algorithm is run in parallel using $10$ processors in each case.
(c) Number of states in the sets $\mathbb{X}_{cut}$ (feedback control) and
$\mathcal{S}_1$, $\mathcal{S}_2$ (hybrid control).
\label{set-2-bd-fig}
}
\end{figure}

\subsubsection*{Hybrid control}
Finally, we consider the hybrid control policy following the procedure
discussed in Subsection~\ref{sub-hybrid} and we confine our attention to the
control set $\mathcal{A}_2$. To assess the approximation quality of the hybrid control algorithm, we compute the cost
under the open loop control policies for various values of $N$ and with $5000$
trajectories for each possible policy.
As ``good'' control policies, we define the suboptimal controls with $n_{ol} = 2$
and $\epsilon_{ol} = 0.05$ (see page \pageref{goodpolicies}). Sets
$\mathcal{S}_j$ are computed from $M_{ol} = 5000$ realizations for each
``good'' open loop policy according to (\ref{closeness}) with $\zeta=2.5$. As
Figure~\ref{set-2-bd-fig}(c) illustrated, the cardinality of the sets $\mathcal{S}_1$ and
$\mathcal{S}_2$ is much smaller than the cardinality of $\mathbb{X}_{cut}$ used in the
feedback control case, which can lead to a tremendous reduction of the computational effort as
compared to Algorithm \ref{algo-feedback} at almost no loss of numerical accuracy (see Figure~\ref{cost-birth-death-fig}).

\subsection{Predator-prey model}
In this subsection, we consider a two dimensional predator-prey model on the state
space $\mathbb{X}=\mathbb{N}^+\times \mathbb{N}^+$. We call $A$ and $B$ the prey and
predator species, and let $x = (x^{(1)}, x^{(2)})
\in \mathbb{X}$ denote the numbers of species $A$ and $B$. We suppose that
both the prey and predator reproduce
or decease naturally, with the predator eating the prey in order to reproduce.
Recalling the notations explained in Subsection~\ref{sub-concrete}, the dynamics of
$A$, $B$ species can be modelled as a jump process on $\mathbb{X}$ according to the rules (see
\cite{ode_kurtz_limit})
\begin{enumerate}
  \item
    \ce{A ->[\lambda_1] 2 A} \,, \hspace{0.2cm}
    \ce{A ->[\mu_1] $\emptyset$}
  \item
    \ce{B ->[\lambda_2] 2 B} \,, \hspace{0.2cm}
    \ce{B ->[\mu_2] $\emptyset$}
  \item
    \ce{A + B ->[b] B} \,, \hspace{0.2cm}
    \ce{A + B ->[c] A + 2B} \,.
 \end{enumerate}
A control corresponds to a vector $\nu = (\lambda_1, \mu_1, \lambda_2, \mu_2,
b, c)$, where each parameter takes positive real values. Now we define the jump vectors $l_1 = (1,0)$, $l_2 = (0,1)$ and consider the normalized
state vector $z = (z^{(1)}, z^{(2)}) = x/N \in \mathbb{R}^2$ for a fixed scaling
parameter $N \gg 1$.  The jump rates for the normalized density process are then given by
\textcolor{revisecolor}{
\begin{align}
  \begin{split}
    &f^{\nu,N}_d(z,\frac{l_1}{N}) = \lambda_1 N z^{(1)}\,, \hspace{0.5cm} f^{\nu,
  N}_d(z, -\frac{l_1}{N}) = N (\mu_1 + bz^{(2)})z^{(1)}\,,\\
  &f^{\nu, N}_d(z, \frac{l_2}{N}) = N (\lambda_2 + cz^{(1)}) z^{(2)}\,, \hspace{0.5cm}
  f^{\nu, N}_d(z, -\frac{l_2}{N}) = \mu_2 N z^{(2)}\,,
\end{split}
\end{align}
which indicate that the process is density dependent 
(see Subsection~\ref{sub-concrete}), with the vector fields $F^{\nu, N}(z)$ in
(\ref{F-nu-N-general}) given by 
\begin{align}
  F^\nu(z) = F^{\nu,N}(z) = \Big((\lambda_1 -\mu_1) z^{(1)} - bz^{(1)}z^{(2)}\,,
  \,cz^{(1)}z^{(2)} - (\mu_2 - \lambda_2) z^{(2)}\Big).
\end{align}
}

Our aim is to study the optimal control problem on a finite time-horizon $[0,T]$, with terminal time $T = 5.0$ and $K = 5$ control stages at times $t = j \times 1.0$, $0 \le j \le 4$. We define the cost functional as
\begin{align}
  J_N(z_0, u) = \mathbf{E}_{z_0}^u\!\left[\int_0^{5.0} \Big(|z_t^{(1), u, N} - 2
  z_t^{(2), u, N}| + |z_t^{(1), u, N} - 1.5|\Big)\, dt\right], \quad u \in
  \mathcal{U}_{\sigma, 0}\,,
\end{align}
where $z^{u,N}(t) = (z^{(1),u,N}_t, z^{(2),u,N}_t) = N^{-1} x^{u,N}(t)$ is the
normalized density jump process with initial condition $z^{u,N}(0) = z_0$. In our numerical
experiment, we set $z_0 = (1.0, 0.4)$ and choose $N = 50$,
$100$, $200$, $500$, $1000$, $2000$, $4000$.

The particular choice of the cost functional $J_N$ is
aimed at maintaining the density of the prey species around $z^{(1)}=1.5$ over
time $[0, 5.0]$, with roughly about two times more prey than predator.
The control set $\mathcal{A}$
contains three different controls and is shown in Table~\ref{tab-control-set-ex2}:
Observe that, in comparison with $\nu^{(0)}$, the prey reproduces faster under
the control $\nu^{(1)}$ and the predators decease more slowly, while
the control $\nu^{(2)}$ has the reverse effect.

\subsubsection*{Open loop control}
We do a brute-force calculation of the optimal open loop control policy based
on ranking all
possible $3^5=243$ policies in $\mathcal{U}_{o, 0}$ according to their cost. In each case, $50000$ trajectories are sampled using both SSA
and tau-leaping methods. From Table~\ref{tab-step-size-ex2}, we conclude that for large $N$ ($\ge 500$),
tau-leaping method outperforms the SSA, as is indicated by the large increment of the effective time step sizes.
Except for the system with $N=50$ whose optimal open loop control policy is $u_1 = (0, 2, 1, 0,
2)$ with the corresponding cost $11.26$, the optimal policies for other larger $N$ are all $u_2 =
(0,2,1,2,2)$, which is also the optimal policy for the limiting ODE system
(for $N=50$, $u_2$ is the second best policy with cost $11.30$), see
Figure~\ref{fig-cost-ex2}.
The empirical means and the standard deviations of the normalized density process $z^{u,N}$ 
are shown in Figure~\ref{fig-mean-var-ol-ex2} for various values of $N$. As can be expected from the theoretical predictions, we observe that the mean values approach the solution of the limiting ODE, with the standard deviations
decreasing as $N$ increases. Convergence of the cost values to the cost
value of the limit ODE system is also observed in
Figure~\ref{fig-cost-ex2}.
\begin{table}
\centering
\begin{tabular}{c|c|cccccc}
  \hline
  \hline
  No. & control & $\lambda_1$ & $\mu_1$ & $\lambda_2$ & $\mu_2$ & $b$ & $c$ \\
  \hline
  $0$ & $\nu^{(0)}$ & $2.5$ & $0.2$ & $0.2$ & $2.0$ & $2.0$ & $2.0$ \\
  $1$ & $\nu^{(1)}$ & \underline{$2.7$} & $0.2$ & $0.2$ & \underline{$1.5$} & $2.0$ & $2.0$ \\
  $2$ & $\nu^{(2)}$ & $2.5$ & $0.2$ & $0.2$ & \underline{$2.5$} & $2.0$ & $2.0$ \\
  \hline
\end{tabular}
\caption{Predator-prey model. The control set $\mathcal{A}$ contains three different controls to modify the rates in the
predator-prey model. The major differences among the controls are indicated by
the underlined rates.  \label{tab-control-set-ex2}}
\end{table}
\begin{table}
\centering
\setlength\tabcolsep{4.0pt}
\small
\begin{tabular}{c|ccccccc}
  \hline
  \hline
  $N$ & $50$ & $100$ & $200$ & $500$ & $1000$ & $2000$ & $4000$ \\
  \hline
  $\Delta_{s}t$ & $1.8 \times 10^{-3}$ &$9.0 \times 10^{-4}$ & $4.5 \times
  10^{-4}$ &$1.8 \times 10^{-4}$ & $9.0 \times 10^{-5}$ & $4.5 \times 10^{-5}$
  & $2.2 \times 10^{-5}$ \\
  $\Delta_{\tau}t$ & $1.8 \times 10^{-3}$ & $9.0 \times 10^{-4}$ & $4.5 \times
  10^{-4}$ & $3.9 \times 10^{-4}$ & $1.1 \times 10^{-3}$ & $2.7 \times 10^{-3}$ & $3.2 \times 10^{-3}$\\
  \hline
  \end{tabular}
  \caption{Predator-prey model. Average time step sizes when the SSA (row with
    label $\Delta_{s}t$) or tau-leaping method (row with label $\Delta_\tau t$) are used to generate realizations of the predator-prey
  model.\label{tab-step-size-ex2}}
\end{table}
\begin{table}
\centering
\begin{tabular}{c|rrrrrrr}
  \hline
  \hline
  $N$ & $50$ & $100$ & $200$ & $500$ & $1000$ & $2000$ & $4000$ \\
  \hline
  $N_g$ & $5$ & $5$ &$3$ &$3$ &$3$ & $3$ & $3$ \\
  $M_{ol}$ & $5000$ & $10000$ & $10000$ & $10000$ & $20000$ & $20000$ & $30000$ \\
 $\min\limits_{1 \le j \le 4} |\mathcal{S}_j|$ &  $4090$ & $8738$ & $12024$ & $11545$ & $23120$ &  $26060$ & $40463$ \\
 $\max\limits_{1 \le j \le 4} |\mathcal{S}_j|$ & $11420$ & $30572$ & $25784$ & $14587$ & $29369$ & $29597$ & $44513$ \\
  $9N^2$ & $22500$ & $90000$ & $360000$ & $2250000$ & $9000000$ & $36000000$ & $144000000$\\
  \hline
  \end{tabular}
  \caption{Predator-prey model with hybrid control. The row ``\,$9N^2$'' shows
    the estimated state space cardinalities
  after truncation if a simple cut-off criterion is used. The row ``$N_g$'' shows the
  number of the ``good'' open control policies, and ``$M_{ol}$'' denotes the number of
  trajectories generated for each ``good'' open policy in the calculation of the sets
  $\mathcal{S}_j$. The other two rows contain the minimum and maximum numbers of states in the sets $\mathcal{S}_j$.
 \label{tab-hybrid-ex2-1}}
\end{table}

\begin{table}
\centering
\begin{tabular}{c|c|rrrrrrr}
  \hline
  \hline
$\epsilon_{near}$ & $N$ &  $50$ & $100$ & $200$ & $500$ & $1000$ & $2000$ & $4000$ \\
  \hline
\multirow{3}{*}{$0.0$}
& $r_{ol}$  & $13.6\%$ & $13.6\%$ &$38.1\%$ &$66.1\%$ &$66.6\%$ &$73.2\%$ & $74.7\textcolor{revisecolor}{\%}$ \\
& time & $1.0$h & $5.3$h & $5.6$h & $7.1$h & $5.0$h & $5.0$h & $8.2$h \\
& cost    & $10.72$ & $9.88$& $9.58$& $9.27$ & $9.18$ &$9.13$ & $9.11$ \\
  \hline
\multirow{4}{*}{$0.02$}
& $r_{ol}$  & $3.3\%$ & $1.1\%$ & $0.9\%$ & $0.6\%$ &$0.3\%$ & $0.4\%$ & $0.3\%$ \\
  & $r_{near}$ & $10.2\%$ & $12.0\%$ & $36.4\%$ & $65.5\%$ & $66.3\%$ & $72.9\%$ & $74.3\%$\\
& time & $1.1$h & $5.5$h & $5.5$h & $7.0$h & $5.7$h & $5.5$h & $7.2$h \\
& cost & $10.60$  &  $9.81$  & $9.47$ & $9.25$ & $9.18$ & $9.13$ & $9.11$\\
  \hline
  \end{tabular}
  \caption{Predator-prey model with hybrid control. The rows ``\,$r_{ol}$'' and
``$r_{near}$'' record the relative frequencies of using an open loop policy or a
feedback policy of a nearest neighbor when the hybrid control policy
is applied (see Subsection~\ref{sub-hybrid}). The row ``time'' shows the
CPU run time (in hours) needed to compute the optimal hybrid control policy
with $20$ processors running in parallel for each $N$. \label{tab-hybrid-ex2-2}}
\end{table}
\begin{figure}[h!]
\centering
  \begin{tabular}{cc}
    \subfigure[]{\includegraphics[width=6.0cm]{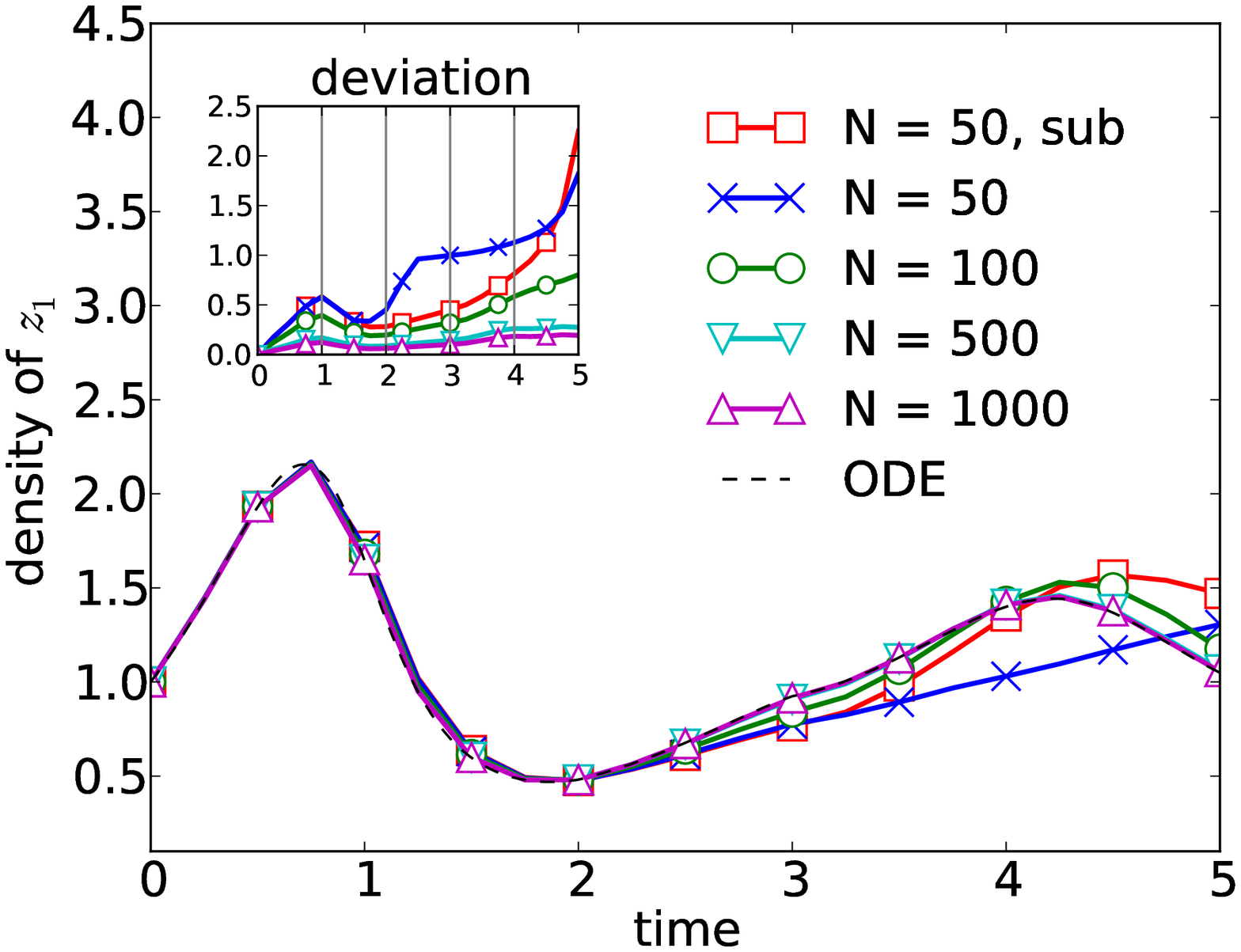}} &
    \subfigure[]{\includegraphics[width=6.0cm]{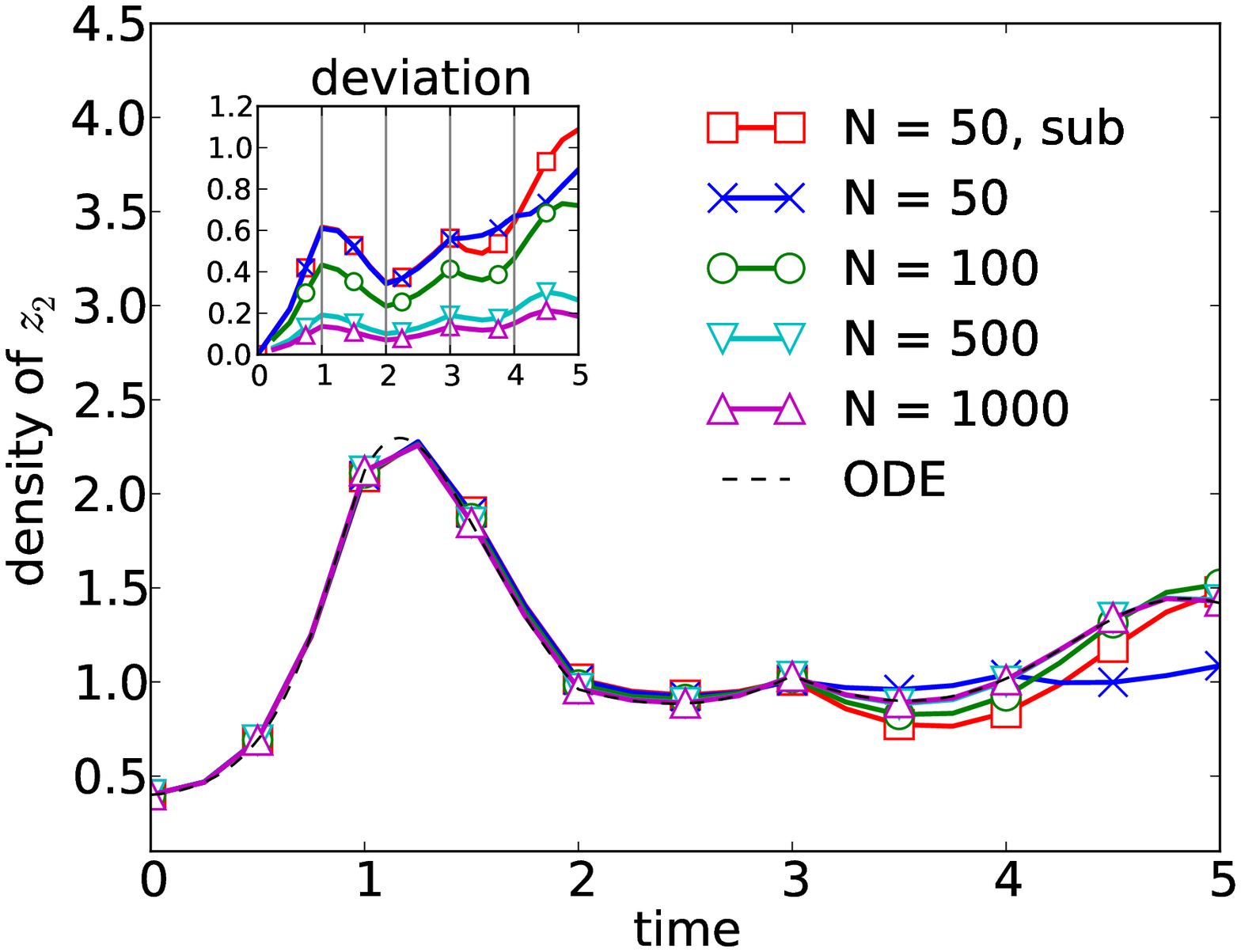}}
\end{tabular}
\caption{Predator-prey model. Evolution of the empirical means and the standard
  deviations (inset plot) of the normalized
  predator and prey states (densities) under the optimal open loop policy. The
  curve labeled by ``$N=50$, sub'' corresponds to the jump process of size $N=50$
that is controlled by the suboptimal policy $u_2$, which becomes the optimal policy for larger $N$.
``ODE'' corresponds to the limiting ODE under the optimal policy
$u_2$.\label{fig-mean-var-ol-ex2}}
\end{figure}
\begin{figure}[h!]
\centering
  \begin{tabular}{cc}
    \subfigure[]{\includegraphics[width=6.0cm]{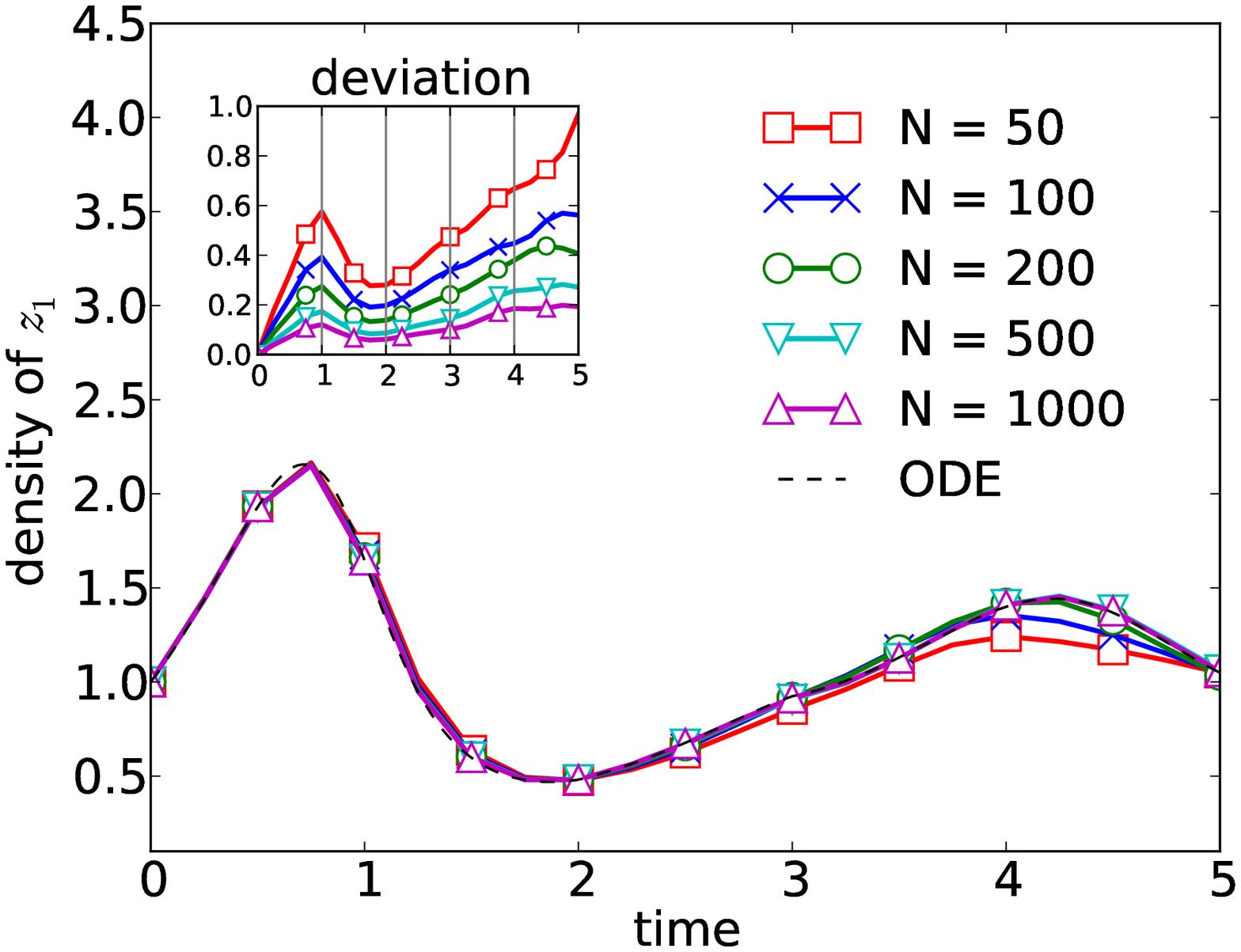}} &
    \subfigure[]{\includegraphics[width=6.0cm]{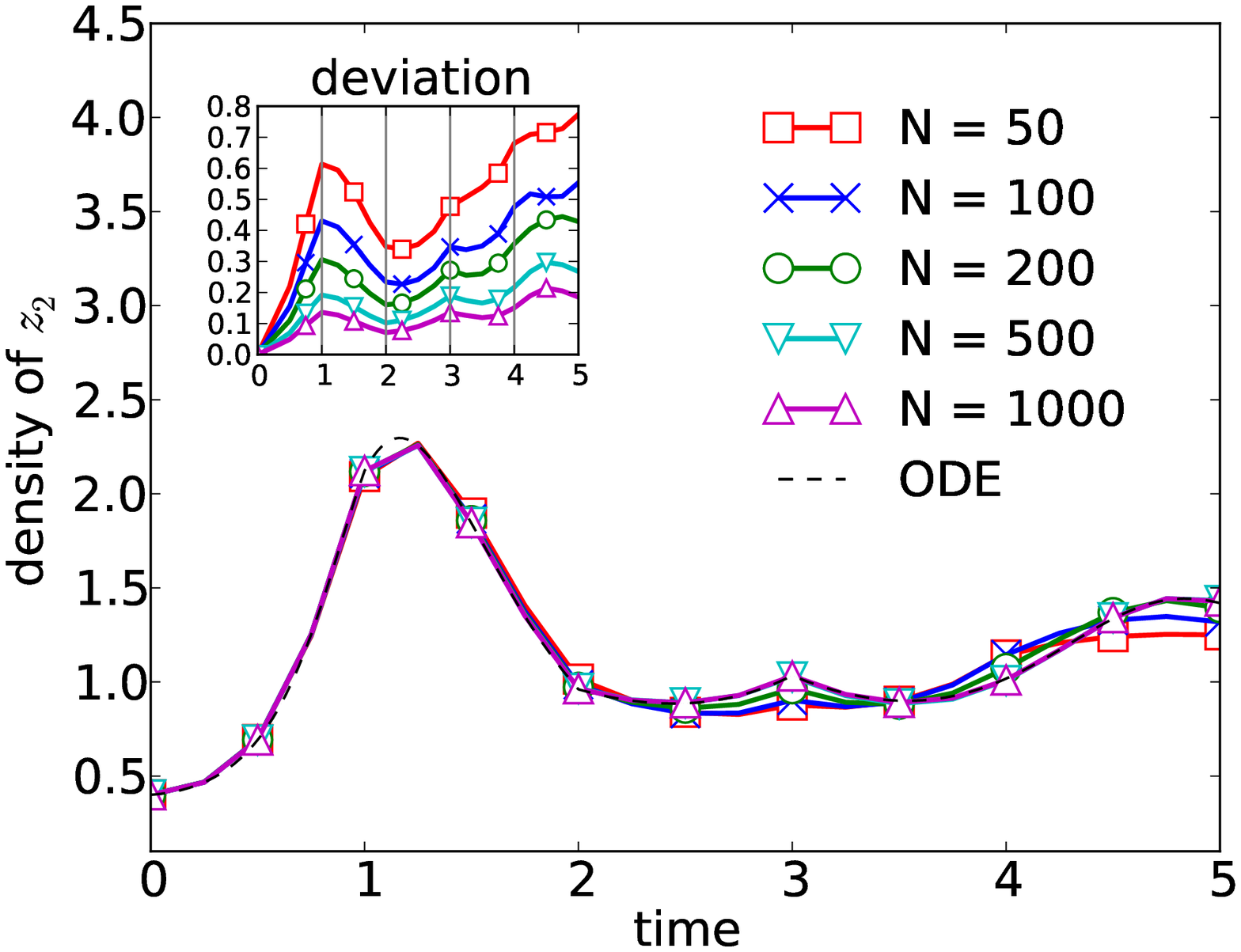}}
\end{tabular}
\caption{Predator-prey model. Evolution of empirical means and standard
  deviations (inset plot) of the normalized
  predator-prey system under the hybrid control policy.
\label{fig-mean-var-hybrid-ex2}}
\end{figure}

\begin{figure}[h!]
\centering
    \includegraphics[width=8.0cm]{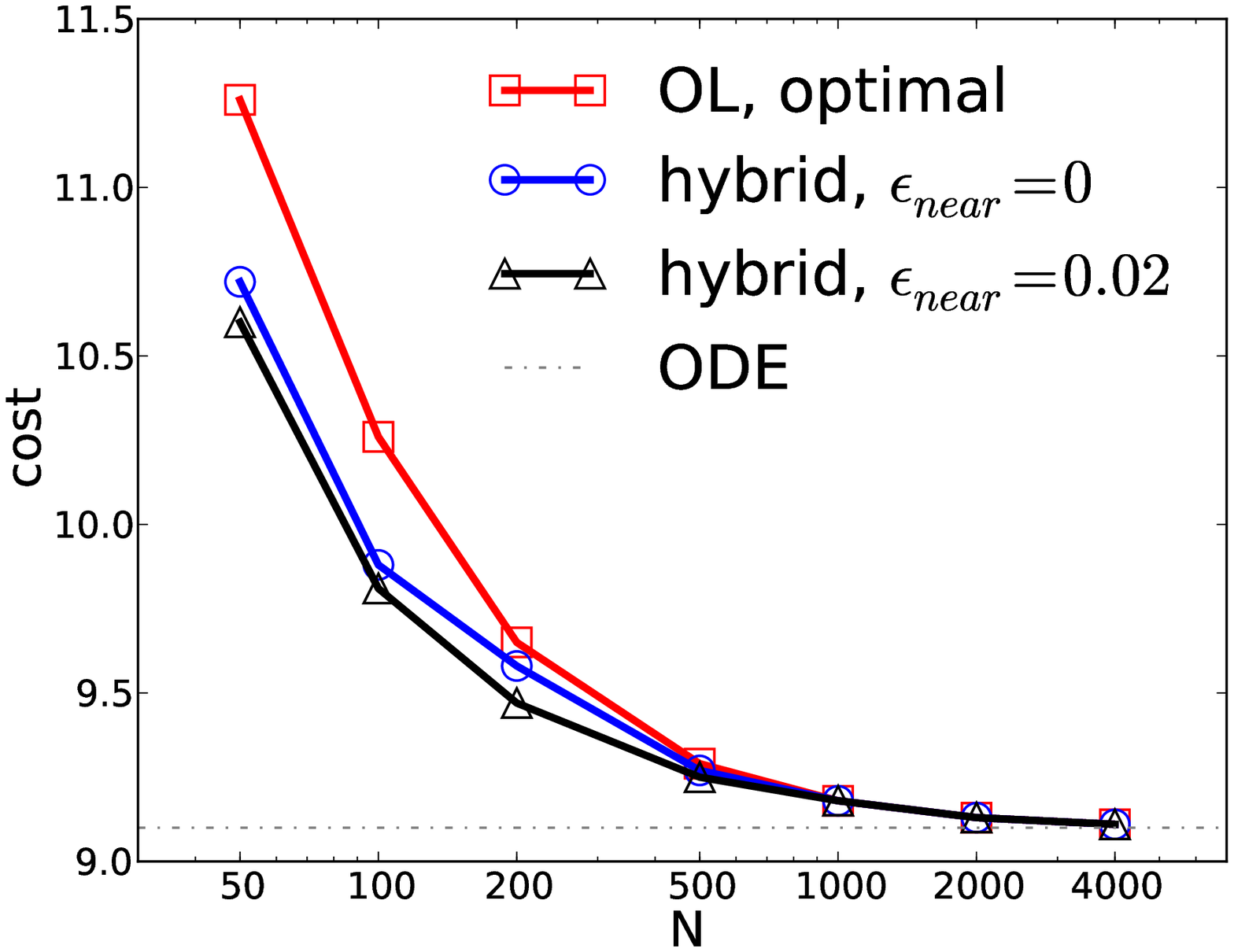}
    \caption{
      Predator-prey model.
      Cost values of the predator-prey model under the optimal open loop
      (``OL'') control policy, the hybrid control
    policies with $\epsilon_{near} = 0$ and $0.02$, for various values of $N$. The dotted horizontal line is the optimal cost for
  the limiting ODE system. \label{fig-cost-ex2}}
\end{figure}

\subsubsection*{Hybrid control}
We continue to study the hybrid control policy introduced in
Subsection~\ref{sub-hybrid}. Firstly, all $243$ possible open loop control policies are
ordered by their costs, among which we identify all ``good'' policies with $n_{ol} = 3$,
$\epsilon_{ol} = 0.05$.  Then, secondly, we estimate the empirical means and
the standard deviations of the process under all ``good''
policies based on $5000$ independent realizations of the process. Thirdly, for each ``good'' policy,
we generate $M_{ol}$ trajectories
once again and collect the accessed states at time $t_j$ in
$\mathcal{S}_j$, $1 \le j < M$ according to the criterion (\ref{closeness})
for $\zeta=3.0$. (Note that $\mathcal{S}_0$ contains only a single element). The minimum and maximum cardinalities
$\min\limits_{1 \le j \le M-1} |\mathcal{S}_j|$ and
$\max\limits_{1 \le j \le M-1} |\mathcal{S}_j|$ of sets $\mathcal{S}_j$ are shown in
Table~\ref{tab-hybrid-ex2-1}.

The reader should bear in mind that, if we wanted to compute the optimal feedback control
policy on a globally truncated state space $\mathbb{X}_{cut}$ (see
Subsection~\ref{sub-feedback}), then it would be necessary to include states whose
normalized components are within $[0, 3.0] \times [0, 3.0]$ as suggested by
the empirical means and standard
deviations of the process (see Figure~\ref{fig-mean-var-ol-ex2}), which
would result in roughly $9N^2$ states in total; even for moderate predator-prey populations, computing the optimal feedback
policy on $\mathbb{X}_{cut}$ is therefore extremely costly. Compared to this approach, the adaptive state truncation that gives
rise to the sets $\mathcal{S}_j$ is much more
efficient in the sense that the overall number of states involved in the
computation of the optimal hybrid policy is much smaller; see
Table~\ref{tab-hybrid-ex2-1} and Figure~\ref{fig-states-ex2}.

Finally, we compute the optimal hybrid policy
using Algorithm~\ref{algo-hybrid} and apply it to the predator-prey
model in the way explained in Subsection~\ref{sub-hybrid}. The resulting cost
values that were estimated based on $50000$ independent realizations are shown in
Table~\ref{tab-hybrid-ex2-2}, Figure~\ref{fig-cost-ex2} and clearly demonstrate the superiority of the hybrid controls over
the optimal open loop control policies (in particular, see Table~\ref{tab-hybrid-ex2-2} for $N=50, 100, 200$). To explain the
observed gain in the numerical speed-up,  Table~\ref{tab-hybrid-ex2-2} also records the relative frequencies $r_{ol}$ of switching to
an open loop policy: For $\epsilon_{near}=0.0$, we observe that the hybrid control frequently switches to the
optimal open loop policy, which is an indicator that the sets $\mathcal{S}_j$ are too
small as the dynamics often hits an ``unknown'' state outside  $\mathcal{S}_j$. Yet, for $\epsilon_{near} = 0.02$, we
find that $r_{ol}$ decreases significantly which suggests that the sets $\mathcal{S}_j$ contain almost all states
that are close to the accessible states under the given control policy.
Note moreover that the resulting cost value for $\epsilon_{near} = 0.02$  is slightly improved over the choice $\epsilon_{near} = 0.0$.

Before we conclude, we would like to stress an important observation that the
standard deviations of the process are smaller under the hybrid control
policy (similarly for the feedback policy) than that under the optimal open loop policy.
This effect can be revealed by comparing Figure~\ref{fig-mean-var-ol-ex2} with Figure~\ref{fig-mean-var-hybrid-ex2} for the same value of $N$, and it suggests that besides providing smaller costs, both hybrid and feedback control
policies have a positive effect on stabilizing the stochastic process.

\begin{figure}[h!]
\centering
    \includegraphics[width=0.95\textwidth]{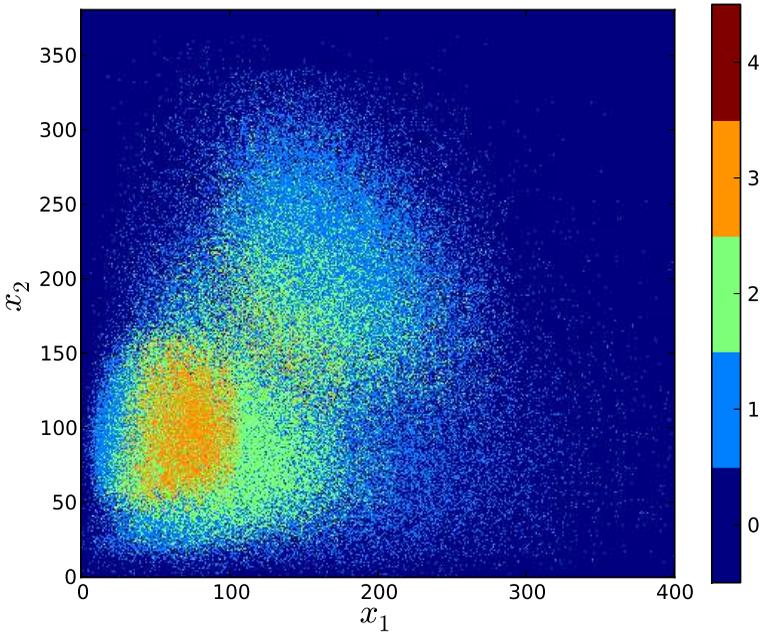}
    \caption{States selected to construct the hybrid control policy in the predator-prey model for $N=100$. The value
      at each grid point $x=(x^{(1)}, x^{(2)})$ counts how many
    sets $\mathcal{S}_j$ contain the state $x$, i.e. the value at
    $x\in\mathbb{X}$ is equal to $\sum\limits_{j=1}^4
  \mathbf{1}_{\mathcal{S}_j}(x)$. \label{fig-states-ex2}}
\end{figure}

\section{Conclusions and future directions}
Due to their wide applicability, Markov Decision Processes
have been the subject of intensive research.
While the theory is well developed, algorithms for numerically computing optimal controls are restricted to small or moderately sized systems.

The aim of this paper was to analyze optimal control problems for Markov jump
process in the large number regime (parameterized by the ``particle'' number
$N \gg 1$), i.e. when the state space is too large to compute the optimal feedback controls using
standard algorithms.
Based on Kurtz's limit theorems, we have established convergence results
for the value functions of the optimal control problems on finite and infinite time-horizons as $N
\rightarrow \infty$. Our results suggest that the optimal open loop control policy for
the limiting deterministic system is a good substitute for the controlled Markov jump
process, for which the optimal feedback policy may be difficult to compute.
Nonetheless, for a given jump process with a possibly large, but finite $N$, the approximation error induced by replacing the
optimal stochastic (feedback) control with the limiting deterministic control is difficult to assess; even for large values of $N$ the stochastic dynamics
controlled by a deterministic open loop control policy is not robust under the intrinsic random perturbations, and may hence
deviate considerably from the optimal regime. To account for this lack of robustness, we proposed an algorithmic strategy to compute a
\textit{hybrid} control policy that is based on a combination of deterministic (open loop) and stochastic (closed loop) controls.
The key idea is to truncate the state space adaptively in time, exploiting data gathered from stochastic simulations under near-optimal open loop
policies, and then to apply the optimal feedback control policy for all times in which the stochastic realizations resides inside the truncated state space (for all other states, the optimal open loop policy is applied). Both the accuracy and the practicability of the proposed \textit{hybrid} algorithm have been demonstrated numerically with birth-death and predator-prey models.

Before we conclude, it is necessary to mention several related topics which go beyond our current work. 
Firstly, throughout the article, we have assumed that the cost can be expressed as a
function of the normalized density process $z^N(t) = N^{-1}x^N(t)$, which in many cases is the natural variable scaling.
\textcolor{revisecolor}{In some cases, however, such as complex chemical reaction networks, it might be necessary to consider a more general scaling
  of the form $z^{(i),N}(t) = N^{-\alpha_i} x^{(i),N}(t)$, $\alpha_i \ge 0$, in
which each chemical species comes with its own scaling order.}
Then, in the limit $N\to\infty$ it may happen that the limit of $z^{(i),N}(t)$ can be deterministic, stochastic or even hybrid when
some of the $\alpha_i$ are equal to zero and others are positive. We emphasize
that in these cases determining the correct scaling of the variables is not a trivial task and 
the convergence analysis is also more involved (see \cite{ball06, kang2013}).
\textcolor{revisecolor}{Secondly, besides the large copy-number $N$, systems in realistic applications may also contain many different species. 
While our analysis is still valid in this case, it may become computationally challenging to compute the hybrid policy proposed in the current work.}
One idea to alleviate the difficulty is to first reduce the dimension of
the system (especially when there are both slow and fast reactions or when the quasi-stationary assumption is satisfied), and then utilize the information
of the reduced system to design numerical algorithms. 
Thirdly, it is also interesting to consider the asymptotic analysis of the
optimal control problem in the case that the control policy \textcolor{revisecolor}{can be switched at
any time or when there is uncertainty in the observation of the system's states}. We leave these aspects for future work.

\section*{Acknowledgement}
The authors acknowledge financial support by the Einstein Center of Mathematics (ECMath) and the DFG-CRC 1114 ``Scaling Cascades in Complex Systems''.

\appendix
\section{A technical lemma}
\label{app-1}
The following inequality has been used in the proof of Theorem~\ref{thm-2}.
\begin{lemma}
  Let $\varphi(z) = |z|^\alpha$, where $z \in \mathbb{R}^n$ and $1 < \alpha
  \le 2$. We have 
  \begin{align}
    & 0 \le \varphi(z + w) - \varphi(w) - z \cdot \nabla \varphi(w)
    \le \frac{4}{\alpha - 1} \varphi\left(\frac{z}{2}\right)\,, \quad \forall z, w \in
    \mathbb{R}^n\,.
    \end{align}
    \label{lemma-1}
\end{lemma}
\begin{proof}
The case $w=0$ can be readily verified. Now assume $w\neq 0$
and consider $z = (z_1, 0, 0, \cdots, 0)^T$, $w = (w_1, w')^T$ where $z_1, w_1 \in
  \mathbb{R}$, $w' \in \mathbb{R}^{n-1}$.
  Defining $g(r) = r^\alpha$ for $r > 0$, it follows that
  \begin{align*}
    \varphi(z & + w) - \varphi(w) - z \cdot \nabla \varphi(w)\\
    =& g\Big(\sqrt{(z_1 + w_1)^2 + |w'|^2}\Big) - g\Big(\sqrt{w_1^2 +
  |w'|^2}\,\Big) -
  g'\Big(\sqrt{w_1^2 + |w'|^2}\,\Big) \frac{w_1z_1}{\sqrt{w_1^2 + |w'|^2}} \\
  =& \int_0^{z_1} \int_0^{r} \Bigg[g''\Big(\sqrt{(s + w_1)^2 + |w'|^2}\,\Big) \frac{(s +
w_1)^2}{(s + w_1)^2 + |w'|^2} \\
& +  g'\left(\sqrt{(s + w_1)^2 + |w'|^2}\right) \Bigg(\frac{1}{\sqrt{(s +
 w_1)^2 + |w'|^2}} - \frac{(s + w_1)^2}{\big((s + w_1)^2 +
 |w'|^2\big)^{\frac{3}{2}}}\Bigg)\Bigg] dsdr \\
  =& \int_0^{z_1} \int_0^{r} \Bigg[g''\left(\sqrt{(s + w_1)^2 + |w'|^2}\right) \frac{(s +
w_1)^2}{(s + w_1)^2 + |w'|^2}\\
& +  \frac{|w'|^2g'\left(\sqrt{(s + w_1)^2 + |w'|^2}\right)}{\left((s + w_1)^2 +
 |w'|^2\right)^{\frac{3}{2}}}\Bigg] dsdr\,.
  \end{align*}
  Since $1 < \alpha \le 2$, we know that $g' , g'' \ge 0$, and $\frac{g'(r)}{r} = \frac{g''(r)}{\alpha-1} = \alpha
  r^{\alpha -2}$ is non-increasing for $r>0$. We also have the simple
  inequality $\frac{a + \frac{b}{\alpha - 1}}{a + b} \le \frac{1}{\alpha - 1}$ , $\forall a, b > 0$. Therefore
  \begin{align*}
     0 \le& \varphi(z + w) - \varphi(w) - z \cdot \nabla \varphi(w)\\
  =& \int_0^{z_1} \int_0^{r} g''\Big(\sqrt{(s + w_1)^2 + |w'|^2}\,\Big) \frac{(s +
  w_1)^2 + \frac{1}{\alpha - 1} |w'|^2}{(s + w_1)^2 + |w'|^2}  \,ds\,dr \\
\le & \frac{1}{\alpha - 1} \int_0^{z_1} \int_0^{r} g''\Big(\sqrt{(s +
  w_1)^2 + |w'|^2}\,\Big)\,ds\,dr\\
\le & \frac{1}{\alpha - 1} \int_0^{z_1} \int_0^{r} g''\big(|s + w_1|\big)\,ds\,dr
  \\
  \le & \frac{2}{\alpha - 1} \int_0^{|z_1|} \int_0^{\frac{r}{2}} g''(s)\,ds\,dr
\le   \frac{4}{\alpha - 1}  g\left(\frac{|z_1|}{2}\right) =
  \frac{4}{\alpha - 1} g\left(\frac{|z|}{2}\right)\,.
  \end{align*}
  For the general case, let $A$ be an $n\times n$ rotation matrix, such that $Az = (z_1, 0, 0, \cdots, 0)^T$, $z_1 \in \mathbb{R}$. Then
  \begin{align*}
    \varphi(z & + w)  - \varphi(w) - z \cdot \nabla \varphi(w)\\
  = & g\big(|z + w|\big) - g\big(|w|\big) - g'\big(|w|\big) \frac{w}{|w|}\cdot z \\
  = & g\big(|Az + Aw|\big) - g\big(|Aw|\big) -  g'\big(|Aw|\big) \frac{Aw}{|Aw|}\cdot Az \\
  = & \varphi(Az + Aw) - \varphi(Aw) - Az \cdot \nabla \varphi(Aw)\\
\le & \frac{4}{\alpha - 1} g\left(\frac{|Az|}{2}\right) = \frac{4}{\alpha -
  1} g\left(\frac{|z|}{2} \right), \,
  \end{align*}
  therefore the conclusion also holds for general $z \in \mathbb{R}^n$.
\end{proof}

\bibliographystyle{siam}
\bibliography{reference}
\end{document}